\documentclass[10pt,a4paper,twoside]{article}
\ifx\pdfpageheight\undefined
   \usepackage[dvips,colorlinks=true,linkcolor=blue,citecolor=red,%
      urlcolor=green]{hyperref}
   \usepackage[dvips]{graphicx}
   \makeatletter
   \edef\Gin@extensions{\Gin@extensions,.mps}
   \makeatother
\else
   \usepackage[pdftex]{graphicx}
   \usepackage[bookmarksopen=false,pdftex=true,breaklinks=true,%
      backref=page,pagebackref=true,plainpages=false,%
      hyperindex=true,pdfstartview=FitH,%
      colorlinks=true,%
      pdfpagelabels=true,linkcolor=blue,%
      citecolor=red,urlcolor=green,hypertexnames=false,%
      colorlinks=false,%
      ]%
   {hyperref}
\fi

\usepackage[utf8]{inputenc} 
\usepackage{amsfonts,amssymb}
\usepackage{amsthm}
\usepackage[french]{babel}
\usepackage[all]{xy}
\usepackage{mathrsfs}
\usepackage{lmodern}
\usepackage{microtype}
\usepackage{array}
\usepackage{stmaryrd}
\usepackage{wasysym}
\usepackage{bm}
\usepackage{xspace}
\usepackage{etoolbox}

\makeatletter\def\th@plain{\slshape}\makeatother
\makeatletter\patchcmd{\th@remark}{\itshape}{\slshape}{}{}\makeatother

\FrenchFootnotes

\newcommand\sibrouillon[1]{}
\newcommand \hum[1]{\sibrouillon{\rdb
{\tt\small hum:  #1}
}}

\newcommand \subsec[1]{\goodbreak\addcontentsline{toc}{subsection}{#1}\subsection*{#1}}

\newcommand \subsect[2]{\goodbreak\addcontentsline{toc}{subsection}{#2}
\subsection*{#1}}

\newcommand \rdb{}
\newcommand \ix[1] {\emph{#1}}
\newcommand \iref[1] {\ref{#1}}
\newcommand \aref[1] {\cite[#1]{ACMC}}

\newcounter{MF}
\newcommand\stMF{\stepcounter{MF}}

\newcommand{\lec}{\stMF\ifodd\value{MF}lecteur \else 
lectrice \fi}
\newcommand{\lecz}{\stMF\ifodd\value{MF}lecteur\else lectrice\fi}

\newcommand{\lecs}{\stMF\ifodd\value{MF}lecteurs \else 
lectrices \fi}
\newcommand{\lecsz}{\stMF\ifodd\value{MF}lecteurs\else 
lectrices\fi}

\newcommand{\alec}{\stMF\ifodd\value{MF}au lecteur \else%
\`a la lectrice \fi}
\newcommand{\alecz}{\stMF\ifodd\value{MF}au lecteur\else%
\`a la lectrice\fi}

\newcommand{\dlec}{\stMF\ifodd\value{MF}du lecteur \else%
de la lectrice \fi}
\newcommand{\dlecz}{\stMF\ifodd\value{MF}du lecteur\else%
de la lectrice\fi}

\newcommand{\llec}{\stMF\ifodd\value{MF}le lecteur \else la lectrice \fi}
\newcommand{\llecz}{\stMF\ifodd\value{MF}le lecteur\else la lectrice\fi}

\newcommand{\Llec}{\stMF\ifodd\value{MF}Le lecteur \else La lectrice \fi}

\newcommand{\lui}{\ifodd\value{MF}lui \else
elle \fi}
\newcommand{\luiz}{\ifodd\value{MF}lui\else
elle\fi}

\newcommand{\celui}{\ifodd\value{MF}celui \else
celle \fi}

\newcommand{\ceux}{\ifodd\value{MF}ceux \else
celles \fi}

\newcommand{\er}{\ifodd\value{MF}er \else
\`ere \fi}

\newcommand{\eux}{\ifodd\value{MF}eux \else
elles \fi}

\newcommand{\eUx}{\ifodd\value{MF}eux \else
euse \fi}

\newcommand{\leux}{\ifodd\value{MF}leux \else
leuse \fi}

\newcommand{\il}{\ifodd\value{MF}il \else
elle \fi}

\newcommand{\e}{\ifodd\value{MF} \else e \fi}
\newcommand{\ez}{\ifodd\value{MF}\else e\fi}

\newcommand{\n}{\ifodd\value{MF}n \else nne \fi}
\newcommand{\nz}{\ifodd\value{MF}n\else nne\fi}

\makeatletter
\newcommand{\la}{\@ifstar{\ifodd\value{MF}le\else
la\fi}{\stMF\ifodd\value{MF}le \else la \fi}}
\makeatother

\newcommand \rem{\rdb
\noi{\it Remarque. }}

\newcommand \rems{\rdb
\noi{\it Remarques. }}

\newcommand \exl{\rdb
\noi{\bf Exemple. }}

\newcommand \exls{\rdb
\noi{\bf Exemples. }}

\newcommand \thref[1] {\thoz~\ref{#1}}
\newcommand \paref[1] {page~\pageref{#1}}

\DeclareRobustCommand{\guig}{\mbox{{\usefont{U}{lasy}%
{\if b\expandafter\@car\f@series\@nil b\else m\fi}{n}%
\char40\kern-0.20em\char40}~}}
\DeclareRobustCommand{\guid}{\mbox{~\usefont{U}{lasy}%
{\if b\expandafter\@car\f@series\@nil b\else m\fi}{n}%
\char41\kern-0.20em\char41}}
\newcommand\gui[1]{\guig{#1}\guid}

\newcommand \facile{\begin{proof}
La \dem est laiss\'ee \alecz.
\end{proof}
}

\newcommand \num {{n$^{\mathrm{ o}}$}}

\newcommand \Cadz {C'est-\`a-dire}
\newcommand \recu {r\'ecur\-rence }

\newcommand \cad {c'est-\`a-dire }

\newcommand \ssi {si, et seu\-le\-ment si, }
\newcommand \ssiz {si, et seu\-le\-ment si,~}

\newcommand \spdg {sans per\-te de g\'en\'e\-ra\-lit\'e }

\newcommand \Propeq {Les pro\-pri\-\'e\-t\'es sui\-van\-tes sont 
\'equi\-va\-len\-tes.}
\newcommand \propeq {les pro\-pri\-\'e\-t\'es sui\-van\-tes sont 
\'equi\-va\-len\-tes.}

\newcommand \Kev {$\gK$-\evc }

\newcommand \Kevz {$\gK$-\evcz}

\newcommand \Alg {$\gA$-\alg}

\newcommand \Amo {$\gA$-mo\-du\-le }
\newcommand \Amos {$\gA$-mo\-du\-les }
\newcommand \Amoz {$\gA$-mo\-du\-le}
\newcommand \Amosz {$\gA$-mo\-du\-les}

\newcommand \acz{alg\'e\-bri\-quement clos}  

\newcommand \acl {an\-neau \icl}
\newcommand \acls {an\-neaux \icl}

\newcommand \adk {an\-neau de Dedekind }

\newcommand \adp {an\-neau de Pr\"u\-fer }

\newcommand \adpz {an\-neau de Pr\"u\-fer}

\newcommand \adv {anneau de valuation }
\newcommand \advs {anneaux de valuation }

\newcommand \advz {anneau de valuation}

\newcommand \Advls {Anneaux \dvlas } 
\newcommand \advl {anneau \dvla } 
\newcommand \advls {anneaux \dvlas } 
\newcommand \advlz {anneau \dvlaz } 
\newcommand \advlsz {anneaux \dvlasz }

\newcommand \agq{alg\'e\-bri\-que }
\newcommand \agqs{alg\'e\-bri\-ques }

\newcommand \aKr {anneau de Krull }
\newcommand \aKrs {anneaux de Krull }
\newcommand \aKrz {anneau de Krull}
\newcommand \aKrsz {anneaux de Krull}

\newcommand \alg {alg\`e\-bre }

\newcommand \algo{algo\-rithme }

\newcommand \algq{al\-go\-rith\-mi\-que }
\newcommand \algqs{al\-go\-rith\-mi\-ques }

\newcommand \alo {an\-neau lo\-cal }
\newcommand \aloz {an\-neau lo\-cal}

\newcommand \ari{arith\-m\'e\-tique }

\newcommand \autos {auto\-mor\-phis\-mes }
\newcommand \autoz {auto\-mor\-phis\-me}

\newcommand \bdp {base de \dcn partielle }
\newcommand \bdpz {base de \dcn partielle}
\newcommand \bdps {bases de \dcn partielle }
\newcommand \bdpsz {bases de \dcn partielle}

\newcommand \bdf {base de \fap }

\newcommand \bif {borne inf\'e\-rieure } %
\newcommand \bifz {borne inf\'e\-rieure} %

\newcommand \bsp {borne sup\'e\-rieure } %

\newcommand \cacz{corps \acz}

\newcommand \cara{carac\-t\'e\-ris\-tique }

\newcommand \carn{carac\-t\'e\-ri\-sation }  
\newcommand \carns{carac\-t\'e\-ri\-sations }

\newcommand \cdi{corps discret }

\newcommand \cdiz{corps discret}
\newcommand \cdisz{corps discrets}

\newcommand \cli {cl\^o\-ture int\'e\-grale }

\newcommand \coe {coef\-fi\-cient }
\newcommand \coes {coef\-fi\-cients }

\newcommand \coesz {coef\-fi\-cients}

\newcommand \coh {co\-h\'e\-rent }
\newcommand \cohs {co\-h\'e\-rents }
\newcommand \cohz {co\-h\'e\-rent}
\newcommand \cohsz {co\-h\'e\-rents}

\newcommand \cohc {co\-h\'e\-rence }

\newcommand \coli {combi\-nai\-son \lin }

\newcommand \com {co\-ma\-xi\-maux }
\newcommand \comz {co\-ma\-xi\-maux}

\newcommand \cop {compl\'e\-men\-taire }
\newcommand \cops {compl\'e\-men\-taires }

\newcommand \corg{\coh r\'egulier }

\newcommand \dcn {d\'ecom\-po\-sition }
\newcommand \dcns {d\'ecom\-po\-sitions }

\newcommand \dcnb {\dcn born\'ee }
\newcommand \dcnbz {\dcn born\'ee}

\newcommand \dcnc {\dcn compl\`ete }
\newcommand \dcncz {\dcn compl\`ete}

\newcommand \dcnp {\dcn partielle }
\newcommand \dcnpz {\dcn partielle}

\newcommand \ddk {dimension de~Krull }
\newcommand \ddkz {dimension de~Krull}

\newcommand \ddp {domaine de Pr\"u\-fer }
\newcommand \ddps {domaines de Pr\"u\-fer }
\newcommand \ddpz {domaine de Pr\"u\-fer}
\newcommand \ddpsz {domaines de Pr\"u\-fer}

\newcommand \demo{d\'emon\-stra\-tion }     
     
\newcommand \demos{d\'emon\-stra\-tions }

\newcommand \dem{d\'emons\-tra\-tion }
\newcommand \demz{d\'emons\-tra\-tion}
\newcommand \dems{d\'emons\-tra\-tions }

\newcommand \deter {d\'eter\-mi\-nant }

\newcommand \dfn{d\'efi\-nition }  
\newcommand \dfns{d\'efi\-nitions }  
\newcommand \dfnz{d\'efi\-nition}

\newcommand \dip{diviseur principal }
\newcommand \dips{diviseurs principaux }

\newcommand \dipsz{diviseurs principaux}

\newcommand \dit{distri\-bu\-ti\-vit\'e }

\newcommand \dok {domaine de Dedekind }
\newcommand \doks {domaines de Dedekind }

\newcommand \doksz {domaines de Dedekind}

\newcommand \dvla {\`a diviseurs }
\newcommand \dvlaz {\`a diviseurs}
\newcommand \dvlas {\`a diviseurs }
\newcommand \dvlasz {\`a diviseurs}

\newcommand \dvlg {diviso\-riel } 
\newcommand \dvlgz {diviso\-riel} 
\newcommand \dvlgs {diviso\-riels } 
\newcommand \dvlgsz {diviso\-riels} 

\newcommand \dvli {\dvlt inver\-sible } 
\newcommand \dvlis {\dvlt inver\-sibles } 
\newcommand \dvliz {\dvlt inver\-sible} 
\newcommand \dvlisz {\dvlt inver\-sibles} 

\newcommand \dvlt {diviso\-riel\-lement } %

\newcommand \dvzz {di\-viseur de z\'e\-ro}

\newcommand \dve {divi\-si\-bi\-lit\'e }
\newcommand \dvez {divi\-si\-bi\-lit\'e}

\newcommand \dvee {\`a \dve explicite }
\newcommand \dveez {\`a \dve explicite}

\newcommand \dvr {diviseur }
\newcommand \dvrs {diviseurs }
\newcommand \dvrz {diviseur}
\newcommand \dvrsz {diviseurs}

\newcommand \dvrps {diviseurs prin\-cipaux }

\newcommand \dvrpsz {diviseurs prin\-cipaux}

\newcommand \egmt {\'ega\-le\-ment }
\newcommand \egmtz {\'ega\-le\-ment}

\newcommand \egt {\'ega\-li\-t\'e }
\newcommand \egts {\'ega\-li\-t\'es }
\newcommand \egtz {\'ega\-li\-t\'e}
\newcommand \egtsz {\'ega\-li\-t\'es}

\newcommand \elr{\'el\'e\-men\-tai\-re }

\newcommand \elt{\'el\'e\-ment }  
\newcommand \elts{\'el\'e\-ments }  
\newcommand \eltz{\'el\'e\-ment}  
\newcommand \eltsz{\'el\'e\-ments}

\newcommand\evc{es\-pa\-ce vec\-to\-riel } 
 
\newcommand\evcz{es\-pa\-ce vec\-to\-riel}

\newcommand \eqve {\'equi\-valente }  
  
\newcommand \eqves {\'equi\-valentes }

\newcommand \eqvsz {\'equi\-valents}  
\newcommand \eqvesz {\'equi\-valentes}  

\newcommand \eqvc {\'equi\-va\-lence }  
\newcommand \eqvcs {\'equi\-va\-lences }  
\newcommand \eqvcz {\'equi\-va\-lence}

\newcommand \fab {\fcn born\'ee }

\newcommand \fabz {\fcn born\'ee}

\newcommand \fat {\fcn totale }

\newcommand \fatz {\fcn totale}

\newcommand \fap {\fcn partielle }

\newcommand \fcn {factorisation }

\newcommand \fdi {for\-te\-ment dis\-cret }

\newcommand \fdiz {for\-te\-ment dis\-cret}

\newcommand\gmq{g\'eo\-m\'e\-trique }

\newcommand\gmqsz{g\'eo\-m\'e\-triques}

\newcommand\gne{g\'e\-n\'e\-ra\-li\-s\'e }

\newcommand\gnl{g\'e\-n\'e\-ral }  
\newcommand\gnle{g\'e\-n\'e\-ra\-le }

\newcommand\gnlz{g\'e\-n\'e\-ral}  
  
\newcommand\gnlsz{g\'e\-n\'e\-raux}  
\newcommand\gnlesz{g\'e\-n\'e\-ra\-les}  

\newcommand\gnlt{g\'e\-n\'e\-ra\-le\-ment }

\newcommand\gnr{g\'e\-n\'e\-ra\-li\-ser }  

\newcommand \gns{g\'en\'e\-ra\-lise }

\newcommand \grl{groupe r\'eti\-cul\'e }
\newcommand \grls{groupes r\'eti\-cul\'es }
\newcommand \grlz{groupe r\'eti\-cul\'e}
\newcommand \grlsz{groupes r\'eti\-cul\'es}

\newcommand\gtr{g\'e\-n\'e\-ra\-teur }  
\newcommand\gtrs{g\'e\-n\'e\-ra\-teurs }  
  
\newcommand\gtrsz{g\'e\-n\'e\-ra\-teurs}  

\newcommand \homo {ho\-mo\-mor\-phisme }

\newcommand \homoz {ho\-mo\-mor\-phisme}

\newcommand \icl {int\'e\-gra\-le\-ment clos }

\newcommand \iclez {int\'e\-gra\-le\-ment close}
\newcommand \iclz {int\'e\-gra\-le\-ment clos}

\newcommand \id {id\'eal }
\newcommand \ids {id\'eaux }
\newcommand \idz {id\'eal}
\newcommand \idsz {id\'eaux}

\newcommand \idema {\id maxi\-mal }
\newcommand \idemas {\ids maxi\-maux }

\newcommand \idemasz {\ids maxi\-maux}

\newcommand \idep {\id pre\-mier }
\newcommand \idepz {\id pre\-mier}
\newcommand \ideps {\ids pre\-miers }

\newcommand \idif {id\'eal \dvlg fini }
\newcommand \idifs {id\'eaux \dvlgs finis }

\newcommand \idlis {\ids \dvlis } 
\newcommand \idlisz {\ids \dvlisz } 

\newcommand \idm {idem\-po\-tent }

\newcommand \idmz {idem\-po\-tent}

\newcommand \idp {id\'e\-al prin\-ci\-pal }
\newcommand \idps {id\'e\-aux prin\-ci\-paux }

\newcommand \idpz {id\'e\-al prin\-ci\-pal}

\newcommand \idtrs {in\-d\'e\-ter\-mi\-n\'ees }

\newcommand \ifr {id\'eal frac\-tion\-nai\-re }
\newcommand \ifrs {id\'eaux frac\-tion\-nai\-res }

\newcommand \imd {imm\'e\-diat }
\newcommand \imde {imm\'e\-diate }

\newcommand \ird {irr\'e\-duc\-tible }
\newcommand \irds {irr\'e\-duc\-tibles }
\newcommand \irdz {irr\'e\-duc\-tible}
\newcommand \irdsz {irr\'e\-duc\-tibles}

\newcommand \iso {iso\-mor\-phisme }
\newcommand \isos {iso\-mor\-phismes }

\newcommand \isoz {iso\-mor\-phisme}

\newcommand \isoc {iso\-morphe }

\newcommand \itf {id\'e\-al \tf}
\newcommand \itfs {id\'e\-aux \tf}
\newcommand \itfz {id\'e\-al \tfz}

\newcommand \ivs {inversibles }
\newcommand \ivz {inversible}
\newcommand \ivsz {inversibles}

\newcommand \ivde {inverse divisorielle } 
\newcommand \ivdes {inverses divisorielles } 
\newcommand \ivdez {inverse divisorielle} 

\newcommand \ivda {inverse divisoriel } 
\newcommand \ivdas {inverses divisoriels } 
\newcommand \ivdaz {inverse divisoriel} 

\newcommand \lgb {local-global }

\newcommand \lgbz {local-global}

\newcommand \lin {lin\'e\-aire }

\newcommand \lot {loca\-le\-ment }
\newcommand \lotz {loca\-le\-ment}

\newcommand \lon {loca\-li\-sa\-tion }
\newcommand \lons {loca\-li\-sa\-tions }
\newcommand \lonz {loca\-li\-sa\-tion}

\newcommand \lop {\lot prin\-ci\-pal }

\newcommand \lopz {\lot prin\-ci\-pal}

\newcommand \mo {mo\-no\"{\i}de }
\newcommand \mos {mo\-no\"{\i}des }
\newcommand \mosz {mo\-no\"{\i}des}

\newcommand \molos {morphismes de \lon }

\newcommand \moquo {morphisme de passage au quotient }

\newcommand \ncr{n\'e\-ces\-sai\-re }       
       
\newcommand \ncrz{n\'e\-ces\-sai\-re}

\newcommand \ncrt{n\'e\-ces\-sai\-re\-ment }

\newcommand \ndz {r\'egu\-lier }

\newcommand \ndzz {r\'egu\-lier}
\newcommand \ndzsz {r\'egu\-liers}

\newcommand \ndze {r\'egu\-li\`ere }

\newcommand \noco {\noe \coh }

\newcommand \nocoz {\noe \cohz }

\newcommand \Noe {Noether }

\newcommand \noe {noeth\'e\-rien }
\newcommand \noes {noeth\'e\-riens }
\newcommand \noee {noeth\'e\-rienne }

\newcommand \noez {noeth\'e\-rien}

\newcommand \noetz {noeth\'e\-ria\-nit\'e}

\newcommand \ort{or\-tho\-go\-nal }  
\newcommand \orte{or\-tho\-go\-na\-le }  
\newcommand \orts{or\-tho\-go\-naux }

\newcommand \ortsz{or\-tho\-go\-naux}

\newcommand \pb{pro\-bl\`e\-me }

\newcommand \pf {de \pn finie }

\newcommand \Plg {Prin\-cipe \lgb }
\newcommand \plg {prin\-cipe \lgb }

\newcommand \plgz {prin\-cipe \lgbz}

\newcommand \Plgc {\Plg con\-cret }
\newcommand \plgc {\plg con\-cret }

\newcommand \pn {pr\'esen\-ta\-tion }

\newcommand \pol {poly\-n\^ome }
\newcommand \pols {poly\-n\^omes }
\newcommand \polz {poly\-n\^ome}
\newcommand \polsz {poly\-n\^omes}

\newcommand \polles{poly\-no\-miales }  
\newcommand \pollesz{poly\-no\-miales}

\newcommand \polcar {\pol ca\-rac\-t\'e\-ris\-ti\-que }

\newcommand \polmin {\pol mini\-mal }

\newcommand \prmtz {pr\'eci\-s\'e\-ment}

\newcommand \prn {pro\-jec\-tion }

\newcommand \prof {profondeur }
\newcommand \profz {profondeur}

\newcommand \Prts {Pro\-pri\-\'et\'es }
\newcommand \prt {pro\-pri\-\'et\'e }
\newcommand \prts {pro\-pri\-\'et\'es }
\newcommand \prtz {pro\-pri\-\'et\'e}
\newcommand \prtsz {pro\-pri\-\'et\'es}

\newcommand \rde {rela\-tion de d\'epen\-dance }

\newcommand \rdi {\rde int\'e\-grale }

\newcommand \sdo {\sdr \orte }

\newcommand \sdr {somme directe }

\newcommand \sdzz {sans \dvzz}

\newcommand \sgr {\sys \gtr }

\newcommand \smq {sym\'e\-trique }

\newcommand \srg {suite r\'egu\-li\`ere }

\newcommand \sulsz {suppl\'e\-men\-taires}

\newcommand \sys {sys\-t\`eme }

\newcommand \tf {de ty\-pe fi\-ni }
\newcommand \tfz {de ty\-pe fi\-ni} 

\newcommand \Tho {Th\'eo\-r\`eme }
\newcommand \tho {th\'eo\-r\`eme }
\newcommand \thos {th\'eo\-r\`emes }
\newcommand \thoz {th\'eo\-r\`eme}

\newcommand \trdi {treil\-lis dis\-tri\-bu\-tif }

\newcommand \trdiz {treil\-lis dis\-tri\-bu\-tif}

\newcommand \unt {uni\-taire }

\newcommand \untz {uni\-taire}

\newcommand \zed {z\'{e}\-ro-di\-men\-sion\-nel }

\newcommand \cof {cons\-truc\-tif }
\newcommand \cofs {cons\-truc\-tifs }
\newcommand \cofz {cons\-truc\-tif}

\newcommand \cov {cons\-truc\-ti\-ve }
\newcommand \covz {cons\-truc\-ti\-ve}
\newcommand \covsz {cons\-truc\-ti\-ves}
\newcommand \covs {cons\-truc\-ti\-ves }

\newcommand \coma {\maths\covs}
\newcommand \comaz {\maths\covsz}
\newcommand \clama {\maths clas\-si\-ques }
\newcommand \clamaz {\maths clas\-si\-ques}

\renewcommand \cot {cons\-truc\-ti\-vement }
\newcommand \cotz {cons\-truc\-ti\-vement}

\newcommand \maths {math\'e\-ma\-tiques }

\newtheorem{theorem}{Théorème}[section]
\newtheorem{thdef}[theorem]{Théorème et définition}

\newtheorem{lemma}[theorem]{Lemme}
\newtheorem{corollary}[theorem]{Corolaire}
\newtheorem{proposition}[theorem]{Proposition}
\newtheorem{propnot}[theorem]{Proposition et notation}
\newtheorem{propdef}[theorem]{Proposition et définition}
\newtheorem{fact}[theorem]{Fait}

\newtheorem{plcc}[theorem]{Principe local-global concret}

\theoremstyle{definition}
\newtheorem{definition}[theorem]{Définition}

\newtheorem{definota}[theorem]{Définition et notation}

\theoremstyle{remark}

\newcommand {\junk}[1]{}

\newcommand \Grandcadre[1]{%
\begin{center}
\begin{tabular}{|c|}
\hline
~\\[-3mm]
#1\\[-3mm]
~\\ 
\hline
\end{tabular}
\end{center}

}

\newcommand\ndsp{\textstyle}

\newcommand \noi {\noindent}
\renewcommand \ss {\smallskip}
\newcommand \sni {\ss\noi}
\newcommand \snii {\noi}
\newcommand \ms {\medskip}
\newcommand \mni {\ms\noi}

\newcommand\eti{^\times}
\newcommand \epr{^\perp}
\newcommand \etl{^*}

\newcommand \etoz{$^*$}

\newcommand \iBA {_{\gB/\!\gA}}

\newcommand \divi {\mid}

\newcommand\equidef{\buildrel{{\rm def}}\over{\quad\Longleftrightarrow\quad}} 
\newcommand\eqdefi{\buildrel{\rm def}\over {\;=\;}}

\newcommand \fraC[2] {{{#1}\over {#2}}}

\newcommand \abs[1] {\left|{#1}\right|}
\newcommand \abS[1] {\big|{#1}\big|}
\newcommand \aqo[2] {#1\sur{\gen{#2}}\!}

\newcommand \gen[1] {\left\langle{#1}\right\rangle} 
\newcommand \geN[1] {\big\langle{#1}\big\rangle}
\newcommand \so[1] {\left\{{#1}\right\}} 
\newcommand \sotq[2] {\so{\,#1\,|\,#2\,}} 
\newcommand \sur[1] {\!\left/#1\right.}
 
\newcommand \und[1] {\underline{#1}}

\newcommand \idg[1] {|\,#1\,|}

\newcommand \dex[1] {[\,#1\,]}

\newcommand\simarrow{\vers{_\sim}}
 
\newcommand\vers[1]{\buildrel{#1}\over \longrightarrow }

\newcommand \lora {\longrightarrow}

\renewcommand \leq{\leqslant}
\renewcommand \preceq{\preccurlyeq}
\renewcommand \geq{\geqslant}

\newcommand \som {\sum\nolimits}

\newcommand\lrb[1] {\llbracket #1 \rrbracket}
\newcommand\lrbn {\lrb{1..n}}

\newcommand\lrbm {\lrb{1..m}}
\newcommand\lrbk {\lrb{1..k}}
\newcommand\lrbp {\lrb{1..p}}
\newcommand\lrbq {\lrb{1..q}}
\newcommand\lrbr {\lrb{1..r}}

\newcommand \snic[1] {\sni\centerline{$#1$}

\ss}

\newcommand \snac[1]{\sni
{\small\centering$#1$\par}

\ss}

\newcommand \eoe {\hbox{}\nobreak\hfill
\vrule width .5em height .5em depth 0mm \par \smallskip}

\newcommand \ov[1] {\overline{#1}}

\newcommand \wi[1] {\widetilde{#1} }

\makeatletter 
\def\revddots{\mathinner{\mkern1mu\raise\p@ 
\vbox{\kern7\p@\hbox{.}}\mkern2mu 
\raise4\p@\hbox{.}\mkern2mu\raise7\p@\hbox{.}\mkern1mu}} 
\makeatother

\newcommand \FF{\mathbb {F}}

\newcommand \NN{\mathbb {N}} 
\newcommand \ZZ{\mathbb {Z}} 
\newcommand \OO{\mathbb{O}}
 
\newcommand \QQ{\mathbb {Q}} 
\newcommand \RR{\mathbb {R}}

\newcommand \gk {\mathbf{k}}
\newcommand \gA {\mathbf{A}}
\newcommand \gB {\mathbf{B}}
\newcommand \gC {\mathbf{C}}
\newcommand \gF {\mathbf{F}}
\newcommand \gK {\mathbf{K}}
\newcommand \gL {\mathbf{L}}

\newcommand \gV {\mathbf{V}}

\newcommand\Ati {\gA^{\!\times}}
\newcommand \Atl {\gA^{\!*}}
\newcommand \Btl {\gB^{*}}
\newcommand \Ktl {\gK^{*}}
\newcommand{\KAt}{\gK\etl\!\sur{\Ati}}
\newcommand{\AAt}{\Atl\!\sur{\Ati}}

\newdimen\xyrowsp
\newcommand{\SCO}[6]{
\xymatrix @R = \xyrowsp {
                                  &1 \ar@{-}[dl] \ar@{-}[dr] \\
#3 \ar@{-}[ddr]                   &   & #6 \ar@{-}[ddl] \\
                                  &\bullet\ar@{-}[d] \\
                                  &\bullet   \\
#2 \ar@{-}[ddr] \ar@{-}[uur]      &   & #5 \ar@{-}[ddl] \ar@{-}[uul] \\
                                  &\bullet \ar@{-}[d] \\
                                  &\bullet  \\
#1 \ar@{-}[uur]                   &   & #4 \ar@{-}[uul] \\
                                  & 0 \ar@{-}[ul] \ar@{-}[ur] \\
}
}

\newcommand\MA[1]{\mathop{#1}\nolimits}

\newcommand \Zar {\MA{\mathsf{Zar}}}

\newcommand \cC {{\cal C}}

\newcommand \rc {\mathrm{c}}

\newcommand \rD {\mathrm{D}}

\newcommand \rN {\mathrm{N}}

\newcommand\fa{\mathfrak{a}}
\newcommand\fb{\mathfrak{b}}
\newcommand\fc{\mathfrak{c}}

\newcommand\fp{\mathfrak{p}}

\newcommand\fq{\mathfrak{q}}

\newcommand\fx{\mathfrak{x}}
\newcommand\fy{\mathfrak{y}}

\newcommand \vu {\vee} 
\newcommand \vi {\wedge} 
\newcommand \Vu {\bigvee}
\newcommand \Vi {\bigwedge}

\newcommand \ua  {{\underline{a}}}
\newcommand \ub  {{\underline{b}}}

\newcommand \uc{{\underline{c}}}
\newcommand \ud  {{\underline{d}}}

\newcommand \ux {{\underline{x}}}

\newcommand \uX {\underline{X}}
\newcommand \uy{{\underline{y}}}

\newcommand \uz{{\underline{z}}}

\newcommand \ak {a_1,\ldots,a_k}
\newcommand \an {a_1,\ldots,a_n}
\newcommand \aq {a_1,\ldots,a_q}
\newcommand \aln {\alpha_1,\ldots,\alpha_n}
\newcommand \bn {b_1,\ldots,b_n}

\newcommand \bbm {b_1,\ldots,b_m}
\newcommand \ck {c_1,\ldots,c_k}
\newcommand \cq {c_1,\ldots,c_q}

\newcommand \sn {s_1,\ldots,s_n} 
\newcommand \xn {x_1,\ldots,x_n}

\newcommand \Xn {X_1,\ldots,X_n}
\newcommand \xr {x_1,\ldots,x_r}

\newcommand \xin {\xi_1,\ldots,\xi_n}

\newcommand \ym {y_1,\ldots,y_m}

\newcommand \AT {\gA[T]}
\newcommand \AX {\gA[X]}

\newcommand \AuX {\gA[\uX]}

\newcommand \AXn {\gA[\Xn]}

\newcommand \BuX {\gB[\uX]}

\newcommand \BX {{\gB[X]}}
\newcommand \BT {{\gB[T]}}

\newcommand \BXn {{\gB[\Xn]}}

\newcommand \CT {{\gC[T]}}

\newcommand \KX {\gK[X]}

\newcommand \KuX {\gK[\uX]}
\newcommand \kuX {\gk[\uX]}

\newcommand \KXn {\gK[\Xn]}

\newcommand \kXn {\gk[\Xn]}

\newcommand \Div {\MA{\mathrm{Div}}}
\newcommand \DivA {\Div\gA }
\newcommand \DivAp {(\Div\gA)^{+} }
\newcommand \DivB {\Div\gB }
\newcommand \DivBp {(\Div\gB)^{+} }

\newcommand \dv {\MA{\mathrm{div}} }
\newcommand \dvA {\dv_\gA }
\newcommand \dvB {\dv_\gB }

\newcommand \Frac {\MA{\mathrm{Frac}}}

\newcommand \Gfr {\MA{\mathrm{Gfr}}}
\newcommand \Gr {\MA{\mathrm{Gr}}}

\newcommand \Ker {\MA{\mathrm{Ker}}}
\newcommand \Lst {\MA{\mathrm{Lst}}}

\newcommand \Id {\MA{\mathrm{Id}}}
\newcommand \Idif {\MA{\mathrm{Idif}}}
\newcommand \Idv {\MA{\mathrm{Idv}}}

\newcommand \Icl {\MA{\mathrm{Icl}}}

\newcommand \mod {\;\mathrm{mod}\;}

\newcommand \Rad {\MA{\mathrm{Rad}}}
\newcommand \Reg {\MA{\mathrm{Reg}}}

\newcommand\hsu{\\ \hspace*{4mm}}

\renewcommand \snii{\sni}

\DeclareMathAlphabet{\mathpzc}{OT1}{pzc}{m}{it}

\begin{document}

\title{ 
Anneaux à diviseurs et anneaux de Krull\\
 (une approche \covz)}

\author{T. Coquand, H. Lombardi }
\date{\today}

\vspace{-2em}

\maketitle

\vspace{-1em}
\centerline{paru dans Communications in Algebra, {\bf 44}: 515--567, 2016}

\smallskip\noindent  Quelques typos dans la bibliographie ont été corrigés par rapport à la version publiée et la version 1 sur ArXiv. Dans la version 3 on a mieux organisé les \demos de 1.19 et 1.20. Dans cette version une amélioration substantielle est la démonstration simplifiée du théorème 1.29. Une table des matières se trouve page 44.

\smallskip \noindent \hum{biblio en plus  pour Lorenzen? ajout de commentaires sur Lorenzen dans l'intro?}



\vspace{1em}

\begin{abstract} 
Nous présentons dans cet article une approche constructive, dans le style de Bishop, de la théorie des diviseurs et des anneaux de Krull. 
 Nous accordons une place centrale aux ``anneaux à diviseurs'', appelés PvMD dans la littérature anglaise. Les résultats classiques sont obtenus comme résultats d'algorithmes explicites sans faire appel aux hypoth\`eses de factorisation compl\`ete. 

\bigskip \centerline{\bf Abstract}

\smallskip  
 We give give an elementary and constructive version of the 
theory of ``Prüfer v-Multiplication Domains" (which we call
``anneaux à diviseurs" in the paper) and Krull Domains. 
The main results of these theories are revisited from a constructive point of
view, following the Bishop style, and without assuming properties of complete factorizations.

\end{abstract}

\subsection*{Introduction}
\addcontentsline{toc}{section}{Introduction}

Nous présentons dans cet article une approche \cov de la théorie des \dvrsz.

L'excellent livre \textsl{Divisor Theory} \cite[Edwards, 1990]{Edw}  traite \cot la \cli d'un anneau factoriel dans une extension finie de son corps des fractions. Edwards suppose aussi que l'anneau factoriel poss\`ede de bonnes \prts pour la \fcn des \polsz.

Nous traitons ici de mani\`ere \cov un cas plus \gnl dans lequel nous n'avons aucune hypoth\`ese de \dcn en facteurs \irdsz. L'exemple le plus important en pratique reste celui des anneaux \gmqsz\footnote{Alg\`ebres \pf sur un \cdiz}  \iclz. Et dans ce cas la \dcn d'un diviseur en somme d'\irds n'est pas assurée \cotz.   

Nous nous situons dans la suite du livre \textsl{Commutative algebra : constructive methods} \cite[Lombardi et Quitté, 2015]{ACMC}, où est développée une théorie \cov des \doks (chapitre XII) indépendante de la possibilité d'une \dcn  des \ids en produit d'\idemasz.
Comme cet article est écrit dans le style des \coma à la Bishop, nous donnerons la version \cov précise que nous choisissons pour beaucoup de notions classiques, même tr\`es bien connues.

\medskip 
Le \gui{tour de force} qui est réalisé par l'adjonction de pgcds idéaux
en théorie des nombres peut-il être \gne de mani\`ere significative?

Oui, pour certains anneaux \iclz, que nous appelons les \textsl{\advlsz}, dénommés \gui{anneaux pseudo-prüfériens} dans les exercices de Bourbaki,
et \gui{Prüfer $v$-multiplication domains} dans la littérature anglaise. Ils constituent une classe d'anneaux suffisamment large  pour lesquels on a une notion raisonnable de \dvrsz, mais où l'on ne réclame pas l'existence  de \dvrs \irdsz.  

La théorie correspondante en \clama semble due principalement à Lorenzen, Jaffard, Lucius et Aubert (\cite{Aubertadvl,Jafadvl,Lorenzenadvl,Luciusadvl,Luciusadvl1}).

Un développement moderne récent de cette théorie se trouve dans \cite[Halter-Koch, 2011]{Halter-Koch}.

Cette théorie généralise la théorie plus classique des anneaux dits de Krull,
dans laquelle on réclame une \dcn unique en facteurs premiers pour les diviseurs. Elle a été élaborée notamment par Krull (1936), Arnold (1929), van der Waerden (1929),  
Prüfer (1932) et Clifford (1938) (\cite{Arnold29,Clifford38,Krulladvl,Prufadvl,vdW29}). Il semble aussi que l'approche de \cite[Borevitch et Chafarevitch (1967)]{Bosha} ait eu une forte influence sur
les exposés ultérieurs de la théorie.  

En pratique les \advls sont au départ des anneaux à pgcd int\`egres ou des \ddpsz. Ensuite la classe des \advls est stable par extensions \polles ou enti\`eres et intégralement closes, ce qui donne les exemples usuels de la littérature.

\medskip Voici un bref résumé de l'article.

\smallskip Dans la section \iref{secTheoDiviseurs}
nous donnons une version \elr et \cov des bases de la théorie des \advlsz.
Nous donnons des \carns simples des \advls (\thos \ref{thADVL}, \ref{corplccAnnDivl} et \ref{cor2plccAnnDivl}).
 Nous démontrons qu'un anneau int\`egre \coh est \dvla \ssi il est \icl (\thref{thiclcohidv}). 
 Nous démontrons que les \advls de \ddk $\leq 1$ sont les \ddps de \ddk $\leq 1$ (\thref{thi2clcohidv}).
 Nous donnons en \ref{plccAnnDivl} un 
\plgc pour la \dvez, les \acls et les \advlsz.

\smallskip Dans la section \iref{secpropdecgrl}  nous abordons les questions de \dcn des diviseurs. Nous expliquons notamment les \grls de dimension $1$ et leurs 
\prts de base.

\smallskip Dans la section  \iref{secAnnDivl} nous donnons les \prts de stabilité essentielles pour les \advlsz: \lonz, anneaux de \polsz, \cli dans une extension \agq du corps de fractions.
 
\smallskip Dans la section \iref{secAKrull} nous donnons un traitement \elr et \cof des anneaux de Krull (les anneaux dont les \dvrs forment un \grl à \dcnbz). Nous indiquons l'\algo de \dcn partielle pour les familles finies de \dvrsz, nous démontrons un \tho d'approximation simultanée, le théor\`eme \gui{un et demi} pour les \aKrsz, et le \tho caractérisant les \aKrs qui ne poss\`edent qu'un nombre fini de \dvrs \irdsz. 

\smallskip La question des \dvrs \irds est traitée pour l'essentiel dans les énoncés
\ref{lemdivirdadvlgnl}, \ref{corlemdivirdadvlgnl}, \ref{corcorlemdivirdadvlgnl} et  \ref{cor0lemdvlloc}, puis précisée en  \ref{lemirreddim1}, \ref{lemGrldcc}, \ref{lemdvlloc2}, \ref{lemdvlloc2Krull} et \ref{lemdvllocDcnc}.

\medskip   Nous terminons cette introduction par quelques explications concernant notre terminologie. 
Dans la littérature anglaise nos \gui{\advlsz} sont \gnlt appelés 
des \gui{Prüfer $v$-multiplication domain}, abrégé en PvMD. Ce n'est vraiment pas tr\`es élégant. Ils ont été introduits en 1967 par 
M.~Griffin dans \cite{Griff} initialement sous le nom de \textsl{$v$-multiplication rings}. Lucius les appelle des \textsl{anneaux avec une théorie des pgcds de type fini} dans \cite{Luciusadvl}.
Dans les exercices de Bourbaki, \textsl{Alg\`ebre Commutative, Diviseurs}, leur nom est \textsl{anneau pseudo-prüferien}.
Pour plus de renseignements sur la littérature existante voir \cite{Chang08,FL06,FZ09,FZ11}. 

Nous aurions voulu appeler ces anneaux   \textsl{anneaux divisoriels}.
Mais dans la littérature anglaise on trouve \gui{divisorial ring} pour un anneau
int\`egre dont tous les \ids sont \gui{\dvlgsz} au sens de Bourbaki, \cad intersections d'\ifrs principaux.
Cette \dfn n'est pas pertinente d'un point de vue \cofz, car il est 
impossible de montrer \cot que l'anneau $\ZZ$ la satisfait\footnote{Une
mésaventure analogue arrive avec la \dfn usuelle de la \noetz: tout \id est \tfz.}. 
Néanmoins, nous avons considéré que le terme \gui{anneau divisoriel},
que nous convoitions, était déjà pris, et nous nous sommes rabattus sur le moins élégant \gui{\advlz}.

 Signalons aussi que nous introduisons la notion purement multiplicative de liste \dvli \paref{defiidvli} dans la section \ref{secTheoDiviseurs}. La notion associée d'\id \dvli  dans un anneau int\`egre co\"{\i}ncide avec celle d'\id $t$-\ivz, définie dans la littérature sur les {PvMD}. Nous pensons cependant que la terminologie \dvli est plus parlante.

Enfin nous définissons \paref{defKRprof2} l'anneau de Nagata divisoriel $\gB_{\dv}(\uX)$ d'un anneau commutatif arbitraire $\gB$. Dans le cas d'un anneau int\`egre,
il est appelé dans la littérature \textsl{l'anneau de Nagata pour la star opération $v$}, et il est noté $\mathrm{Na}(\gB,v)(\uX)$.

\section[Anneaux \dvlasz]{Anneaux \dvlasz}
\label{secTheoDiviseurs}

\Grandcadre{Dans tout l'article, $\gA$ est un anneau int\`egre, de corps de fractions $\gK$, et l'on note \\
 $\Atl=\Reg\gA$ (le \mo des \elts \ndzsz) et $\Ati$ le groupe des unités.}

Nous définissons un anneau int\`egre comme un anneau qui satisfait l'axiome \gui{tout \elt est nul ou \ndzz}. 
Nous n'excluons donc pas à priori le cas de l'anneau trivial. Dans ce cas
$\Atl=\gA$ et l'anneau total des fractions~$\gK$ est \egmt trivial.
Pour qualifier un \elt de $\Atl$, nous parlerons donc plut\^ot d'\elt \ndz que d'\elt non nul.

Notez que l'anneau $\gK$ vérifie
l'axiome  \gui{tout \elt est nul ou \ivz}, que $\gA$ soit trivial ou non.

\medskip 
Dans la suite, nous parlerons de $\AAt$ comme du \textsl{\mo de \dve de $\gA$}, et de $\KAt$ comme du \textsl{groupe de \dve de~$\gA$}.
Le premier est vu avec sa structure ordonnée de \mo positif\footnote{On définit un \textsl{\mo positif} comme un \mo commutatif simplifiable~\hbox{$(M,+,0)$} pour lequel  est satisfaite 
l'implication~$x+y=0\; \Longrightarrow\;  x=y=0$. Le \mo $M$ peut alors être vu comme la partie positive d'un groupe ordonné $G$ (en tant que groupe, $G$ est le symétrisé de $M$.)  
} et le second,
isomorphe au symétrisé du premier,
 est vu comme un groupe ordonné (ses \eltsz~\hbox{$\geq 0$}
sont les classes des \elts de $\Atl$). 

Nous désignons par $\Gr_\gA(\an)$ la \prof de $\fa=\gen{\an}$,  (abréviation pour la \prof du \Amo $\gA$
relativement à l'\id $\fa$).
Pour cet article nous suffira de rappeler les \dfns suivantes: 
une liste $(\ua)=(\an)$ est dite de \prof $\geq 1$ si 
les \egts $a_1x=\dots=a_nx=0$ impliquent $x=0$. La liste $(\ua)$ est dite de \profz~\hbox{$\geq 2$} si en outre, toute liste $(\xn)$
proportionnelle (i.e. $a_ix_j=a_jx_i$ pour tous $i$, $j$) est multiple de~$(\ua)$ (i.e. il existe $x$ tel que $x_j=xa_j$ pour tout $j$).
Ces \prts sont attachées à l'\idz~$\fa$: on peut changer de \sgr pour l'\idz. 

Rappelons aussi les résultats classiques suivants (nous les utilisons pour $k=1$ ou $2$): si $\Gr(\fa)\geq k$ et $\Gr(\fb)\geq k$ alors $\Gr(\fa\fb)\geq k$; si $\Gr(\fa)\geq k$ et $\fa\subseteq \fb$ alors $\Gr(\fb)\geq k$;
si $\Gr(\an)\geq k$ alors pour tout $m$, $\Gr(a_1^{m},\dots,a_n^{m})\geq k$.

\subsec{Pgcd fort, ppcm et profondeur $\geq 2$}

On sait que dans un anneau int\`egre, deux \elts qui admettent un pgcd n'admettent pas \ncrt un ppcm, bien que la réciproque soit valable.
Voici des précisions sur ce sujet.

\begin{propdef} \label{defiStrongPGCD} \textsl{(Pgcd fort, ppcm et profondeur $\geq 2$)}\\
Soient  $\gA$ un anneau int\`egre, et $a_1$, \dots, $a_n$, $b$, $g\in\Atl$.  
\begin{enumerate}
\item \label{i1defiStrongPGCD} On dit que la famille  $(a_i)$   admet l'\elt $g$ comme \ix{pgcd fort} si sont satisfaites les conditions \eqves suivantes.
\begin{enumerate}
\item \label{i1adefiStrongPGCD} L'\elt $g$ vu dans $\KAt$ est le pgcd (la \bifz) de la liste~$(\an)$.
\item \label{i1bdefiStrongPGCD} Pour tout $x\in\Atl$ l'\elt $xg$ est un pgcd dans $\gA$ \hbox{de  $(xa_1,\dots,xa_n)$}.
\item \label{i1cdefiStrongPGCD}  \'Etant donné un multiple commun $a$ des $a_i$ dans $\Atl$, l'\elt $a/g$ est un ppcm \hbox{des $a/a_i$} (autrement dit $\fraC a g \,\gA=\bigcap_{i=1}^n \fraC a{a_i} \,\gA$)
\item \label{i1ddefiStrongPGCD}  (Formulation avec des \ifrs dans $\gK$) L'\elt $1/g$ est un ppcm \hbox{des $1/a_i$} dans $\gK$, autrement dit $ \,\fraC1 g \,\gA=\bigcap_{i=1}^n\fraC1 {a_i} \,\gA$.
\end{enumerate}
\item \label{i2defiStrongPGCD} Si $g$ est le pgcd fort de $(\an)$, c'est aussi le pgcd fort de n'importe quel autre \sgr de l'\id $\fa=\gen{\an}$. On dira donc que $g$ est le \textsl{pgcd fort de l'\itfz~$\fa$}.
\item \label{i3defiStrongPGCD} Si $g$ est le pgcd fort de $(\an)$, $bg$ est le pgcd fort de $(ba_1, \dots, ba_n)$.
\item \label{i4defiStrongPGCD} On a  $\Gr(\an)\geq 2$ \ssi $1$ est  pgcd fort de $(\an)$.
\end{enumerate}  
\end{propdef}
\facile

Naturellement, un pgcd fort est un pgcd.

\medskip 
\rem  \label{notaIvifrtf} Soit $\gA$ un anneau int\`egre de corps de fractions $\gK$.
Pour deux sous-\Amos non nuls  $\fa$ et $\fb\subseteq \gK$, on note 
${(\fa:\fb)_\gK=\sotq{x\in\gK}{x\fb\subseteq \fa}.}$
Dans la terminologie des star-opérations, on note $\fa^{-1}= (\gA:\fa)_\gK$.
Alors $\fa\mapsto(\fa^{-1})^{-1}$ est une star-opération, \gnlt appelée $v$-opération, et $\fa^{v}=(\fa^{-1})^{-1}$ est l'intersection des \ifrs principaux (non nuls) qui contiennent~$\fa$.
\\
Avec ces notations, les \prts \eqves du point \textsl{1} ci-dessus
sont aussi \eqves à~$\gen{\an}^{-1}=\geN{\fraC1g}$ ou à
$\big(\gen{\an}^{-1}\big)^{-1}=\gen{g}$.  Lorsque l'anneau est \coh et $\fa$, $\fb$ \tf comme \Amosz, $(\fa:\fb)_\gK$
et $(\fa^{-1})^{-1}$ sont \tfz.
\eoe

\subsec{Idéaux \dvlisz}

\begin{definition} \label{defiidvli}
Une liste $(\ua)=(\an)$ dans $\Atl$ est dite \textsl{\dvliz}  si l'on a une liste $(\ub)=(\bbm)$ dans~$\Atl$ telle
que la famille~\hbox{$(a_ib_j)_{i\in \lrbn,j\in \lrbm}$} admet un pgcd fort $g$ dans~$\Atl$.\\
Les deux listes $(\ua)$ et $(\ub)$ sont dites \textsl{\ivdes l'une de l'autre}.%
\index{divisoriellement inversible!liste --- (dans un anneau int\`egre)}\end{definition}

\exls \\
1) Dans un anneau à pgcd, les pgcds sont forts et toute famille finie admet la liste $(1)$ comme 
\ivdez.
\\
2) Si $\gen{\an}\gen{\bbm}=\gen{g}$ avec $g\in\Atl$, la liste $(\ua)$ admet la liste $(\ub)$
comme \ivde et $g$ est le pgcd fort des $a_ib_j$.
\eoe

\medskip 
En notant $\fa=\gen{\ua}$ et $\fb=\gen{\ub}$,
cette \prt des listes $(\ua)$ et $(\ub)$ ne dépend que de l'\idz~$\fa\fb$. Et donc la \prt pour la liste $(\ua)$ d'être \dvli ne dépend que de l'\itf $\fa=\gen{\ua}$.
Ceci introduit une nouvelle \dfnz.

\begin{definition} \label{defiidvli2} Soit $\gA$ un anneau int\`egre.
\begin{itemize}
\item Un \itf $\fa$ de $\gA$ est dit \textsl{\dvliz}
s'il existe un \itf $\fb$ tel que $\fa\fb$ admette un pgcd fort.
On dit aussi  que l'\itf $\fa$ et l'\itf $\fb$ sont \textsl{\ivdas l'un de l'autre}.
\item  L'anneau  $\gA$ est appelé un \textsl{\advlz} si tout \itf non nul
 est \dvliz.%
\index{divisoriellement inversible!ideal@\itf --- (dans un anneau int\`egre)}%
\index{anneau!a divis@à diviseurs}
\end{itemize}
\end{definition}

Notez que l'\ivda d'un \itfz, s'il existe, n'est en aucun cas unique. 
Il s'agit un peu du même flottement que lorsque l'on parle de l'inverse d'un \itf dans un \adpz. Mais ici, ce sera encore plus flottant,
car deux \ivdas d'un même \itf sont rarement  isomorphes comme~\Amosz.   
Par exemple dans un \advl 
si $\Gr(\fc)\geq 2$  et si $\fa\fb$ admet un pgcd fort $g$, 
alors $\fa(\fb\fc)$ admet le même pgcd fort. Mais si $\gA$ n'est pas un \adpz, en \gnlz, $\fb$ et $\fb\fc$
ne sont pas des \Amos isomorphes.   

\medskip  
\rem La notion de liste \dvli est une notion purement multiplicative
qui peut être définie pour un anneau commutatif arbitraire en utilisant
la \carn dans le point \textsl{\ref{i4defiStrongPGCD}} 
de la proposition~\ref{defiStrongPGCD}. La notion associée d'\id \dvli  dans un anneau int\`egre co\"{\i}ncide avec celle d'\id $t$-\ivz, définie dans la littérature sur les~{PvMD}. Nous pensons cependant que la terminologie \dvli est plus parlante.
\eoe

\begin{lemma} \label{lemdvli}
Soit  $(\an)$ une liste  \dvli dans~$\Atl$ et 
$x$ \ndz dans l'\id $\gen{\an}$. 
\begin{enumerate}
\item On peut trouver une liste $(\cq)=(\uc)$ tel que $x$ soit le pgcd fort 
des $a_ic_k$.
\item  Si l'anneau $\gA$ est \cohz, on peut prendre
pour $(\cq)$ un \sgr de l'\id transporteur~$(\gen{x}:\fa)_\gA$.
\end{enumerate}
\end{lemma}
%
\begin{proof}
\textsl{1.} Considérons une liste $({\bbm})$ pour laquelle les $a_ib_j$ admettent un pgcd fort $g$. Dans  $\KAt$, $g$ est le pgcd de la famille 
$({a_ib_j})_{i,j}$, donc~$x=\fraC x g \,g$ est le pgcd de la famille  $\big(\fraC{xa_ib_j}g\big)_{i,j}$. 
Comme les $\fraC{a_ib_j}g$  sont dans $\gA$ et $x$ est \coli des $a_i$, on a~$\fraC{xb_j}g \in \gA$ et on peut prendre
pour $(\uc)$ la liste des $\fraC{xb_j}g$.

\snii \textsl{2.} Lorsque $\gA$ est \cohz, le transporteur $\fc'=(\gen{x}:\fa)_\gA$
est un \itf qui contient l'\id $\fc=\geN{\fraC{xb_j}g,\, j\in\lrbm}$ construit
au point~\textsl{1}. On doit vérifier \hbox{que $\fa\fc'$} admet aussi $x$ pour pgcd fort. Et  puisque $\fa\fc\subseteq \fa\fc'$ il suffit \hbox{que $x$} divise tous les \gtrs de $\fa\fc'$,
ce qui est clair. 
\end{proof}

\subsect{Un projet pour le groupe des \dvrs
d'un anneau int\`egre}{Projet pour le groupe des \dvrsz}

Comme nous l'avons déjà souligné, du point de vue \algq la \fat 
 n'est pas une \prt  \gui{facile}, et nous sommes surtout 
intéressés
par le fait d'avoir un bon \grl construit de mani\`ere raisonnable à 
partir du \mo positif $\AAt$ (ou du groupe 
ordonné~$\KAt$). 

Une mani\`ere
d'expliquer ce que l'on veut réaliser est 
de décrire formellement les \prts du \textsl{groupe des \dvrs de $\gA$},
que nous noterons  $\DivA$, et de l'application $\dvA$ de $\Atl $ \hbox{dans $(\DivA)^+$}, qui à tout \elt de $\Atl$ associe le \textsl{\dip (positif ou nul)} qu'il définit.
Nous utilisons une notation additive pour le groupe $\DivA$. Nous nous inspirons ici de la présentation de la théorie des \dvrs donnée dans le livre \cite{Bosha} consacré à la théorie des nombres.

Mais comme le sugg\`ere \cite[Aubert]{Aubertadvl}, nous ne mentionnerons ici que la structure multiplicative de~$\Atl$: l'addition dans $\gA$ ne doit pas intervenir ici!
Voir cependant la proposition~\iref{lemADVL}.

\rdb
\goodbreak\mni
{\bf Projet \dvlg 1.} \label{projet1}
{\it Les \prts requises pour \fbox{$\dvA:\Atl\to(\DivA)^+$} sont les suivantes.
\begin{enumerate}
\item [$\rD_1$]\label{i1Projet1}
$\DivA$ est un \grl et $\dvA$ un morphisme de \mosz:

\snic{\dvA(1)=0,\quad \dvA(ab)=\dvA(a)+\dvA(b).}
\item [$\rD_2$]\label{i2Projet1} Pour tous $a$, $b\in\Atl$,
${\dvA(a)\leq\dvA(b)\;\Longleftrightarrow\; a\divi b.}$     
\item [$\rD_3$]\label{i3Projet1}
Tout \elt  de $(\DivA)^+$ est la \bif dans $\DivA$ d'une famille finie d'\elts de  $\dvA(\Atl)$. 
\end{enumerate}
}

\medskip  On peut  écrire ce projet sous la forme \eqve suivante,
en demandant de traiter tous les \dipsz, y compris ceux qui proviennent 
de~$\Ktl$.

\goodbreak\mni
{\bf Projet \dvlg 2.} \label{projet2}
{\it Les \prts requises pour  \fbox{$\dvA:\Ktl\to\DivA$} sont les suivantes.
\begin{enumerate}
\item [$\rD'_1$] \label{i1Projet2}
$\DivA$ est un \grl et $\dvA$ passe au quotient~$\KAt$  en donnant un morphisme de groupes ordonnés.
\[ 
\begin{array}{ccc} 
\forall x,\,y\in\Ktl: \dvA(xy)=\dvA(x)+\dvA(y),    \\[1mm] 
\dvA(1)=0,\quad \forall a\in\Atl: \dvA(a)\geq 0. 
\end{array}
\]
\item [$\rD'_2$] \label{i2Projet2} Pour tout $x\in\Ktl$,
${\dvA(x)\geq 0 \iff x\in\gA.}$     
%
%
%
\item  [$\rD'_3$] \label{i3Projet2}
Tout \elt  de $\DivA$ est la \bif dans $\DivA$ d'une famille finie d'\elts de  $\dvA(\Ktl)$. 
\end{enumerate}
}

\medskip Dans la suite nous utiliserons parfois \gui{$\dv$} au lieu de \gui{$\dvA$} lorsque le contexte sera clair.

\medskip \rem
La condition $\rD_3$ implique la \prt qu'un \elt de $(\DivA)^+$ est égal à la \bif dans $\DivA$ de tous les \elts de $\dv(\Atl)$ qui le majorent. 
Nous utiliserons souvent cette \prt par la suite.

\noindent Mais naturellement la condition $\rD_3$ est à priori plus forte. \\
Dans \cite{Bosha} c'est une \prt encore plus faible qui est énoncée: deux \elts de $(\DivA)^+$ qui ont les mêmes majorants
dans~$\dv(\Atl)$ sont égaux.  
Par contre ces auteurs introduisent deux conditions \sulsz. 
\\
D'une part ils demandent que $\dvA(a+b)\geq \dvA(a)\vi\dvA(b)$ (\prt non multiplicative). 
\\
D'autre part ils demandent à $\DivA$ d'être un \grl à \dcn compl\`ete. 
\eoe

\mni 
{\bf Convention.}
 Il peut être pratique de  rajouter\footnote{Lorsque $\gA\neq 0$ les conditions requises sont satisfaites avec $+\infty >\DivA$. Par contre si $\gA=0$, on ne rajoute rien du tout, car $\dv(0)=\dv(1)$.} un \elt $+\infty$ à~$\DivA$ en posant~\hbox{$\dv(0)=+\infty$},
$+\infty\geq \delta$
 et $\delta+\infty=\infty+\delta=+\infty\;$  pour \hbox{tout $\delta\in\DivA\cup\so{+\infty}.$}
\\ Alors les \prts décrites dans le Projet \dvlg restent valables
 en acceptant $0$ là où on le refusait.
Cette convention présente deux intérêts. Premi\`erement, $\DivAp\cup\so{+\infty}$ est un \trdiz\footnote{Lorsque $\gA\neq 0$, il manque juste l'\elt maximum dans $\DivAp$ pour en faire un \trdiz.}. Deuxi\`emement, cela permet
de traiter de mani\`ere plus uniforme les espaces spectraux implicitement présents dans la théorie.
\eoe

\subsec{Le \tho de base}
\begin{thdef} \label{thADVL} \label{defADVL} \emph{(Anneaux \dvlasz)}\\
 Pour que le projet \dvlg puisse être réalisé pour l'anneau~$\gA$ il faut et suffit que $\gA$ soit un \advl (toute famille finie non vide de $\Atl$ est \dvliz).
\\ Dans ce cas le couple $(\dvA,\DivA)$ est unique à \iso unique pr\`es. 
\begin{itemize}
\item On dit alors que $\gA$ est un \emph{\advl avec $\dvA:\Ktl\to\DivA$ pour théorie des \dvrsz}.\index{theorie des di@théorie des diviseurs}

\item  On note alors $\dvA(\an)$ pour $\dvA(a_1)\vi\cdots\vi\dvA(a_n)$.

\item  Un \elt $\alpha\in\DivA$ est appelé un \ix{diviseur principal}
s'il est de la forme~$\dv(x)$ pour \hbox{un $x\in\Ktl$}.

\end{itemize}

\end{thdef}

\begin{proof}  
Nous laissons \alec le soin de vérifier que les deux projets \dvlgs sont \eqvsz. 

\noindent Supposons le projet \dvlg réalisé et montrons que toute liste $(\an)$ dans $\Atl$ est \dvliz.
On note $\dvA(\an)$ pour $\dvA(a_1)\vi\cdots\vi\dvA(a_n)$.  On a  \ncrt par \dit

\snic{\dvA(\an)+\dvA(\cq)   =   \dvA\big((a_ic_j)_{i\in\lrbn,j\in\lrbq}\big)  .}

\noindent 
et trivialement $\dvA(\an)\vi\dvA(\cq)  =   \dvA(\an,\cq)$.\\
On a $\dvA(a_1)\geq \dvA(\an)$ de sorte que l'on peut écrire (en vertu de $\rD_3$)

\snic{\dvA(a_1)-\dvA(\an)=\dvA(\cq)}

\noindent  pour une famille finie $(\cq)=(\uc)$ dans $\Atl$. L'\egt

\snic{\dvA(a_1)=\dvA(\ua)+\dvA(\uc)=\dvA\big((a_ic_j)_{i\in\lrbn,j\in\lrbq}\big)}

\noindent  nous dit que $\dvA(a_1)$ est la \bif de la famille  $\big(\dvA(a_ic_j)\big)_{i,j}$ dans $\DivA$.
Et vu~$\rD'_2$, cela implique que  $a_1$ est la \bif de la famille  $(a_ic_j)_{i,j}$ dans $\KAt$. Ce qui signifie que~$a_1$ est pgcd fort des $a_ic_j$ dans $\Atl$. Ainsi la condition que toute famille finie soit \dvli est bien satisfaite.

\snii Voyons maintenant la question de l'unicité.\\
Les \elts de $\DivA$ s'écrivent tous \ncrt sous forme $\dvA(\ua)-\dvA(\ub)$ pour deux listes finies $(\ua)=(\an)$ et $(\ub)=(\bbm)$ dans~$\Atl$. On peut même demander \hbox{que $\dvA(\ua,\ub)=0$}, autrement dit que $1$ soit pgcd fort des $a_i$ et $b_j$ dans $\Atl$.\\
Pour que l'unicité de la solution du projet divisoriel soit acquise, il suffit de voir qu'on n'a pas le choix pour décider une \egt

\snic{\dvA(\ua)-\dvA(\ub)=\dvA(\uc)-\dvA(\ud)}

\noindent pour des listes finies $(\ua)$, $(\ub)$, $(\uc)$, $(\ud)$ de $\Atl$.
Cela revient à $\dvA(\ua)+\dvA(\ud)=\dvA(\ub)+\dvA(\uc)$. 
\\
De mani\`ere plus \gnle on donne une \prt \cara pour une \egt 

\snic{\dvA(\an)=\dvA(\bbm).}

\noindent 
Ici, on peut avancer l'un des deux arguments suivants, au choix.
\\
\textsl{Premier argument.} 
\\
Si $(\cq)$ est une famille \ivde de $(\an)$ avec $g$ pgcd fort des $(a_ic_j)$, l'\egt $\dvA(\ua)=\dvA(\ub)$ équivaut au fait que
$g$ est pgcd fort des~$(b_kc_j)$.\\
\textsl{Deuxi\`eme argument.}  
\\
On donne une \prt \cara pour une in\egt $\dvA(\ua)\leq \dvA(b)$. 
C'est la \prt suivante. 
\hsu--- \textsl{Dans $\Ktl,\,b$ est multiple de tous les \dvrs communs à $(\an). \quad (*)$ \label{projet1D4}}

\noindent En effet,  
puisqu'on doit avoir $\dvA(1/x)=-\dvA(x)$ en raison de $\rD'_1$,
la condition $\rD'_3$ est \eqve à sa formulation duale: tout \elt  de $\DivA$ est la \bsp dans $\DivA$ d'une famille finie d'\elts de  $\dvA(\Ktl)$, et à fortiori il est la \bsp de l'ensemble des $\dvA(x)$ qu'il majore.
\\
En particulier 

\snic{\dvA(b)\geq \alpha=\dvA(\ua) \iff \dvA(b)\hbox{  majore } M_\alpha=\sotq{\dvA(x)}{x\in \Ktl,\,\dvA(x)\leq \alpha}\,.}

\noindent 
Et $M_\alpha$ est égal à $\sotq{\dvA(x)}{x\in \Ktl,\,\&_i\; x \mid a_i}$.
Et vu~$\rD_2$, $\dvA(b)\geq \alpha$ signifie exactement que
 dans~$\Ktl$, $b$ est multiple des \dvrs communs à la liste $(\ua).$ 

\snii Supposons maintenant la condition d'existence des \ivdas satisfaite et montrons que l'on peut construire $\DivA$ et $\dvA$ conformément au projet \dvlgz.
\\
Pour cela on 
munit l'ensemble   $\Lst(\gA)\etl$  (ensemble des listes finies non vides d'\elts de~$\Atl$) de la relation de préordre $\preceq$ directement inspirée de la condition $(*)$ précédente:

\snic{(\an)\preceq (\bbm) \equidef \hbox{ chaque } b_j \hbox{ est multiple des \dvrs communs à } (\ua).}

\noindent 
On  définit $\DivAp$ comme l'ensemble ordonné correspondant (où $\alpha=\beta\Leftrightarrow \alpha\preceq \beta \hbox{ et } \beta\preceq\alpha$). 
\\
On vérifie alors que l'on peut définir une addition sur $\DivAp$ en posant 

\snic{(\ua)+(\ub)\eqdefi(a_ib_j)_{i,j}.}

\noindent  En d'autres termes cette loi à priori mal définie \gui{passe au quotient}. On vérifie ensuite que l'on obtient par symétrisation un \grl $\DivA$ convenable. Les détails sont laissés \alecz.   
\end{proof}

Le résultat dans le lemme qui suit est donné page 388 dans \cite[Zafrullah, 2006]{ZafTrib}:
il faut lire \gui{PvMD} au lieu de \gui{\advlz} et  \gui{$t$-\ivz} au lieu de  \gui{\dvliz}. 
\begin{lemma} \label{lem2suffisent}
Un anneau int\`egre dans lequel tout \id fid\`ele à \und{deux} \gtrs est \dvli est 
un \advlz.
\end{lemma}
\begin{proof}
On utilise l'astuce de Dedekind. Pour 3 \ids
arbitraires~$\fa$,~$\fb$,~$\fc$ dans un anneau on a toujours l'\egt

\snic{(\fa+\fb)(\fb+\fc)(\fc+\fa)=(\fa+\fb+\fc)(\fa\fb+\fb\fc+\fa\fc).}

\noindent  En outre un produit d'\idlis est \dvliz. Donc si $\fa+\fb$, $\fb+\fc$ et $\fa+\fc$ sont \dvlisz, il en va de même pour $\fa+\fb+\fc$. Cela permet de passer des \ids à deux \gtr aux \ids à trois \gtrsz, puis de proche en proche à un nombre quelconque de \gtrsz. 
\end{proof}
\exls On a deux exemples fondamentaux, à partir desquels seront construits la plupart des autres exemples intéresssants en pratique.
\\ 1) Un anneau int\`egre à pgcd $\gA$ est \dvla et l'on a 

\snic{(\DivA,+,0,\leq )\simeq (\Ktl/\Ati, \,\cdot\,, \ov 1, \,\mid\, ).}

\noindent  En outre un \ivda de n'importe quel \itf est $\gen{1}$.\\ 
Un \advl dont tous les \dvrs sont principaux est un anneau à pgcd.

\snii 2) Un \ddp est un \advl et l'on a

\snic{(\DivA,+,0,\leq,\vi)\simeq (\Gfr(\gA), \,\cdot\, , \gen{1}, 
\supseteq,+   ).}

\noindent 
(rappelons que $\Gfr(\gB)$ désigne en \gnl le groupe des  \ifrs \ivs de~$\gB$: pour un \ddp
ce sont tous les \ifrs \tf fid\`eles).\\
En outre un \itf fid\`ele\footnote{Si $\gA$ est un anneau int\`egre non trivial, un \itf est fid\`ele \ssi il est non nul.} admet un inverse (au sens 
des \ids \ivsz) qui est aussi un \ivdaz.

\snii\rdb 3)  On donne souvent comme premier exemple d'un \adp qui n'est pas un anneau de Bézout l'anneau $\ZZ[\alpha]$ où $\alpha=\sqrt{-5}$ dans lequel on a $2\times 3=(1+\alpha )(1-\alpha )$ avec les $4$ \elts \irdsz. Le myst\`ere de la \dcn unique en facteurs premiers est alors éclairci par les 
\dcns en produits d'\ideps données par les \egts 
\[ 
\begin{array}{rcl} 
\gen{2}=\gen{2,1+\alpha } ^{2}  &,   &\gen{3}=\gen{3,1+\alpha }\gen{3,1-\alpha }\;,   \\[.3em] 
\gen{1+\alpha }=\gen{2,1+\alpha }\gen{3,1+\alpha }  & \hbox{et}  & 
\gen{1-\alpha }=\gen{2,1+\alpha }\gen{3,1-\alpha }.   
\end{array}
\]
avec les trois \ids $\gen{2,1+\alpha }$, $\gen{3,1+\alpha }$, $\gen{3,1-\alpha }$ \irds distincts.

\snii \label{advlexemple3}Voici un exemple du même style avec un \advl qui n'est ni un anneau à pgcd ni un \ddpz. On consid\`ere un \cdi $\gk$ et l'on définit 

\snic{\gA=\aqo{\gk[A,B,C,D]}{AD-BC}=\gk[a,b,c,d].}

\noindent 
Il s'agit d'un anneau int\`egre  dans lequel
$a$, $b$, $c$, $d$ sont des \elts \irdsz.
En effet l'anneau quotient $\gA$ reste gradué et $a$, $b$, $c$, $d$ sont de degré~$1$. 
La suite $(a,d)$ est \ndze et les localisés en $a$ et $d$ sont des anneaux à pgcd. Par exemple $\gA[\fraC1a]\simeq\gk[A,B,C,1/A]$. Donc~$\gA$ est un \advl (\plg \ref{plccAnnDivl}). En outre $\dvA(a,d)=0$ car l'\egt a lieu dans les deux localisés (\plg \ref{plccAnnDivl}).
\\
On vérifie que $\gen{b,a}\gen{b,d}=\gen{b}\gen{a,b,c,d}$, donc $\dvA(b,a)+\dvA(b,d)=\dvA(b)$ {car $\dv(a,b,c,d)=0$}. \\
Donc  l'\egt $ad=bc$ sans \dcn unique apparente en facteurs premiers s'explique par les \dcns en sommes de \dvrs deux à deux \orts dans $\DivA$ données par les \egts
\[\!\!\! 
\begin{array}{ccc} 
\dvA(a)=\dvA(a,b)+\dvA(a,c)  &,   & \dvA(d)=\dvA(d,b)+\dvA(d,c)\;,  \\[.3em] 
\dvA(b)=\dvA(a,b)+\dvA(d,b)  &\hbox{et}   &  \dvA(c)=\dvA(a,c)+\dvA(d,c).  
\end{array}
\] 
Notons qu'il est un peu plus délicat de certifier que les quatre diviseurs $\dvA(a,b)$, $\dvA(a,c)$, $\dvA(d,b)$ et $\dvA(d,c)$  sont \irdsz.
Voir à ce sujet la poursuite de cet exemple pages \pageref{advl1exemple3} et \pageref{advl3exemple3}.
\\
Enfin l'\elt $a$ est \ird mais il n'est pas premier car $\aqo\gA a\simeq\aqo{\gk[B,C,D]}{BC}$. Donc $\gA$ n'est pas un anneau à pgcd\footnote{Autre argument: le couple $(a,b)$ a pour pgcd $1$, mais ce n'est pas un pgcd fort
car~$ad$  est multiple de $a$ et $b$ sans être multiple de $ab$
(sinon $d$ serait multiple de $b$).}.  
\\
Et l'anneau $\gA$ n'est pas non plus un \ddp car l'\egt \hbox{$\dvA(a,d)=0$} impliquerait $\gen{a,d}=\gen{1}$, ce qui n'est pas le cas\footnote{Un argument plus savant: $\gA$ étant \nocoz, il serait de dimension $\leq 1$ (\thoz~\hbox{XII-7.8} dans~\cite{ACMC}), or il contient la \srg $(a,d,b+c)$.}. 
\eoe
%

\begin{proposition} \label{cordvla1}
Lorsque  $\gA$ est un \advl avec  $\dvA:\Ktl\to \DivA$ comme théorie des \dvrsz, on a les \prts suivantes.
\begin{enumerate}
\item Pour $b$, $a_1$, \dots, $a_n$ dans $\Atl$, 
\begin{enumerate}
\item $\dvA(b)\leq\dvA(\an)$ \ssi
$b$   divise les $a_i$ dans~$\Atl$,
\item $\dvA(b)\geq\dvA(\an)$ \ssi
$b$  est multiple de tous les \dvrs communs  à $(\an)$ dans~$\Ktl$,
\item $\dvA(b)=\dvA(\an)$
\ssi $b$ est pgcd fort de~$(\an)$.
\end{enumerate}
\item Soit $\alpha$ un \elt arbitraire de $\DivA$.
\begin{enumerate}
\item $\alpha$ est la \bif d'une famille finie de \dips (en conséquence il est la \bif des \dips qu'il minore).
\item $\alpha$ est la \bsp d'une famille finie de \dips (en conséquence il est la \bsp des \dips qu'il majore).
\item On peut écrire $\alpha$ sous la forme $\dvA(\ua)-\dvA(\ub)$ pour deux listes finies $(\ua)$ et $(\ub)$ dans~$\Atl$ avec  $\dvA(\ua,\ub)=0$ (autrement dit  $1$ est pgcd fort des $a_i$ et $b_j$ dans~$\Atl$).
\end{enumerate}
\item On consid\`ere deux listes $(\ux)=(\xn)$ et $(\uy)=(\ym)$ dans~$\Ktl$ 
\begin{enumerate}
\item On a l'\egtz~$\;\Vi_{i}\dvA(x_i)=\Vi_{j}\dvA(y_j)\;$  \ssi $(\ux)$ et $(\uy)$ ont les mêmes \dvrs communs dans~$\Ktl$.
\item On a l'\egtz~$\;\Vu_{i}\dvA(x_i)=\Vu_{j}\dvA(y_j)\;$  \ssi $(\ux)$ et $(\uy)$ ont les mêmes multiples communs dans~$\Ktl$.
\end{enumerate}
\end{enumerate}
\end{proposition}
\begin{proof}
Tout ceci résulte des considérations développées dans la \dem du \thref{thADVL}.
\end{proof}

On a alors comme corolaire une \prt faisant intervenir (enfin) la structure additive de l'anneau. 
\begin{propdef} \label{lemADVL}
On suppose que $\gA$ est un \advlz. 
\begin{enumerate}
\item Pour tous $a$, $b\in\Atl$
on a $\dvA(a)\vi\dvA(b)\leq \dvA(a+b)$.
\item Pour $a_1$, \dots, $a_n\in\gA$ et $a\in\gen{\an}$,
on a 

\snic{\dvA(a_1)\vi\dots\vi\dvA(a_n)\leq \dvA(a).}

\noindent En notant $\fa=\gen{\an}$
on voit que le \dvr $\alpha=\Vi_i\dvA(a_i)$ ne dépend que de~$\fa$.
On note donc $\alpha=\dvA(\fa)=\dvA(\an)$.\\
On dit que \emph{$\alpha$ est le \dvr de l'\itf $\fa$}. 
\item Pour deux \itfs $\fa$ et $\fb$ de $\gA$ on a alors 

\snic{\fb\supseteq \fa\Rightarrow \dvA(\fb)\leq \dvA(\fa) \quad \hbox{ et }\quad \dvA(\fa\fb)=\dvA(\fa)+\dvA(\fb).}

\end{enumerate}
\end{propdef}
\begin{proof}
\textsl{1.} En effet les \dvrs communs 
\hbox{à  $(a,b)$} dans $\Ktl\!\sur{\Ati}$ 
sont les mêmes que les \dvrs communs à $(a,b,a+b)$.\\
Vu la proposition \iref{cordvla1} on a donc $\dv(a)\vi \dv(b)=\dv(a)\vi \dv(b)\vi \dv(a+b)$.

\snii\textsl{2.} Même chose.

\snii\textsl{3.} Par \dit dans le \grl $\DivA$.
\end{proof}

\medskip 
On peut \gnr la proposition \ref{lemADVL}
aux \ifrs \tf de $\gA$.
  On a alors les implications suivantes
($x\in\gK$, $\fx$, $\fy$  \ifrs \tf  $\subseteq \gK$):
\[ 
\begin{array}{rclcrcl}
x\in\fx  & \Longrightarrow  & \dvA(\fx)\leq \dvA(x),  &\qquad & 
\fx\supseteq \fy  & \Longrightarrow  &  \dvA(\fx)\leq \dvA(\fy). 
\end{array}
\]
Notez que pour $\fx=\gen{\xn}$ l'in\egt $\dvA(\fx)\leq \dvA(x)$ signifie que

\snic{\dvA(\xn)=\dvA(\xn,x)   \hbox{ i.e. }   \dvA(\fx)=\dvA(\fx+\gen{x}).}

\noindent 
Cela signifie aussi que tout $y\in\gK$ qui divise les  $x_i$ divise $x$(\footnote{Pour $x$ et les $x_i$ dans $\gA$ c'est une condition du type \gui{pgcd fort}: pour tout $z\in\gA$ et $c\in\Atl$, si $z$ divise les~$cx_i$, alors~$z$ divise~$cx$.}).

\subsec{Exprimer un \dvr comme \bsp de \dvrpsz}

\begin{propnot} \label{lemsupfinidivpr} Soit $\gA$ un \advlz. 
\begin{enumerate}
\item Un \dvr $\geq 0$ s'écrit $\alpha=\dv(\fa)$ pour un \itf $\fa$. 
Pour exprimer $\alpha$ comme \bsp de \dvrps on consid\`ere un \ivda $\fb=\gen{\bbm}$ de l'\id $\fa$. On a donc $\alpha+\beta=\dv(g)$ ($\beta=\dv(\fb)$) pour un $g\in\Atl$. D'où finalement 

\snic{\alpha=\Vu_{j=1}^{m}\dv(\fraC{g}{b_j})=\Vu_{j=1}^{m}\dv(a_j)$ pour des $a_j\in\Ktl.}
\item De la même mani\`ere, un \elt arbitraire de $\DivA$ peut s'écrire sous forme $\alpha-\dv(b)$ pour un $\alpha\in\DivAp$ et un $b\in\Atl$. Il s'écrit donc explicitement comme une \bsp finie de \dvrpsz. 
\item Si $\alpha\in\DivA$ s'écrit $\Vu_{j=1}^{m}\dv(a_j)$, on a pour $x\in\gK$: 

\snic{\dvA(x)\geq \alpha  \iff  x\in \bigcap_ia_i\gA.}

\noindent On note \hbox{$\Idv_\gA(\alpha)$} ou $\Idv(\alpha)$ cet \ifr 

\snic{\Idv(\alpha)=\sotq{x\in\gK}{\dvA(x)\geq \alpha}=\bigcap_ia_i\gA.}

\noindent 
Enfin pour un \ifr $\fc=\sum_{i=1}^{q} c_i\gA$, on note de mani\`ere abrégée
$\Idv(\fc)$ ou $\Idv(\cq)$ au lieu de $\Idv\!\big(\!\dv(\fc)\big)$.  
\item Pour $\alpha$ et $\gamma\in\DivA$ on a $\alpha\leq \gamma$ \ssi $\Idv(\alpha)\supseteq \Idv(\gamma)$ dans $\gK$.  \\ 
En particulier $\alpha= \gamma$ \ssi $\Idv(\alpha)= \Idv(\gamma)$.  
\item Dans le cas de figure du point 1, on a $\Idv(\alpha)=\gA \cap \,\bigcap_j\fraC{g}{b_j}\gA=(g:\fb)_\gA$.
\item Si  $\Idv(\alpha)=\bigcap_ia_i\gA=\sum_{h=1}^{m}c_h\gA$, alors 
${\alpha=\dvA(c_1,\dots,c_m).}$ \\
Ainsi lorsque $\gA$ est \cohz, l'\ifr $\Idv(\alpha)$ est \tf pour tout
\dvrz~$\alpha$.
\end{enumerate}
\end{propnot}
\begin{proof} \textsl{1} et \textsl{2.} Clair. 

\snii
\textsl{3.}  De mani\`ere \gnle un \dvr est \bif des \dips qui le majorent. Or dire que $\dvA(x)$ majore les
$\dvA(a_j)$, \hbox{autrement} dit qu'il majore leur \bsp $\alpha$, c'est dire \hbox{que $x\in \bigcap_ia_i\gA$}.

\snii
\textsl{4} et \textsl{5.} Clair d'apr\`es \textsl{3.}

\snii \textsl{6.} Posons  $\gamma=\dvA(c_1,\dots,c_m)$.
Un \dip $\dvA(x)$ majore  $\alpha$ \ssi $x\in \bigcap_ia_i\gA$. 
Ainsi $\gamma\geq \alpha$, et par ailleurs tout \dip qui majore~$\alpha$ majore~$\gamma$.  \\
Donc $\gamma=\alpha$ car ils sont majorés par les mêmes \dipsz. 
\end{proof}
\exl \label{advl1exemple3} Avec l'exemple 3) \paref{advlexemple3}, puisque $\dvA(a,b)+\dvA(a,c)=\dvA(a)$, le point~\textsl{5} ci-dessus nous donne
 l'\egt $\Idv_\gA(a,b)=(a:c)_\gA$. 
 Un calcul montre alors \hbox{que $(a:c)_\gA=\gen{a,b}$}. Et par suite $\gen{a,b}=\Idv_\gA(a,b)$.
\eoe

\medskip En \gnl  pour un \ifr $\fc=\gen{\cq}$ on peut avoir une inclusion stricte $\fc\subsetneq \Idv(\fc)$ (c'est le cas si $\fc\subsetneq \Icl(\fc)$). Lorsque l'anneau est \cohz, les intersections finies d'\ifrs principaux sont des \ifrs \tf qui ont donc un statut particulier.

\medskip  \rems 1) Pour un \ifr  non nul \tf $\fa$ de l'\advl $\gA$ on a $\Idv(\fa)=(\fa^{-1})^{-1}$, qui est l'intersection des \ifrs principaux  contenant $\fa$.

\noindent 2) Tout \dvr étant \bsp d'une famille finie de \dipsz, on pourrait choisir de représenter les \dvrs par les intersections finies d'\ifrs principaux comme \gui{forme canonique}.
L'avantage est alors que $\alpha=\beta$ \ssiz\hbox{$\Idv(\alpha)=\Idv(\beta)$}.
Notons que cependant dans l'exemple précédent: $\dvA(a,b)$ est $>0$ mais il ne peut s'écrire comme \bsp d'\elts $\dvA(x_i)$ pour des $x_i\in\gA$
car tout diviseur commun de~$a$ et~$b$ dans $\gA$ est une unité.
\eoe
 
\subsec{Idéaux divisoriels finis}

La proposition \ref{lemsupfinidivpr} justifie la \dfn suivante, et donne le corolaire qui la suit.

\begin{definota} \label{def.idif}~\\
Soit $\gA$ un anneau et $\gK=\Frac\gA$ son anneau total de fractions.
\begin{itemize}
\item On appelle \textsl{\idif de $\gA$} une intersection finie d'\ifrs principaux fid\`eles dans $\gK$ (i.e., de la forme $x\gA$ pour un $x\in\Ktl$). Si $\gA$ est un \advlz, un \idif est n'importe quel  \ifr de la forme $\Idv(\alpha)$
pour un $\alpha\in\DivA$. 
\item On note $\Idif(\gA)$ l'ensemble des \idifs de~$\gA$.  
\item Pour un \ifr $\fa$ qui n'est pas supposé \tfz, on définit $\Idv(\fa)$ comme l'intersection des \ifrs principaux fid\`eles qui contiennent $\fa$. 
En outre, dans le cas d'un \advl et pour un \(\alpha\in \DivA\), on écrit $\dv(\fa)=\alpha$ si $\Idv(\fa)=\Idv(\alpha)$. Ceci revient à dire que $\fa\subseteq \Idv(\alpha)$ et que 
$\alpha=\dv(\fb)$ pour un \ifr \tf $\fb\subseteq \Idv(\fa)$.
\end{itemize}

\end{definota}
\noindent NB. Un \idif n'est pas \ncrt un \ifr \tfz, mais c'est le cas lorsque $\gA$ est \cohz. Par ailleurs,  $\dv(\fa)=0$ signifie que~$\fa$ contient une suite de \profz~\hbox{$\geq 2$}.

\begin{corollary} \label{corlemsupfinidivpr}
Pour un \advlz, les applications 
\[ 
\begin{array}{rcl} 
\Idif(\gA)\lora \DivA  &,   &  \bigcap_ix_i\gA\,\longmapsto\,\Vu_i\dvA(x_i)\;\;\hbox{et} \\[.3em] 
\DivA\lora \Idif(\gA)  &,   & \alpha\,\longmapsto\,\Idv(\alpha)  \end{array}
\]
sont des bijections réciproques bien définies. 
\end{corollary}

\subsec{Quand le groupe des \dvrs est-il discret?}

Un \grl $G$ est dit \textsl{discret} si la relation d'ordre $\alpha\leq \beta$
est décidable\footnote{En \comaz, la notion de construction (et donc celle de décidabilité) est une notion primitive qui n'est pas susceptible d'une \dfn en termes de machines de Turing. Voir \cite[Coquand, 2014]{CoqRec}.}. Pour cela il faut et suffit que l'\egt  $\gamma=0$ pour un $\gamma\in G^{+}$ soit décidable, car $\alpha\leq \beta\iff \alpha-(\alpha\vi\beta)=0$. Pour le \grl $\DivA$, cela veut dire que l'on sait tester si une liste finie $(\an)$ dans $\Atl$ admet $1$ comme pgcd fort.

En fait savoir tester si $\dvA(x)\leq \dvA(y)$ pour $x$ et $y\in\Ktl$ revient à savoir tester la \dve dans~$\Atl$, et cela  va suffire. 
En effet, pour un \advl $\gA$, on peut tester \gui{$\alpha\leq \beta$?} pour~$\alpha$ et~$\beta$ dans~$\DivA$ comme suit.
On écrit $\alpha=\Vu_{i=1}^{n}\dvA(x_i)$ \hbox{et  $\beta=\Vi_{j=1}^{m}\dvA(y_j)$} pour des $x_i$ et $y_j\in\Ktl$.
 
On doit donc tester $\dvA(x_i)\leq \dvA(y_j)$ pour chaque  couple $(i,j)$. 
D'où le résultat.

\begin{lemma} \label{thDivDiscret} Pour un \advl $\gA$ \propeq
\begin{enumerate}
\item Le groupe $\DivA$ est discret.
\item La \dve dans $\Atl$ est décidable (autrement dit $\gA$ est une partie détachable de $\gK$). On dit aussi dans ce cas que $\gA$ est un \emph{anneau \dveez}. 
\end{enumerate}
C'est le cas lorsque $\gA$ est \fdiz. 
\end{lemma}

%
\subsec{Diviseurs \irdsz}

Rappelons qu'un \elt $\pi>0$ d'un \grl $G$ est  dit \textsl{\irdz} si toute \egt $\pi=\eta+\zeta$
dans~$G^{+}$ implique $\eta=0$ ou $\zeta=0$.
Par ailleurs le \gui{lemme de Gauss} dit que

\snic{(\xi\perp \eta\hbox{ et }\xi\leq \eta+\zeta)  \;\Longrightarrow\;  \xi\leq \zeta.}

\noindent On en déduit le \gui{lemme d'Euclide}, qui  pour un \elt \ird $\pi$ et deux \elts $\eta$ et $\zeta\in G^{+}$, donne, si $G$ est discret, l'implication,

\snic{\pi\leq \eta+\zeta  \;\Longrightarrow\;  \pi\leq \eta \hbox{ ou }\pi\leq \zeta}

\noindent En langage de \dve on dirait \gui{tout \elt \ird est premier}.

\smallskip Comme dans \cite{ACMC}, nous appelons \gui{\idepz} tout \id qui donne par passage au quotient un anneau \sdzz, \cad vérifiant $xy=0\;\Rightarrow\; (x=0 \hbox{ ou } y=0)$.
En particulier l'\id $\gen{1}$ est premier.

\smallskip On obtient pour les \dvrs \irds les deux \thos importants suivants.

\begin{theorem} \label{lemdivirdadvlgnl}
Dans un \advl  \dveez, un \dvr $\alpha>0$ est \ird \ssi $\Idv(\alpha)$
est un \idepz.\\ 
On obtient donc
une bijection entre  les ensembles suivants.
\begin{itemize}
\item Les \dvrs \irdsz. 
\item Les \idifs   premiers $\neq \gen{1}$. 
\end{itemize} 
\end{theorem}
\begin{proof} Supposons $\alpha$ \ird et montrons que $\Idv(\alpha)$ est un
\idepz. Si $xy\in \Idv(\alpha)$, on a $\dv(x)+\dv(y)\geq \alpha$, donc 
 $\dv(x)\geq \alpha$ ou $\dv(y)\geq \alpha$ (car $\alpha$ est \gui{premier}),
\cad $x\in \Idv(\alpha)$ ou $y\in \Idv(\alpha)$.

\snii Supposons $\Idv(\alpha)$ premier et $\alpha\leq \beta+\gamma$ 
avec 

\snic{\beta=\dv(\bn)\geq 0 \hbox{ et }\gamma=\dv(\cq)\geq 0.}

\noindent On montre que
$\alpha\leq \beta$ ou $\alpha\leq \gamma$. Ainsi, puisque~\hbox{$\alpha>0$}, il est \gui{premier} et à fortiori \irdz.
Commme chaque $b_ic_j$ est dans $\Idv(\alpha)$ (car $\dv(b_ic_j)\geq \alpha$),
on obtient $b_i\in\Idv(\alpha)$ \hbox{ou $c_j\in\Idv(\alpha)$}.
Or par \dit du \gui{ou} sur le \gui{et} dans le calcul des propositions, on a, en notant $P_i$ pour  \gui{$b_i\in\Idv(\alpha)$} et $Q_j$
pour \gui{$c_j\in\Idv(\alpha)$}

\snic{(\&_i P_i) \hbox{ ou } (\&_j Q_j) = \&_{i,j} (P_i \hbox{ ou } Q_j).}

\snii Et si par exemple  $\&_i\big(b_i\in\Idv(\alpha)\big)$, cela donne $\alpha\leq \beta$.
\end{proof}

\exls 
 1)  Soit $\gA$  un \adv int\`egre \dveez. C'est un domaine de Bézout pour lequel le groupe $\DivA$ est  totalement ordonné discret et il y a au plus un \dvr \irdz~$\pi$, qui s'écrit $\pi=\dv(p)$ avec $\Rad(\gA)=\gen{p}=\Idv(\pi)$. Ainsi on a un \dvr \ird \ssi $\Rad(\gA)$ est un \idpz.
S'il y a d'autres \ideps non nuls, \cad si le rang  du groupe $\DivA$ \hbox{est $>1$},
ils ne sont pas de la forme $\Idv(\alpha)$. Donc  ce ne sont pas des \idps
et ils ne sont pas \tfz.   

\snii 2)  Soit  $\gA$  un anneau à pgcd.
Les \dvrs \irds correspondent aux \elts \irds de l'anneau (à association pr\`es), ou encore aux \ideps principaux non nuls. Si $\gA$ est factoriel et si $x\in\fp$ premier ($\neq \gen{0}, \gen{1}$), un des \elts \irds qui divisent $x$, disons~$p$, doit être dans~$\fp$. Donc si $\fp\neq \gen{p}$, il y a au moins deux \elts \irds (non associés) dans $\fp$, ce qui fait une suite de profondeur $2$, et $\dv(\fp)=0$. 
\eoe

\begin{theorem} \label{corlemdivirdadvlgnl}
Dans un \advl  \dvee non trivial, si $\fp$ est un \idep \tf non nul avec $\pi=\dv(\fp)>0$, alors $\fp=\Idv(\fp)$ et $\pi$ est un \dvr \irdz.
\end{theorem}
\begin{proof}
D'apr\`es le \thref{lemdivirdadvlgnl}, il suffit de montrer que $\fp=\Idv(\pi)$, l'inclusion $\fp\subseteq \Idv(\pi)$ étant triviale.\\
Soit $a\neq 0$ un \elt de $\fp$. D'apr\`es le lemme \ref{lemdvli}, il existe
deux listes $(\ub)=(b_1, \dots, b_m)$ 
 \hbox{et  $(\uc)=(\cq)$} dans $\Atl$, la deuxi\`eme de profondeur~\hbox{$\geq 2$}, telles que 
$${\fp\,\gen{b_1, \dots, b_m}}= \gen{a}\gen{\cq}
\eqno(*)
$$
Comme $\pi>0=\dv(\cq)$ il y a un $c_j$ tel que $\dv(c_j)\not\geq \pi$, et à fortiori $c_j\notin\fp$.
\\
Soit $x\in\Idv(\fp)=(\gen{a}:\gen{\ub})$ (point \textsl{5} de la proposition \ref{lemsupfinidivpr}). On \hbox{a $x\gen{\ub}\subseteq \gen{a}$}, avec 
$(*)$
cela donne $a\fp\supseteq x\,\fp \gen{\ub}= xa\gen{\uc}$, \hbox{donc   $\fp\supseteq  x\gen{\uc}$}. Enfin, comme $xc_j\in\fp$ et $c_j\notin\fp$, on \hbox{obtient $x\in\fp$}. Ce qu'il fallait démontrer.
\end{proof}

\medskip \exl 
 \label{advl3exemple3} Dans l'exemple 3) \paref{advlexemple3}, les  \dvrs  $\dv(a,b)$, $\dv(a,c)$, $\dv(d,b)$ et~$\dv(d,c)$ sont \irds car les \ids $\gen{a,b}$, $\gen{a,c}$, $\gen{d,b}$ \hbox{et $\gen{d,c}$} sont premiers.
Comme $\dv(a,b)+\dv(a,c)=\dv(a)>0$, on a par raison de symétrie $\dv(a,b)>0$ et $\dv(a,c)>0$. On peut appliquer le \thref{corlemdivirdadvlgnl}. On obtient en outre sans calcul l'\egt
 $\Idv(a,b)=\gen{a,b}$ et les trois autres \egts analogues. 
\eoe

\begin{corollary} \label{corcorlemdivirdadvlgnl}
Dans un \advl  \coh \dvee  non trivial, on a les \prts \eqves suivantes pour un \id $\fq\neq \gen{0}$ arbitraire.
\begin{itemize}
\item L'\id $\fq$ est premier, \tfz, et $\dv(\fq)>0$.
\item L'\id $\fq$ est premier, \tf et $\dv(\fq)$ est \irdz.  
\item L'\id $\fq$ est un \idif premier $\neq \gen{1}$. 
\item Il existe un \dvr \ird $\pi$ tel que $\fq=\Idv(\pi)$.  
\end{itemize}
Et dans un tel cas $\fq$ est détachable.
\end{corollary}

En d'autres termes, on obtient
une bijection entre  les ensembles suivants:
\begin{itemize}
\item les \dvrs \irdsz, 
\item les \idifs premiers $\neq \gen{1}$, 
\end{itemize} 
et une \egt entre les ensembles suivants:
\begin{itemize}
\item les \idifs premiers $\neq \gen{1}$, 
\item les \itfs premiers $\fq\neq \gen{0}$ tels que $\dv(\fq)>0$. 
\item les \itfs premiers $\fq\neq \gen{0},\gen{1}$ tels que $\fq=\Idv(\fq)$. 
\end{itemize} 

%
\subsec{Propriétés de cl\^oture intégrale}

On note
$\Icl_\gA(\fa)$ la \cli de $\fa$ dans $\gA$. Quand le contexte est clair, on utilise~$\Icl(\fa)$.

\begin{theorem} \label{prop2Idv} 
Soit $\gA$ un anneau \dvla et $\fa$ un \itfz.
\begin{enumerate}
\item Pour tout $x$  entier sur $\fa$, on a $\dvA(x)\geq \dvA(\fa)$.
\\
En conséquence,
\begin{itemize}
\item $\dvA(\fa)$ ne dépend que de $\Icl(\fa)$,
\item  une inclusion $\fb\subseteq \Icl(\fa)$  implique $\dvA(\fb)\geq \dvA(\fa)$,
\item  on a $\Icl(\fa)\subseteq \Idv(\fa)$ et l'\id $\Idv_\gA(\fa)$ est \iclz.
\end{itemize}
\item En particulier $\gA$ est \iclz.
%
%
\end{enumerate}
\end{theorem}
%
\begin{proof} \textsl{1.}
On note $\xi=\dvA(x)$ et $\alpha=\dvA(\fa)$. 
 \\ La \rdi s'écrit

\snic{x^n=u_1x^{n-1}+\cdots+u_{n-1}x+u_n$ avec $u_k\in\fa^k.}

\noindent On a des in\egts 
$\dvA(u_k)\geq k\alpha$ et donc 

\snic{
n\xi\,\geq\, \Vi_{k=1}^{n} \big((n-k)\xi+\dv(u_k)\big)\geq\, \Vi_{k=1}^{n} \big((n-k)\xi+k \alpha\big), 
}

\noindent et l'on conclut
que $\xi\geq \alpha$  
par \aref{Fact XI-2.12, item \textsl{13.}}

\snii \textsl{2.} En effet, un anneau int\`egre est \icl \ssi ses \idps sont \iclz.
\end{proof}
\exl Dans un anneau à pgcd int\`egre, l'\id $\Idv(\fa)$ est l'\idp engendré par le pgcd des \gtrs de $\fa$. \\
Par exemple sur l'anneau  $\gA=\gk[x,y]$ ($\gk$ un \cdiz), on a 

\snic{\Idv_\gA\big(\gen{x^{4},y^{3}}\big)=\gen{1}$ et $\Icl_\gA\big(\gen{x^{4},y^{3}}\big)=\gen{x^4, y^3, x^2y^2, x^3y}.}

\noindent Donc $\gen{x^{4},y^{3}}\subsetneq\Icl_\gA\big(\gen{x^{4},y^{3}}\big) \subsetneq \Idv_\gA\big(\gen{x^{4},y^{3}}\big)$.
\eoe

\medskip 
\rem La \dem du \tho précédent n'a pas utilisé toute la force
d'un anneau à \dvrsz. Il suffisait d'avoir un \grl $G$ et une application
$\dv:\Atl/\Ati\to G$  qui satisfait les points $\rD_1$
et $\rD_2$ du projet \dvlg ainsi que le point \textsl{1} de la 
proposition~\iref{lemADVL}. \\
\Cadz: $\dvA$ est un morphisme de groupes ordonnés qui réfléchit les in\egts et qui satisfait l'in\egt $\dvA(a)\vi\dvA(b)\leq \dvA(a+b)$.
\eoe

\begin{corollary} \label{corprop2Idv}
Soit $\gA$ un \advlz. Pour \hbox{$p$, $q\in\AuX$} on a 

\snic{\dv\big(\rc(pq)\big)=\dv\big(\rc(p)\big)+\dv\big(\rc(q)\big).}
\end{corollary}
%
\begin{proof}
Cela résulte de $\dv(\fa\fb)=\dv(\fa)+\dv(\fb)$, du \tho de Kronecker (\cad \hbox{$\rc(p)\rc(q)\subseteq \Icl\big(\rc(pq)\big)$}, et du \thoz~\iref{prop2Idv}.
\end{proof}

Un autre corolaire est l'\eqvc donnée ci-apr\`es pour les anneaux \cohsz.
Notons que ce \tho est une version \cov non \noee du \tho bien connu des \clama qui affirme qu'un anneau \noe \icl est un anneau de Krull (voir \cite[Matsumura, Theorem 12.4]{Mat}).
\begin{theorem} \label{thiclcohidv}
Pour un anneau $\gA$  \und{\cohz} int\`egre, \propeq
\begin{enumerate}
\item L'anneau $\gA$ est \iclz.
\item L'anneau $\gA$ est \dvlaz.
\item Pour toute famille finie $(\an)$ dans $\Atl$, si 
$\gen{\bbm}=\big(\!\gen{a_1}:\gen{\an}\!\big)_\gA$,
la famille des $a_ib_j$ admet $a_1$ comme pgcd fort. 
\end{enumerate}
\end{theorem}
%
\begin{proof} On sait déjà que \textsl{2} $\Rightarrow$ \textsl{1}, et que lorsque $\gA$ est \cohz, \textsl{2} $\Leftrightarrow$ \textsl{3} 
(lemme~\ref{lemdvli}).

\noindent Il reste à montrer \textsl{1} $\Rightarrow$ \textsl{3.}
On pose $\fa=\gen{\an}$, $\fb=\gen{\bbm }$. On veut montrer que les $a_ib_j$ admettent $a_1$ comme
pgcd fort. 

\noindent Pour un $y\in\Atl$ on montre que $ya_1$ est un pgcd des~$ya_ib_j$. Il est clair que~$ya_1$ divise tous les~$ya_ib_j$. Soit un $x\in\Atl$ qui divise tous les~$ya_ib_j$, on doit montrer que la fraction $t=a_1y/x$ est dans $\gA$. 

\noindent On consid\`ere $t_j=tb_j=(ya_1b_j)/x$, qui est dans $\gA$, et on montre que $t_j\in\fb$.\\
On a $t_ja_i=a_1(ya_ib_j)/x\in\gen{a_1}$ pour chaque~$i$, \hbox{donc 
$t_j\in(\gen{a_1}:\fa)_\gA=\fb$}. 
Ainsi l'\eltz~$t$ de~$\gK$ vérifie $t\fb\subseteq \fb$. Comme $\fb$ contient $a_1\in\Atl$, c'est \hbox{un \Amoz} fid\`ele. En outre il est  \tfz, et puisque~$\gA$ est \iclz, on obtient $t\in\gA$. 
\end{proof}
NB: il existe des anneaux factoriels (cas particuliers d'\aKrsz) non \cohs (exemple 5.2 dans \cite[Glaz, 2001]{Glaz01}, voir  \paref{exempleFactnonCoh}).

\mni{\bf Exemple important.}
Tout anneau int\`egre \corg est \dvla car il est \iclz.
\eoe

\medskip 
\rem On a déjà noté que pour un \advl et un \ifr \tf $\fa$,
on a $\Idv(\fa)=(\fa^{-1})^{-1}$.
Donc pour un anneau  \coh \iclz,  un \itf  $\fa$ \hbox{et $a\in\fa$} 
on peut calculer un \sgr fini de $\Idv(\fa)$ par la formule
$\Idv(\fa)=(\gen{a}:(\gen{a}:\fa)_\gA)_\gA$.
\eoe

\medskip \noindent {\bf Terminologie.} Dans la littérature classique, pour un anneau \noe \iclz~$\gA$, le groupe $\DivA$ est souvent appelé le \textsl{groupe des \dvrs de Weil}, et l'on appelle \textsl{\dvr effectif} un \elt de $\DivAp$, et \textsl{\dvr positif} un \dvr effectif non nul. 
La signification du mot effectif étant différente en \comaz, nous utiliserons la terminologie suivante: \textsl{\dvr positif ou nul} (au lieu de \dvr effectif) et \textsl{\dvr strictement positif} (au lieu de \dvr positif).
\eoe

\subsec{\Advls de dimension $\leq 1$}

Un anneau int\`egre \zed est un corps discret. Ceci r\`egle la question des \advls de dimension $\leq 0$.

Dans ce paragraphe, nous voyons que la question de la dimension $\leq 1$ admet une réponse surprenante par sa simplicité.

Comme corolaire du \thref{thiclcohidv}, puisqu'un \ddp de \ddkz~\hbox{$\leq 1$} est la même chose qu'un
anneau int\`egre \coh \icl de \ddk $\leq 1$ (\aref{Theorem XII-6.2}) on obtient l'\eqvc suivante: un anneau  int\`egre  de  \ddkz~$\leq 1$ est 
un \adp \ssi c'est un \advl \cohz.

Mais en fait on a mieux, car on peut supprimer la \cohc dans cette \eqvcz.
On obtient alors une version \cov non \noee du \tho des \clama qui affirme qu'un
anneau de Krull de dimension $1$ est un \dok (\cite[Theorem 12.5]{Mat}). 

\begin{theorem} \label{thi2clcohidv} 
Pour un anneau  int\`egre  \propeq
\begin{enumerate}
\item $\gA$ est un \adp de \ddk $\leq 1$.
\item $\gA$ est un \advl  de \ddk $\leq 1$.
\end{enumerate}
\end{theorem}
%
\begin{proof} Il faut prouver \textsl{2} $\Rightarrow$ \textsl{1.} On suppose que $\gA$ est un \advl et on doit montrer qu'un \id $\fa$ \tf  fid\`ele est \ivz.
On a un \itf $\fb$ tel que~$\fa\fb$ admet un pgcd fort~$g\in\Atl$, donc par le lemme \ref{lemthi2clcohidv},  $\fa\fb=\gen{g}$  et $\fa$ est bien un \id \ivz.    
\end{proof}
%

\begin{lemma} \label{lemthi2clcohidv}
Dans un anneau int\`egre $\gA$ de \ddk $\leq 1$, si $g$ est un pgcd fort
de $(\an)$, alors $\gen{\an}=\gen{g}$. 
\end{lemma}
\begin{proof} 
On suppose \spdg les $a_i\in\Atl$.
Puisque $g$ est un pgcd fort des~$a_i$,~$1$ est un pgcd fort des $b_i=\fraC{a_i}g$ et il suffit de montrer \hbox{que $\gen{\bn}=\gen{1}$}. 
Pour cela, on reprend mutatis mutandis la \dem du  
\tho XI-3.12 dans~\cite{ACMC}
qui affirme qu'un anneau int\`egre à pgcd de dimension $\leq 1$ est un anneau de~Bézout.
\\
Dans l'anneau \zed $\aqo\gA{b_1}$,  chaque $b_i$ satisfait $\geN{b_i^{m_i}}=\gen{e_i}$
pour un \idmz~$e_i$. Si~$e$ est l'\idm pgcd des
$e_i$ modulo $b_1$, on a dans~$\gA$ l'\egt $\gen{b_1}=\gen{b_1,e}\gen{b_1,1-e}$,
donc l'\id $\gen{b_1,e}$ est \lopz. 
Et $\gen{b_1,e}=\gen{b_1,b_2^{m_2},\dots,b_n^{m_n}}$. 
Comme $\gen{\bn}$ est de profondeur $\geq 2$, il en va de même pour $\gen{b_1,b_2^{m_2},\dots,b_n^{m_n}}$.
Ainsi $\gen{b_1,e}$ est \lop et de profondeur $\geq 2$. On en déduit qu'il est égal à~$\gen{1}$ car un \idp de profondeur $\geq 2$ est égal à $\gen{1}$. On conclut avec $\gen{\bn}\supseteq \gen{b_1,b_2^{m_2},\dots,b_n^{m_n}}=\gen{1}$.
\end{proof}

\subsec{Anneaux de valuation discr\`ete}

On rappelle qu'un \textsl{\adv discr\`ete} est par \dfn un anneau int\`egre $\gV$ 
donné avec un \elt $p\notin\gV\eti$, et dans lequel tout \elt de~$\gV\etl$ s'écrit $p^{n}u$ pour un  $n\in \NN$ et un $u\in\gV\eti$. 
On dit alors que $p$ est une \ix{uniformisante} de $\gV$. On peut  voir
$\gV$ comme un anneau principal à \fat avec $p$ pour seul \ird (modulo l'association).

\begin{lemma} \label{lemVALD} \emph{(Anneaux de valuation discr\`ete, 1)}\\
Soit  $\gA$ un anneau int\`egre.
\Propeq %
\begin{enumerate}
\item \label{i1lemVALD} $\gA$ est un \adv discr\`ete.
\item \label{i2lemVALD} $\gA$ est un \advl et $\DivA\simeq (\ZZ,\geq)$. 
\item \label{i3lemVALD} $\gA$ est un \advl local de dimension $\leq 1$, $\DivA$ est discret et contient un \elt \irdz. 
\item \label{i4lemVALD} $\gA$ est un anneau principal local à \fat et il existe un $a\in\Atl\setminus\Ati$. 
\end{enumerate}
\end{lemma}
\begin{proof} \textsl{\ref{i1lemVALD}} $\Rightarrow$ \textsl{\ref{i2lemVALD}}, 
\textsl{\ref{i3lemVALD}} et \textsl{\ref{i4lemVALD}.} Clair

\sni\textsl{\ref{i2lemVALD}} $\Rightarrow$ \textsl{\ref{i1lemVALD}.} 
Soit $\pi$ le \gtr $>0$ de $\DivA$. On a $\pi=\dv(\ua)$ pour une famille finie $(\an)=(\ua)$. Puisque $\dv(a_i)=n_i\pi$ pour des $n_i\geq 0$, l'un des $n_i$ est égal à $1$. \hbox{Ainsi $\pi=\dv(p)$} pour un $p\in\Atl$. Pour tout autre \elt $a$ de $\Atl$
on a $\dv(a)=\dv(p^{n})$ pour un $n\geq 0$, donc $a$ et $p^{n}$ sont associés. Ainsi $\gA$ est un \adv discr\`ete d'uniformisante $p$.

\sni\textsl{\ref{i3lemVALD}} $\Rightarrow$ \textsl{\ref{i1lemVALD}.} En tant qu'\advl local de dimension~\hbox{$1$, $\gA$} est un \adv (\thref{thi2clcohidv}).
Puisque $\DivA$ est discret, $\gA$ est \dveez, et $\gA$ est réunion disjointe explicite de $\Ati=\sotq x{\dv(x)=0}$ et $\Rad\gA=\sotq x{\dv(x)>0}$.
Si $\pi$ est un \dvr \irdz, il est \elt $>0$ minimum dans le groupe totalement ordonné $\DivA$,  on l'écrit $\pi=\dv(p)$. 
Donc pour tout $a\in\Rad\gA$, $p$ divise $a$. En outre
la \ddkz~$\leq 1$ de l'anneau nous donne une \egt $p^n(1+px)+ay=0$, donc $a$ divise $p^{n}$ pour un $n>0$. Si $n$ est la plus petite valeur possible et $az=p^{n}$ 
alors $z\in\Ati$ car $0\leq \dv(z)<\dv(p)$.
Ainsi $\gA$ est un \adv discr\`ete d'uniformisante $p$.

\sni\textsl{\ref{i4lemVALD}} $\Rightarrow$ \textsl{\ref{i3lemVALD}.}
En effet $\gA$ est un \advz, et $\DivA$ est discret parce que $\gA$ est à \fatz. En outre, $\dv(a)>0$, donc il existe un \elt \irdz. 
\end{proof}

\rem Sous la seule hypoth\`ese que $\gA$ est un anneau principal local
à \fab
avec un  $a\in\Atl\setminus\Ati$, il n'y a pas d'\algo \gnl pour produire
un \elt \ird dans $\gA$.
\eoe

\subsec{Localisations d'un \advlz, 1}

\begin{theorem} \label{lemdvlloc0}
Soient $\gA$ un \advl et $S$ un filtre ne contenant pas~$0$.
 L'anneau $S^{-1}\gA=\gA_S$ est un \advl et 
 il y a un unique morphisme  de \grls 
$\varphi_S:\DivA\to\Div \gA_S$
tel 
{que $\varphi_S\big(\dv_\gA(a)\big)=\dv_{\gA_S}(a)$} pour tout $a\in\Atl$. 
Ce morphisme est surjectif, donc $\Div \gA_S\simeq (\DivA)/\Ker\varphi_S$. 
\end{theorem}
%
\begin{proof}  
On note tout d'abord que $\gA_S$ reste un anneau int\`egre. Ensuite vue la proposition~\ref{defiStrongPGCD} (point \textsl{1c}), le résultat tient à ce qu'un ppcm reste un ppcm apr\`es \lonz,
car les \lons préservent les intersections
finies d'\ids (\aref{\hbox{Fact II-6.5}}). Ceci implique que les pgcd forts  restent des pgcd forts, que les listes \dvlis restent \dvlisz, que l'application 
$$
\varphi_S:\dvA(\xn)\mapsto\dv_{\gA_S}(\xn)
$$
\noindent 
est bien définie, 
de $\DivA$ vers $\Div\gA_S$, et que c'est un morphisme de \grlsz.
L'unicité est claire.
\end{proof}

Pour un filtre $S$ d'un anneau $\gA$ on appelle \textsl{hauteur de $S$} la \ddk de $\gA_S$. En fait, du point de vue \cofz, la phrase bien définie est: \gui{le filtre~$S$ est un filtre de hauteur $\leq k$}. Le filtre~$S$ n'est pas supposé détachable. Un filtre~$S$ est dit \textsl{premier} si l'anneau $\gA_S$ est local. Autrement dit si l'implication suivante est satisfaite
$$
x+y\in S\Rightarrow x\in S \hbox{ ou }y\in S.
$$
\noindent Lorsque $S$ est premier, détachable et ne contient pas $0$, son \cop  est un \idep détachable $\fp\neq \gA$,
et on appelle  \textsl{hauteur de $\fp$} la hauteur de~$S$.   Enfin un \idep détachable $\neq \gA$  admet pour \cop un filtre premier.
On obtient ainsi une bijection entre les \ideps détachables $\neq \gA$  
 et les filtres premiers détachables $\neq \gA$.

Le lemme suivant est un complément pour le \thref{corlemdivirdadvlgnl}.
\begin{lemma} \label{cor0lemdvlloc}
Soit un \advl $\gA$ et $\fp$ un \itf premier détachable de hauteur~$1$. Alors $\dv(\fp)$ est un \dvr \irdz.
\end{lemma}
\begin{proof}
On doit montrer que  $\dv(\fp)>0$. Si $S=\gA\!\setminus \fp$, on sait que $\gA_S$ est un \advl avec un morphisme naturel surjectif de \grls $\DivA\to\Div\gA_S$ (\thref{lemdvlloc0}), et par hypoth\`ese $\gA_S$ est de \ddk $1$. Donc par le \thref{thi2clcohidv} c'est un \adpz. Comme il est local c'est un \advz,
d'\idema $\fp\gA_S$.
On a $\dv_{\gA_S}(\fp)=\dv_{\gA_S}(\an)=\Vi_i\dv_{\gA_S}(a_i)$ pour des $a_i\in\Atl\cap\fp$. Comme les 
$\dv_{\gA_S}(a_i)$ sont deux à deux comparables et tous $>0$, on a $\dv_{\gA_S}(\fp)>0$, ce qui implique $\dvA(\fp)>0$. 
\end{proof}
%

\subsec{Un anneau à la Kronecker}

Dans ce paragraphe, les \dems sont laissées \alecz.
 
\smallskip 
Soit $\gB$ un anneau arbitraire et $\Xn$ des indéterminées. On rappelle que l'on note $\rc_{\gB,\uX}(p)$ ou $\rc(p)$ le contenu de $p\in\BuX=\BXn$,
\cad l'\id de $\gB$ engendré par les \coes de $p$. 

On fixe les \idtrs $\uX$ et  
on définit l'ensemble
$$
S_{\dv}(\gB) =S_{\dv} = \sotq{f\in\gB[\uX]}{\Gr(\rc_\gB(f))\geq 2}.
$$
\begin{fact} \label{factSKRprof2}
L'ensemble $S_{\dv}$ est un filtre.
\end{fact}

\begin{definition} \label{defKRprof2}
Le sous-anneau 
${\gB_{\dv}(\uX)=S_{\dv}^{-1}\gB[\uX]}$
de $(\Frac\gB)(\uX)$ est appelé \textsl{l'anneau de Nagata divisoriel de $\gB$}.
\end{definition}

\rem Cette \dfn est à distinguer de celle des \gui{Kronecker function rings}
usuels de la littérature anglaise en théorie multiplicative des \ids (voir à ce sujet le survey \cite{FL06}). Le filtre $S_{\dv}$ appara\^{\i}t  dans la littérature usuelle dans le cas d'un anneau int\`egre $\gB$ de corps de fractions $\gK$, et l'anneau  $\gB_{\dv}(\uX)$ est alors appelé  \textsl{anneau de Nagata pour la star opération $v$} définie comme suit:
$v:\fa\mapsto(\fa^{-1})^{-1}$. Pour un anneau int\`egre, notre
$\gB_{\dv}(X)$ est noté $\mathrm{Na}(\gB,v)$ ou $\mathrm{Na}(\gB,v)(X)$ par d'autres auteurs.
\eoe

\medskip On vérifie alors les \prts suivantes. 
\begin{proposition} \label{propKRprof2}~
 \begin{enumerate}
\item  Pour $a$, $b\in\Reg(\gB)$,
$a$ divise $b$ dans $\gB$ \ssi $a$ divise $b$ dans $\gB_{\dv}(\uX)$.
\item \label{i2propKRprof2} Si $\gA$ est un \advlz, alors
\begin{enumerate}
\item  $\gA_{\dv}(\uX)$ est un anneau de Bézout,
\item  tout $p\in\AuX$ est le pgcd  dans $\gA_{\dv}(\uX)$ de ses \coesz,
\item  pour $f$, $g\in\AuX$, $\fraC f g\in\gA_{\dv}(\uX) \iff\dvA(\rc(g))\leq \dvA(\rc(f))$. 
\end{enumerate}
\end{enumerate}
\end{proposition}

En particulier on obtient \fbox{$\DivA\simeq \Div {(\gA_{\dv}(\uX))}\simeq \gA_{\dv}(\uX)\etl/\gA_{\dv}(\uX)\eti$}.

Le  \thref{corplccAnnDivl} démontre la réciproque suivante: si 
$\gA$ est int\`egre et si~$\gA_{\dv}(\uX)$ est \ari (à fortiori si c'est un anneau de Bézout), alors~$\gA$ est un \advlz.

\subsec{Principe local-global  et applications}

\newcommand{\dfd}{de profondeur $\geq 2$ }
\newcommand{\dfdz}{de profondeur $\geq 2$}
\newcommand{\dfde}{de profondeur $\geq 2$ }
\newcommand{\dfdez}{de profondeur $\geq 2$}

\begin{plcc} \label{plccAnnDivl} 
\emph{(\Plgc pour la \dvez, les \idlisz, les \acls  et les \advlsz)}.
\\
Soit $\gA$ un anneau int\`egre et 
$(\sn)$ une suite \dfdz. On note $\gA_i=\gA[\fraC1{s_i}]$. 
\begin{enumerate}
\item Soient $a$, $b$, $a_1$, \dots, $a_k\in\gA$.
\begin{enumerate}
\item \label{i1plccAnnDivl} $a$ divise $b$ dans $\gA$ \ssi $a$ divise $b$ dans chaque $\gA_i.$
\item \label{i2plccAnnDivl} $a$ est un pgcd fort de $(a_1,\dots,a_k)$ dans $\gA$ \ssi $a$ est un pgcd fort de $(a_1,\dots,a_k)$ dans chaque $\gA_i$.
\item \label{i3plccAnnDivl} L'\id $\fa=\gen{a_1,\dots,a_k}$ est \dvli dans $\gA$ \ssi il est
\dvli dans chaque $\gA_i$.
\end{enumerate}
\item \label{I2plccAnnDivl} L'anneau $\gA$ est \icl \ssi chaque anneau~$\gA_i$ est \iclz.
\item \label{i4plccAnnDivl} L'anneau $\gA$ est \dvla \ssi chaque anneau $\gA_i$ est \dvlaz.
\end{enumerate}
\end{plcc}
%
\begin{proof} \textsl{1b} (et à fortiori \textsl{1a}).
Pour simplifier on se contente de localiser en trois \mosz~$S$, $T$ et~$U$.
Soient $c$, $y\in\Atl$ et supposons que $c$ divise les $ya_i$ dans $\gA$.
Apr\`es \lon on trouve que~$c$ divise  $sya$ pour \hbox{un $s\in S$}, que~$c$ divise~$tya$ pour \hbox{un $t\in T$} et que~$c$ divise~$uya$ pour \hbox{un $u\in U$}. Comme $(s,t,u)$ admet $1$
pour pgcd fort dans $\gA$, cela implique que $c$ divise $ya$ dans $\gA$. 

\snii \textsl{1c.} La \dem est analogue, donnons les détails.
\\
On consid\`ere un \itf $\fa=\gen{a_1,\dots,a_\ell}$ avec $a_1\in\Atl$. 
On doit trouver un \itf $\fb$ tel que
$\fa\fb$ admette $a_1$ comme pgcd fort (lemme \iref{lemdvli}). Pour simplifier on se contente de localiser en deux \mos $S$ et $T$. 
\\
Dans le premier localisé on trouve
$b_1$, \dots, $b_m$ tels que les $u_{ij}=a_ib_j$ admettent~$a_1$ pour pgcd fort (lemme \iref{lemdvli}). 
\\
Dans le deuxi\`eme localisé  on trouve
$c_1$, \dots, $c_p$ tels que les $v_{ik}=a_ic_k$ admettent~$a_1$ pour pgcd fort.
\\
On peut supposer que les $b_j$ et $c_k$ sont dans $\gA$ ainsi que \hbox{les 
$u_{ij}/a_1$} \hbox{et $v_{ik}/a_1$}.
On va montrer que la famille finie formée par les $u_{ij}$ et $v_{ik}$ admet $a_1$ pour pgcd fort dans
$\gA$. Pour cela on consid\`ere $c$ \hbox{et $y\in\Atl$} tels que $c$ divise tous \hbox{les $yu_{ij}$ et $yv_{ik}$}. 
Apr\`es \lon on trouve que $c$ divise  $sya_1$ pour \hbox{un $s\in S$} et~$c$ divise~$tya_1$ pour \hbox{un $t\in T$}. Comme $s$ et~$t$ admettent~$1$
pour pgcd fort dans~$\gA$, cela implique que $c$ divise $ya_1$ dans $\gA$.
Ainsi l'\itfz~\hbox{$\fb=\gen{\bbm,c_1,\dots,c_p}$} satisfait la condition que $\fa\fb$
admet~$a_1$ pour pgcd fort.

\snii 
\textsl{\iref{I2plccAnnDivl} et \iref{i4plccAnnDivl}.} Résultent des points \textsl{1a} et \textsl{1c}.
\end{proof}

Le corolaire suivant compl\`ete la proposition \iref{propKRprof2}.
Nous retrouvons ici un résultat de \cite[Kang, 1989]{Kang}.
\begin{theorem} \label{corplccAnnDivl} \emph{(L'anneau de Nagata divisoriel)}\\
Un anneau int\`egre $\gA$ est un \advl \ssi son anneau de Nagata divisoriel~$\gA_{\dv}(\uX)$
est un \ddpz, ou encore un domaine de Bézout.
\end{theorem}
\begin{proof}
Nous montrons que si $\gA_{\dv}(\uX)$ est un \ddpz, tout \id $\gen{a,b}$ est \dvliz. Dans $\gA_{\dv}(\uX)$ on a des \elts $u$, $v$, $s$, $t$ tels \hbox{que $sa=ub$, $tb=va$}, \hbox{et $s+t=1$}. On peut prendre $u$, $v$, $s$, $t\in \AuX$ auquel cas on obtient
$\Gr_\gA(\rc(s+t))\geq 2$. Les \coes de $s$ et $t$ engendrent donc un \id   \dfd (i.e. $\Gr_\gA(\rc(s)+\rc(t))\geq 2$). 
\\
Or lorsqu'on inverse un \coe $s_k$ de $s$ l'\egt $sa=ub$ dans $\AuX$ nous donne $\gen{a,b}=\gen{b}$ dans~$\gA[1/s_k]$, et lorsqu'on inverse un \coe 
$t_\ell$ de~$t$ on obtient $\gen{a,b}=\gen{a}$ dans $\gA[1/t_\ell]$. Dans tous les cas l'\id $\gen{a,b}$ est \dvli dans le localisé. On conclut par le point \textsl{1c} du \plgz~\iref{plccAnnDivl} que $\gen{a,b}$ est \dvli dans $\gA$.  
\end{proof} 
%

Un autre corolaire intéressant est le suivant.
\begin{theorem} \label{cor2plccAnnDivl} \emph{(Idéaux \dvlis comme \itfs \gui{\lotz} principaux)}
\begin{enumerate}
\item Dans un anneau int\`egre  
un \itf fidèle est \dvli \ssi il existe une suite $(s_1,\dots,s_p)$ \dfd telle qu'après \lon en chaque $s_i$,
l'\id devient principal.
\item En conséquence un anneau intègre est un \advl \ssi tout \itf devient principal après \lon en les \elts d'une suite \dfd.
%
%
\end{enumerate}
\end{theorem}
\begin{proof} Il suffit de démontrer le point \textsl{1}.\\
La condition est suffisante d'apr\`es le point \textsl{1c} 
du \plgz~\ref{plccAnnDivl}.

\snii \textsl{La condition est \ncrz.} On considère un \id $\fa=\gen{\an}$  et un \id $\fb=\gen{b_1,\dots,b_m}$
avec $(\an)\star(b_1,\dots,b_m)=g\,(s_{ij})_{i\in\lrbn,j\in\lrbm}$, $g$ régulier et $\Gr(\underline s)\geq 2$.
On a $gs_{ij}a_k=a_ib_ja_k=gs_{kj}a_i$, donc $s_{ij}a_k=s_{kj}a_i$. D’où  $s_{ij}\fa\subseteq \gen{a_i}$, et $\fa=\gen{a_i}$ dans $\gA[1/s_{ij}]$.
\end{proof}
\penalty-2500
\section{Propriétés de \dcn  des \grlsz} \label{secpropdecgrl}

\subsec{Groupes réticulés quotients}

Un sous groupe $H$ d'un groupe (abélien) ordonné est dit \textsl{convexe}
s'il vérifie la \prtz: $0\leq x\leq y$ et $y\in H$ impliquent $x\in H$. Sous cette condition,~$G/H$ est muni d'une structure de \textsl{groupe ordonné quotient}, i.e. vérifiant la \prtz:
$(G/H)^{+}=G^{+}+H$. Certains auteurs disent \textsl{sous-groupe isolé}.%
\index{convexe!sous-groupe ---}\index{isole@isolé!sous-groupe ---}%
\index{sous-groupe!isole@isolé}\index{sous-groupe!convexe}
 
Le noyau $H$ de la \prn canonique d'un \grl $G$ sur un \grl quotient $G/H$ est
ce que l'on appelle un \textsl{sous-groupe solide}\footnote{On peut consulter \cite{BKW}. Le \tho 2.2.1 donne les \prts \eqves suivantes pour un sous-groupe: (a) $H$ est solide, (b) $H$ est un sous-\grl convexe, (c) $H$ est convexe et $\xi\in H\Rightarrow \xi^{+}\in H$, (d) $H$ est convexe et filtrant.}, \cad un sous-groupe vérifiant la \prtz: $\xi\in H$  \hbox{et $\abs \zeta\leq \abs \xi$} impliquent~\hbox{$\zeta\in H$}.

Pour $\gamma\in G$ on note $\cC(\gamma)$ le sous-groupe solide engendré par $\gamma$, qui est l'ensemble des $\xi$ tels que $\abs{\xi}$ soit majoré par un \elt $n\abs \gamma$
($n\in\NN$). On a donc $\cC(\gamma)=\cC(\abs \gamma)$.

Pour $\gamma$ et $\eta\in G^{+}$ on a $\cC(\gamma)\cap\cC(\eta)=\cC(\gamma\vi \eta)$,
et $\cC(\gamma)\perp \cC(\eta)$ \ssi $\gamma\perp \eta$. 
Enfin $\cC(\gamma+\eta)=\cC(\gamma \vu \eta)$ est le plus petit sous-groupe solide 
contenant~$\cC(\gamma)$ et~$\cC(\eta)$. 

Nous  définissons maintenant l'analogue des \molos en un \mo (définis dans la catégorie des anneaux commutatifs en \aref{XV-4.5}) dans la catégorie des \grlsz.

\begin{definition}
\label{defHomQuoGRL} \textsl{(Morphismes de passage au quotient pour les \grlsz)}
 Soit  $G$ un \grl et $H$ un sous-groupe solide de $G$.  Un
morphisme $\Pi:G\to G'$ est appelé un \textsl{\moquo par~$H$} s'il est surjectif et si $\Ker\Pi=H$.
\end{definition}

Un \moquo pour un $H$ donné est \gui{unique à \iso unique pr\`es}: il y a un unique \homo $G'\to G/H$ qui fait commuter le diagramme convenable, et c'est un \isoz.  

Notez que comme $\Pi(z^{+})=\Pi(z)^{+}$ pour tout $z$, on a $\Pi(G)^{+}=\Pi(G^{+})$.

\subsec{Un principe de recouvrements par quotients}

Rappelons le principe \cof suivant énoncé en 
\aref{XI-2.10}.

\mni{\bf Principe de recouvrement par quotients pour les \grlsz.} 
\rdb \label{Prqgrl}\\
\textsl{Pour démontrer une \egt $\alpha=\beta$ dans un \grlz, on peut toujours supposer que les \elts
(en nombre fini) qui se présentent dans un calcul pour une \dem 
de l'\egt sont comparables,
si l'on en a besoin pour faire la \demz.
Ce principe s'applique aussi bien pour des in\egts que pour des \egts
puisque $\alpha\leq \beta$ équivaut à $\alpha\vi \beta=\alpha$.}

\medskip  Ce principe peut être considéré comme la version \cov du \tho de \clama qui dit qu'un \grl est toujours représentable comme sous-\grl d'un produit de groupes totalement ordonnés.

\subsec{Groupes réticulés de dimension $1$}

Pour des sous-\grls $H$, $K$, $L$ d'un \grl $G$, la notation $K=H\boxplus L$  signifie que $H\perp L$ et~$K=H\oplus L$. Autrement dit, lorsque le produit cartésien $H\times L$ est muni de la structure produit catégorique, l'application  $H\times L\to K, \,(\xi,\eta)\mapsto \xi+\eta$
est un \iso de \grlsz. 

En particulier, en notant $H\epr=\sotq{x\in G}{\forall h\in H, \,x\perp h}$, on a alors $H+H\epr=H\boxplus H\epr$.

\begin{lemma} \label{lemGrl02}
Soient un \grl $G$, et $\alpha$, $\beta\in G^{+}$. \Propeq
\begin{enumerate}
\item $\beta\in \cC(\alpha)\boxplus\cC(\alpha)\epr$.
\item $\exists n\in\NN$, $\exists \beta_1,\,\beta_2\in G^{+}$, $\beta_1\leq n\alpha$,
$\beta_2\perp \alpha$ et $\beta=\beta_1+\beta_2$.
\item $\exists n\in\NN$, $\beta\vi n\alpha=\beta\vi (n+1)\alpha$,
i.e. $n\alpha\geq \beta\vi (n+1)\alpha$.
\item $\exists n\in\NN$, $\forall m\geq n$, $\beta\vi m\alpha=\beta\vi (m+1)\alpha$.
\item $\cC(\beta)\subseteq  \cC(\alpha)\boxplus\cC(\alpha)\epr$.
\end{enumerate}
 
\end{lemma}
\begin{proof} \textsl{1} $\Leftrightarrow$ \textsl{2}, \textsl{4} $\Rightarrow$ \textsl{3}, et \textsl{5} $\Rightarrow$ \textsl{1.} Clair

\snii\textsl{2} $\Rightarrow$ \textsl{4.} Il suffit de le démontrer lorsque $\beta_2\geq \alpha$ et lorsque $\beta_2\leq \alpha$. 
\\
Si $\beta_2\geq \alpha$ alors $\alpha=0$, 
et $\beta\vi m\alpha=0$ pour tout $m$. Si $\beta_2\leq \alpha$ alors $\beta_2=0$ donc $\beta=\beta_1\leq n\alpha\leq m\alpha$ pour $m\geq n$, donc $\beta\vi n\alpha=\beta\vi m\alpha$.

\snii\textsl{3} $\Rightarrow$ \textsl{2.}  On pose $\beta_1=\beta\vi n\alpha$ et $\beta_2=\beta-\beta_1$.
Il suffit de montrer $\beta_2\vi \alpha=0$ lorsque $\beta\geq n\alpha$ et lorsque $\beta\leq n\alpha$.
Si $\beta\leq n\alpha$, alors $\beta_1=\beta$ et $\beta_2=0$. Si $\beta\geq n\alpha$, alors $\beta_1=n\alpha$, et 

\snac{ 
 n\alpha    =   \beta \vi n\alpha   =  \beta\vi (n+1)\alpha = (\beta_1+\beta_2)\vi(n\alpha+\alpha) 
  =     (n\alpha+\beta_2)\vi(n\alpha+\alpha) 
     =   n\alpha+(\beta_2\vi \alpha),} 

\noindent  donc $\beta_2\vi \alpha=0$.

\snii\textsl{1} $\Rightarrow$ \textsl{5.} La \sdo de deux sous-groupes
solides est un sous-groupe solide.
\end{proof}
%

\begin{definition} \label{defdim1grl} ~\\
Un \grl $G$ est dit \textsl{de dimension $\leq 1$} si pour tout  $\xi\in G^+$, on a $G=\cC(\xi)\boxplus\cC(\xi)\epr$. Pour des \prts \eqvesz, 
voir le lemme~\iref{lemGrl02}.
\end{definition}

\begin{lemma} \label{factQuodiscret}
Soit $G$ un \grl de dimension $\leq 1$ et $\xi\in G$. \Propeq
\begin{enumerate}
\item $G$ est discret
\item $\cC(\xi)$ et $G/\cC(\xi)$ sont discrets.
\item $\cC(\xi)$  est détachable et discret.
\item  $\cC(\xi)\epr$   est détachable et discret.
\end{enumerate}
\end{lemma}
\begin{proof} \textsl{1} $\Leftrightarrow$ \textsl{2.} L'\egt $G=\cC(\xi)\boxplus\cC(\xi)\epr$ donne les \isos $G/\cC(\xi)\simeq \cC(\xi)\epr$ et~$G\simeq \cC(\xi)\times G/\cC(\xi)$. Et un produit de deux groupes est discret \ssi chaque facteur est discret.

\snii \textsl{2} $\Leftrightarrow$ \textsl{3.} $G/\cC(\xi)$ est discret
\ssi  $\cC(\xi)$  est détachable.

\snii \textsl{3} $\Leftrightarrow$ \textsl{4.} L'\egt $G=\cC(\xi)\boxplus\cC(\xi)\epr$ montre que dans cette affaire $\cC(\xi)$ et  $\cC(\xi)\epr$ jouent des r\^oles symétriques.
\end{proof}
%

\begin{lemma} \label{lemirreddim1} \emph{(\'Eléments \irds dans un \grl de dimension~\hbox{$\leq 1$})}  \\ 
Soit $G$ un \grl discret de dimension $\leq 1$ et $\pi>0$ dans~$G$.
\Propeq
\begin{enumerate}
\item $\pi$ est un \elt \irdz.
\item $\cC(\pi)=\ZZ\pi\simeq (\ZZ,\geq 0)$.
\item Tout $\alpha\in G^{+}$ s'écrit de mani\`ere unique sous forme $n\pi+\alpha_2$ avec $n\in \NN$ et $\alpha_2\perp \pi$ dans~$G^{+}$.
\end{enumerate}
\end{lemma}
\begin{proof} \textsl{1} $\Rightarrow$ \textsl{2}  et  \textsl{3} $\Rightarrow$ \textsl{1.} Clair.

\snii
\textsl{2} $\Rightarrow$ \textsl{3.}
Soit $\alpha\in G^{+}$. On écrit $\alpha=\alpha_1+\alpha_2$ avec $0\leq \alpha_1\leq m\pi$, $m\in\NN$ \hbox{et $\alpha_2\perp \pi$} dans~$G^{+}$
(point \textsl{2} du lemme \ref{lemGrl02}).  Puisque $G$ est discret on a \hbox{un $n\in\NN$} tel que $n\pi\leq \alpha_1<(n+1)\pi$. 
\\Donc $\pi=(\alpha_1-n\pi)+((n+1)\pi-\alpha_1)$ tous deux $\geq 0$ et le deuxi\`eme $> 0$. Ceci implique que le premier est nul.
\end{proof}
%

\subsec{Groupes totalement ordonnés de dimension $1$}
Un \textsl{groupe totalement ordonné} est un groupe ordonné dans lequel on \hbox{a $x\geq 0$} \hbox{ou $x\leq 0$} pour tout $x$. Il est facile de voir que c'est un \grlz.

On note $(\OO,\geq )$ l'ensemble ordonné des réels  pour lesquels on dispose d'un test de comparaison aux rationnels.  C'est aussi la réunion disjointe de~$\QQ$ et des irrationnels,
définis par leurs fractions continues infinies.
L'ensemble $\OO$ est stable pour les opérations $x\mapsto (ax+b)/(cx+d)$,
où $a$, $b$, $c$, $d$ sont des entiers avec $ad-bc\neq 0$.
On a aussi une application bien définie \gui{identité}, de $\OO$ vers $\RR$. 
Par contre on ne peut pas démontrer \cot que~$\OO$ soit égal à~$\RR$, ni que~$\OO$ soit discret, ou stable pour l'addition, ou stable pour la multiplication. Une partie $R$ de $\OO$ sera appelée un sous-groupe (additif) de $\OO$
si l'on a une fonction $+:R\times R\to \RR$ qui redonne l'addition dans $\RR$.

\begin{lemma} \label{lemTOdim1} \emph{(Groupes totalement ordonnés 
discrets archimédiens)}\\
Soit $G$ un \grl discret non nul. \Propeq
\begin{enumerate}
\item $G$ est totalement ordonné de dimension $1$.
\item Pour tous $\alpha$, $\beta>0$ il existe $m\in N$ tel que $m\alpha>\beta$.
\item Pour tout $\alpha>0$, $G=\cC(\alpha)$.
\item $G$ est de dimension $1$ et pour tout $\alpha>0$, $G=\cC(\alpha)$.
\item  $G$ est isomorphe à $(R,\geq)$ pour un sous-groupe discret de
$(\OO,\geq )$.
\item Pour tout $\alpha>0$ dans $G$, on a un \iso $\varphi_\alpha:G\to R_\alpha$ tel que $\varphi_\alpha(\alpha)=1$, où $(R_\alpha,\geq )$  est un sous-groupe discret de
$(\OO,\geq )$. En outre $\varphi_\alpha$ et $R_\alpha$ sont déterminés de mani\`ere unique.
\end{enumerate}
\end{lemma}
\begin{proof} \textsl{1} $\Rightarrow$ \textsl{2}, \textsl{6} $\Rightarrow$ \textsl{5} $\Rightarrow$ \textsl{2},
 et \textsl{2} $\Leftrightarrow$ \textsl{3} $\Leftrightarrow$ \textsl{4.} Clair

\snii
\textsl{4} $\Rightarrow$ \textsl{1.}  (l'ordre est total).
Pour $\alpha$ et $\beta\in G^{+}$ il suffit de montrer que $\alpha\vi\beta=0$ 
implique $\alpha=0$ ou~\hbox{$\beta=0$}. On écrit $\alpha=\alpha_1+\alpha\vi\beta$
et $\beta=\beta_1+\alpha\vi\beta$. On a $\alpha_1\perp\beta_1$. Donc si $\alpha_1>0$, une in\egt $\beta_1\leq m\alpha_1$ implique $\beta_1=0$, donc $\beta\leq \alpha$.

\snii \textsl{1} et \textsl{2} $\Rightarrow$ \textsl{6.} Calcul  classique du développement en fraction continue.
\end{proof}
%

\subsec{Sous-groupes premiers d'un \grlz}

\begin{definition} \label{defsgpprem} ~\\
Un sous-groupe solide $H$ d'un \grl $G$ est dit \textsl{premier} 
si le quotient~$G/H$ est un groupe totalement ordonné.
\end{definition}
Naturellement $G/H$ est discret \ssi $H$ est une partie détachable de~$G$.

\begin{lemma} \label{lem2irreddim1} 
Soit $G$ un \grl discret de dimension $\leq 1$ et $\pi$ un \elt \irdz.
\\
Alors $\pi\epr$ est un sous groupe premier et $G/\pi\epr\simeq \cC(\pi)= \ZZ\pi$.
\end{lemma}
\begin{proof}
D'apr\`es la \dfn \ref{defdim1grl}, $G/\pi\epr\simeq \cC(\pi)$ et l'\egt $\cC(\pi)= \ZZ\pi$ traduit le fait que $\pi$ est \ird (lemme \ref{lemirreddim1}).
\end{proof}

\rem \label{RemDimGrl} La dimension des \grls qui intervient dans la \dfn \ref{defdim1grl}
est la dimension de Krull du \trdiz, noté $\Zar G$, obtenu en quotientant le \trdi
$G^+\cup\so\infty$ par la relation d'\eqvc suivante: 

\smallskip 
\centerline{$\xi\sim \zeta \iff \exists n>0$ tel que $\xi\leq n\zeta$ et $\zeta\leq n\xi $.} 

\snii Autrement dit encore $\Zar G $ est l'ensemble des sous-groupes $\cC(\xi)$ \hbox{($\xi\in G^{+}$)}, auquel on rajoute un \elt $+\infty$, avec sa structure de \trdi naturelle. 
\\
Un \grl est \zed \ssi il est nul.
Un groupe totalement ordonné a pour dimension son rang défini de mani\`ere \gui{usuelle}. 
Les groupes totalement ordonnés~$\ZZ$ et~$\QQ$ sont de rang~$1$, un produit fini catégorique a pour rang le maximum des rangs des facteurs, un produit lexicographique a pour rang la somme des rangs. En \clama la dimension d'un \grl peut être définie 
comme la longueur maximum d'une cha\^{\i}ne de sous-groupes premiers
(ici, comme pour les \ideps d'un anneau commutatif, la \hbox{cha\^{\i}ne $H_0\subsetneq H_1\subsetneq H_2$} est dite de longueur 2, mais
notez que la cha\^{\i}ne maximale dans $\ZZ$ est $0\subsetneq \ZZ$).
\\
Pour un \adv  $\gV$ la dimension de $\Div\gV$ (son rang en tant que groupe totalement ordonné) est égale à la dimension valuative de $\gV$, qui est égale à sa dimension de Krull. 
Ceci s'étend aux \ddps mais pas aux anneaux de Krull:
un anneau~\hbox{$\KXn$} (où $\gK$ est un \cdiz) est un anneau de Krull
de \ddk $n$ (égale à sa dimension valuative), alors 
que son groupe de \dvrs reste de dimension~$1$. 
\eoe

\subsec{\Prts de \dcn \gnlesz}
Nous donnons 
quelques \dfns \covs liées aux \prts de \dcn (voir \cite[Chap. XI]{ACMC}).

\begin{definition} \label{defisgrls} \label{defigrldec}
 \textsl{(Quelques \prts de \dcn dans les \grlsz)}
\begin{enumerate}
\item Une famille $(a_i)_{i\in I}$  d'\eltsz~$>0$ dans un \grl \textsl{admet une \dcnpz}
si l'on peut trouver une famille finie $(p_j)_{j\in J}$ d'\eltsz~$>0$ deux à deux \orts telle que chaque $a_i$ s'écrive $\sum_{j\in J}r_{ij}p_j$
avec les $r_{ij}\in\NN$.
La famille $(p_j)_{j\in J}$ est alors appelée une \textsl{\bdpz} pour la famille~$(a_i)_{i\in I}$.
\item Un \grl  est dit \textsl{à \dcnpz}
 s'il est discret et si toute famille finie d'\elts $>0$ admet une \dcnpz.
\item Un \grl est dit \textsl{à \dcnbz} lorsque pour \hbox{tout $x\geq 0$} il existe un entier $n$ tel que, lorsque $x=\sum_{j=1}^nx_j$ avec les $x_j\geq 0$, au moins l'un des $x_j$ est nul.
\item Le \grl à \dcnb $G$ est dit \textsl{absolument borné} s'il existe un entier $n$ tel que dans toute famille de $n$ \elts deux à deux \orts dans $G^{+}$, il y a un \elt nul.
\item Un \grl  est  dit \textsl{à \dcncz} s'il est discret et si tout 
\eltz~$>0$ est une somme d'\elts \irdsz. 
\end{enumerate}
\end{definition}

Un exemple classique de \grl à \dcnp mais pas à \dcnb est le groupe des diviseurs de l'anneau de tous les entiers \agqs complexes (qui est un anneau de Bézout de dimension~$1$).
 
Nous donnons maintenant quelques résultats \cofs de base  (voir \cite{ACMC} pour la plupart d'entre eux, notamment pour le point \textsl{\iref{i1propgrldec}}, qui est donné par le \thoz~\hbox{XI-2.16} de~\cite{ACMC}).
 
\begin{proposition} \label{propgrldec} \emph{(Quelques relations liant les \prts de \dcn dans les \grlsz)}
\begin{enumerate}
\item \label{i3propgrldec} Un \grl  à \dcnc est à \dcnbz. 
\item \label{i1propgrldec} Un \grl discret et à \dcnb est à \dcnpz.
\item \label{i9propgrldec} Un \grl à \dcn partielle est de dimension $\leq 1$.
\item \label{i4propgrldec} Pour un \grl discret non nul $G$ \propeq
\begin{enumerate}
\item $G$ est à \dcncz.
\item $G$ est à \dcnb et il poss\`ede un test  d'irré\-ducti\-bilité\footnote{\`A la question \gui{$\pi$ est-il \irdz?}, le test doit donner l'une des deux réponses suivantes: 
\begin{itemize}
\item \gui{oui}, ou 
\item  \gui{non et voici $\pi_1$, $\pi_2>0$ tels que $\pi=\pi_1+\pi_2$}.
\end{itemize}} pour les \eltsz~$>0$.
\item $G$ est à \dcnb et tout \elt $>0$ est minoré par un \elt \irdz.
\item  $G$ est isomorphe comme \grl à un groupe $\ZZ^{(I)}$
pour un ensemble discret~$I$. On peut prendre pour $I$ l'ensemble des \elts \irds de $G$.
\end{enumerate}
\end{enumerate}
\end{proposition}

\begin{proof}
Vu les résultats présentés dans \cite{ACMC}, seul
 le point   \textsl{\iref{i9propgrldec}} réclame une \demz.

\snii \textsl{\iref{i9propgrldec}.} On consid\`ere une \bdp $(\pi_1,\dots,\pi_k)$ pour 
le couple $(\xi,\zeta)$.
On écrit $\xi=\sum_{i\in I}n_i\pi_i$ avec des $n_i>0$. On note $J$ la partie complémentaire de $I$ dans $\lrbk$, 

\snic{\zeta=\sum_im_i\pi_i=\zeta_1+\zeta_2\hbox{  avec  }\zeta_1=\sum_{i\in I}m_i\pi_i\hbox{  et }\zeta_2=\sum_{i\in J}m_i\pi_i.}

\noindent 
On a bien $\zeta_2\perp \xi$, et $\zeta_1\leq m\xi$ si $m$ majore les $m_i/n_i$ \hbox{pour $i\in I$}. 
\end{proof}

\begin{lemma} \label{lemGrldcc}
Soit $G$ un \grl à \dcnc et $I_G$ l'ensemble de ses \elts \irdsz.
\begin{enumerate}
\item Pour tout sous-groupe solide détachable $H$ de $G$, on a $G=H\boxplus H\epr$.
\item L'application $H\mapsto I_H = I_G \cap H$ établit une bijection entre l'ensemble des sous-groupes solides détachables de $G$  
et l'ensemble des parties détachables de $I_G$.
\end{enumerate}
\end{lemma}
%
\begin{proof}
On démontre que si un \ird $\pi$ n'est pas dans $I_H$, il est dans $H\epr$.
En effet, si $\alpha\in H^{+}$, l'\elt $\alpha\vi\pi$, qui est égal à $\pi$ ou $0$, est \ncrt nul. Le reste suit.
\end{proof}
\rem Dans le lemme précédent, l'\egt $G=H\boxplus H\epr$ ne peut pas être démontrée \cot si on remplace l'hypoth\`ese \gui{$G$ à \dcncz} par \gui{$G$ à \dcnbz}.
Ceci crée quelques subtilités  \algqs que l'on retrouvera par la suite.
\eoe

\begin{lemma} \label{lemGrldcp}
Soit $(\alpha_i)_{i\in I}$ une famille qui admet une \bdp dans un \grlz~$G$: $\alpha_i=\sum_{j\in J}r_{ij}\pi_j$ pour des $\pi_j$ deux à deux \ortsz.
\\
Alors on a: 
$\Vi_{i}\alpha_i=\som_{j\in J}\inf_{i\in I}(r_{ij})\,\pi_j$
et $\Vu_{i}\alpha_i=\som_{j\in J}\sup_{i\in I}(r_{ij})\,\pi_j.$ 
\end{lemma}

Nous utilisons maintenant le principe \aref{XI-2.10} cité \paref{Prqgrl} 
pour obtenir des résultats qui s'av\`ereront utiles pour  
le groupe des \dvrs d'un \advl en raison de l'in\egtz~\hbox{$\dv(a+b)\geq \dv(a)\vi\dv(b)$}.

\begin{lemma} \label{lemGrl01}
Soient $\alpha_1$, \dots, $\alpha_n$ dans un \grl $G$. 
\\
On pose $\gamma_i=\Vi_{j\neq i}\alpha_j$
et l'on suppose que $\gamma_i\leq \alpha_i$ pour tout $i\in\lrbn$. 
\\
Alors on a: $$\Vi\nolimits_{i\in\lrbn}\big(\Vi\nolimits_{j\in\lrbn, j\neq i}\big(\abs{\alpha_i-\alpha_j}+\som_{k\notin \so{i,j}}(\alpha_i-\alpha_k)^{+}\big)\big)=0.$$ 
\end{lemma}
%
\begin{proof}
Il suffit de le démontrer lorsque les $\alpha_i$ sont totalement ordonnés. Par exemple si~$\alpha_i$ et~$\alpha_j$ sont plus petits que tous les autres, les in\egtsz~\hbox{$\gamma_i\leq \alpha_i$} et $\gamma_j\leq \alpha_j$ 
donnent~\hbox{$\alpha_j\leq \alpha_i$} \hbox{et $\alpha_i\leq \alpha_j$}, donc $\alpha_j=\alpha_i$.
\end{proof}

L'in\egt $\dv(a+b)\geq \dv(a)\vi \dv(b)$ peut se lire de mani\`ere
\smq en disant que \hbox{si $a_1+a_2+a_3=0$} alors chaque $\dv(a_i)$
majore la \bif des deux autres. Ceci se \gns par \recu comme suit.
\begin{fact} \label{factDivSomme}
Dans un \advl on a l'implication  
$$
\som_{j=1}^{n}a_j=0\;\Rightarrow\;\Vi\nolimits_{i,j:j> i}\abS{\dv(a_j)-\dv(a_i)}=0.
$$
\end{fact}
%
\begin{proof} 
Pour chaque $j$, on a $-a_j\in\geN{a_i\mid i\neq j},$ donc $\dv(a_j)\geq \Vi_{i\neq j}\dv(a_i)$. On conclut avec le lemme \iref{lemGrl01}. 
\end{proof}
%

\begin{lemma} \label{lemGrl03}
Soient un \grl $G$, et $\beta_1$, \dots, $\beta_p$, $\xi$, $\zeta$ dans $G^{+}$. 
\\
On suppose que $\Vi_{j\in\lrbp}\abs{\xi-\beta_j}=0=\Vi_{j\in\lrbp}\abs{\zeta-\beta_j}$.
\begin{enumerate}
\item On a $\xi\leq \Vu_{j\in\lrbp}\beta_j$.
\item On suppose que pour chaque  $ j$, $k\in \lrbp$, $\beta_j\in\cC(\beta_k)\boxplus \cC(\beta_k)\epr$. 
\\
Alors,  pour $\gamma$, $\delta\in\so{\xi, \zeta, \beta_1, \dots, \beta_p}$, on a $\gamma\in\cC(\delta)\boxplus \cC(\delta)\epr$, i.e. $\cC(\gamma)\subseteq \cC(\delta)\boxplus \cC(\delta)\epr$.
\end{enumerate}
\end{lemma}
%
\begin{proof} Pour le point \textsl{1} il faut démontrer une in\egtz, et pour le point \textsl{2} on veut une \egt $\gamma+m\delta=\gamma+(m+1)\delta$. Il suffit donc (principe \aref{XI-2.10}) de faire la \dem lorsque les $\abs{\xi-\beta_j}$ sont totalement ordonnés ainsi que les $\abs{\zeta-\beta_j}$. 
\\
Dans ce cas, il y a un indice~$h$ pour lequel  \fbox{$\xi=\beta_h$}, et donc $\xi\leq \Vu_{j\in\lrbp}\beta_j$.  
\\
Ceci donne le point \textsl{1.} Voyons le point \textsl{2.}
\\
On a aussi un indice~$\ell$ pour lequel  \fbox{$\zeta=\beta_\ell$}.
\\
D'apr\`es le point \textsl{4} du lemme \iref{lemGrl02} il y a un entier $m$ tel que 

\snic{m\beta_j\vi\beta_k=(1+m)\beta_j\vi\beta_k\hbox{  pour tous les  }j, \,k.}

\noindent  Ceci reste vrai en rempla{\c c}ant $\beta_j$ et/ou $\beta_k$
par $\xi$ ou $\zeta$. 
\end{proof}


\subsubsection*{\Prts de \dcn pour un \grl quotient}

\begin{proposition} \label{propQuoGrlDcn}
Soit $G$ un \grl et $\pi:G\to G'$ un \moquo par un sous-groupe solide $H$.
\begin{enumerate}
\item \label{i1propQuoGrlDcn} Si $G$ est de dimension~$\leq 1$, $G'$ l'est \egmtz.
\item \label{i3propQuoGrlDcn} Si $G$ est à \dcnbz , $G'$ l'est \egmtz.
\item \label{i4propQuoGrlDcn} Si $G$ est à \dcnp et si $G'$ est discret, $G'$ est à \dcnpz.
\item \label{i5propQuoGrlDcn} Si $G$ est à \dcnc et si $G'$ est discret, $G'$ est à \dcncz.
\item \label{i6propQuoGrlDcn} Si $G$ est discret à \dcnb et si $H=\alpha\epr$ pour un $\alpha>0$, alors $G'$ est \isoc à $\cC(\alpha)$ et  absolument borné.
\end{enumerate} 
\end{proposition}
%
\begin{proof}
\textsl{\ref{i1propQuoGrlDcn}.} On doit montrer que dans le quotient, pour tous $x$, $y$
on a un entier $n$ tel que $y\vi nx=y\vi(n+1)x$. Or c'est déjà vrai avant de passer au quotient.

\snii\textsl{\ref{i3propQuoGrlDcn}.} On consid\`ere un \elt $x\in G^{+}$. On suppose que si $x=\sum_{i=1}^{n}x_i$ avec \hbox{des $x_i\in G^{+}$}, alors l'un des $x_i$ est nul.
\\
Supposons que l'on ait $\pi(x)=\sum_{i=1}^{n}\pi(y_i)$ avec des $\pi(y_i)\geq 0$ dans $G'$. Comme $\pi(y_i)=\pi(y_i^{+})$, on peut supposer les $y_i\geq 0$. On a $x=u+\sum_{i=1}^{n}y_i$ avec $u\in H$. On écrit $u=u^{+}-u^{-}$, on remplace $y_1$ par $u^{+}+y_1$.\\
On a maintenant $x=\sum_{i=1}^{n}y_i-u^{-}$.
 On a donc $u^{-}\leq \sum_{i=1}^{n}y_i$
et par le \tho de Riesz \aref{XI-2.11~\textsl{1}}, on écrit $u^{-}= \sum_{i=1}^{n}u_i$ avec $0\leq u_i\leq y_i$. Enfin on remplace chaque $y_i$ par $z_i=y_i-u_i$ et l'on a 
$x=\sum_{i=1}^{n}z_i$ avec \hbox{des $z_i\geq 0$}, donc l'un des~$z_i$ est nul,
et pour cet indice~$i$, on \hbox{a $\pi(y_i)=\pi(z_i)=0$}. 

\snii
\textsl{\ref{i4propQuoGrlDcn}.} On consid\`ere une famille finie $\big(\pi(x_i)\big)_{i\in I}$
on calcule une \bdf de la famille $( x_i )_{i\in I}$. On supprime les \elts de cette base qui deviennent nuls dans le quotient.  Les \elts restants sont $>0$ dans le quotient, et deux à deux \ortsz.

\snii
\textsl{\ref{i5propQuoGrlDcn}.} Pour $\pi(x)>0$ dans $G'$, on consid\`ere une \fcn compl\`ete
de~$x$ dans~$G$. Il suffit de vérifer la \prt suivante: si $p$ est \irdz, alors~$\pi(p)$
est nul ou \irdz. Supposons que $\pi(p)=\pi(q)+\pi(r)$,
\hbox{tous $\geq 0$} dans~$G'$. On peut supposer $q=q^{+}$ et $r=r^{+}$. 
\\
On écrit
$p+u^{-}=q+r+u^{+}$ avec~$u\in H$. Le \tho de Riesz \aref{XI-2.11~\textsl{2}}
nous donne des $p_i$ et $v_i\geq 0$ satisfaisant 

\snic{p=p_1+p_2+p_3$, $u^{-}=v_1+v_2+v_3$, $p_1+v_1=q$, $p_2+v_2=r$ et $p_3+v_3=u^{+},}

\noindent 
d'où~\hbox{$\pi(q)=0$} si $p_1=0$ et $\pi(r)=0$ si $p_2=0$.

\snii
\textsl{\ref{i6propQuoGrlDcn}.} En effet $G$ est de dimension $\leq 1$ 
donc $G'\simeq\cC(\alpha)$, qui est absolument borné.
\end{proof}

Par contre la \prt pour un \grl d'être discret ne passe pas toujours au quotient. Fort heureusement on a le lemme \ref{factQuodiscret}.

\section{Propriétés de stabilité pour les \advlsz} \label{secAnnDivl}
\subsec{Localisations d'un \advlz, 2}

\begin{theorem} \label{lemdvlloc}
Soient $\gA$ un \advlz, $S$ un filtre ne contenant pas~$0$, et~$H_S$ le sous-groupe solide de $\DivA$ engendré par les~$\dv_\gA(s)$ \hbox{pour~$s\in S$}.
On a les \prts suivantes.
\begin{enumerate}
\item L'anneau $S^{-1}\gA=\gA_S$ est un \advl et 
 il y a un unique morphisme  de \grls 
$\varphi_S:\DivA\to\Div \gA_S$
tel 
{que $\varphi_S\big(\dv_\gA(a)\big)=\dv_{\gA_S}(a)$} pour tout $a\in\Atl$. 
Ce morphisme est surjectif, donc $\Div \gA_S\simeq (\DivA)/\Ker\varphi_S$. 
\item  On a $H_S^{+}=\sotq{\alpha\in\DivA}{\exists s\in S,\,0\leq \alpha\leq \dvA(s)}$,  et les \dips dans~$H_S^{+}$ sont les \elts de $\dvA(S)$.
\item Supposons que $\DivA=H_S\boxplus H'$. Alors le morphisme $\varphi_S$ est un \moquo par $H_S$: il permet  d'identifier   
$\Div\gA_S$  au \grl quotient $(\DivA)/H_S\simeq H'$.
\end{enumerate}
 
\end{theorem}
%
\begin{proof}  
\textsl{1.}  C'est le \thref{lemdvlloc0}.

\snii \textsl{2.} En notant $H_1=\sotq{\alpha\in\DivA}{\exists s\in S,\,0\leq \alpha\leq \dvA(s)}$, on vérifie facilement que $H_1+H_1\subseteq H_1$
et que $H_1-H_1$ est un sous-groupe solide dont la partie positive est égale à $H_1$. 
 
\snii \textsl{3.} Il est clair que $\varphi_S$ est surjectif et que $H_S\subseteq \Ker\varphi_S$. On doit montrer l'inclusion réciproque.
Un \elt du noyau s'écrit $\alpha=\dvA(\ak)$ pour une suite $(\ak)$ de \profz~\hbox{$\geq 2$} dans $\gA_S$.
On peut supposer que les $a_i$ sont dans $\gA$.
Si $\alpha=\alpha_1+\alpha_2$ avec $\alpha_1\in H_S$ et~$\alpha_2\in H'$, on a $\varphi_S(\alpha)=\varphi_S(\alpha_2)$. On peut donc supposer $\alpha\perp H_S$, \cad que la suite $(\ak,s)$ est de \prof $\geq 2$ dans $\gA$ pour tout $s\in S$. On est donc ramené à montrer que si $(\ak)$ est de \profz~$\geq 2$ dans~$\gA_S$ et $(\ak,s)$ est de \prof $\geq 2$ dans $\gA$ pour \hbox{tout $s\in S$}, alors~$(\ak)$ est de \prof $\geq 2$ dans $\gA$. Soit donc une suite $(\ck)$
proportionnelle à $(\ak)$. Il existe $c\in\Atl$ et $s\in S$ tels que~$sc_i=c a_i$ pour $i\in\lrbk$. Ceci implique que les suites $(\ak,s)$ et $(\ck,c)$ sont proportionnelles, donc il existe $d\in\Atl$ tel que
$(\ck,c)=d(\ak,s)$. En particulier $(\ck)=d(\ak)$, ce qui termine la \demz. 
\end{proof}

\subsec{Anneaux avec groupe des \dvrs de dimension $1$}

\begin{lemma} \label{lemBezRad}
Soit $\gA$ un domaine de Bézout avec $\DivA$ de dimension $\leq 1$.
Pour un $a\in\Atl$ \propeq
\begin{enumerate}
\item $a\in\Rad\gA$.
\item $\DivA=\cC\big(\!\dvA(a)\big)$.
\end{enumerate}
\end{lemma}
\begin{proof}
\textsl{1} $\Rightarrow$ \textsl{2.}
L'anneau est de Bézout donc tous les \dvrs sont principaux. Pour un $b\in\Atl$ arbitraire,
puisque $\DivA$ est de dimension $\leq 1$, par le point \textsl{2} 
du lemme~\ref{lemGrl02}, on écrit $b=b_1b_2$ avec $b_1$ et $b_2\in\Atl$, $b_1\mid a^{n}$ et $\dvA(a,b_2)=0$.
Puisque l'anneau est de Bézout,  $1\in\gen{a,b_2}$, et puisque 
$a\in\Rad\gA$, $b_2\in\Ati$. Donc $\dvA(b)=\dvA(b_1)\in\cC(a)$.

\snii \textsl{2} $\Rightarrow$ \textsl{1.} Pour tout $b\in\Atl$ on a un $n\in\NN$
tel que $b\mid a^{n}$. Si $b=1+ax$, on a $\gen{b,a}=\gen{1}$, \hbox{donc $\gen{b}=\gen{b,a^{n}}=\gen{1}$}.
\end{proof}
%

\begin{proposition} \label{propgrldec2} 
Soit $\gA$ un \advl  
de \ddkz~\hbox{$\leq 1$}. Alors $\DivA$ est de dimension $\leq 1$.
\end{proposition}

%
\begin{proof} Les \advls de dimension $\leq 1$ sont des \ddps (\thref{thi2clcohidv}). On utilise alors la \prt de \fcn des \itfs fid\`eles
dans un \ddp de dimension $\leq 1$ donnée par  
\aref{Theorem XII-7.2}.
\\
On pourrait aussi argumenter \gui{plus directement} en montrant que la \ddk d'un \ddp
est égale à la dimension de son groupe des \dvrsz. 
\end{proof}

La proposition qui suit est un corolaire du \thref{lemdvlloc}.

\begin{proposition} \label{corlemdvlloc}
Soit $\gA$ un \advl et $S$ le filtre engendré par \hbox{un $s\in\Atl$} (i.e. l'ensemble des $x$ qui divisent une puissance de $s$). D'apr\`es le \thref{lemdvlloc}, l'anneau $\gA_S$ est aussi un \advlz.
\begin{enumerate}
\item Si $\DivA$ est discret de dimension $1$,
il en va de même pour $\Div\gA_S$.
\item Si en outre $\DivA$ est à \dcnp ou à \dcnb ou à \dcncz,  
il en va de même pour $\gA_S$.
\item Si $\gA$ est \noco \fdiz, alors $\gA_S$
est \egmt \noco \fdiz.
\end{enumerate}
\end{proposition}
\begin{proof}  On a $H_S=\cC(\dv(s))=\sotq{\xi\in\DivA}{\exists n\in\NN,\,\abs\xi\leq n\dv(s)}$.

\snii \textsl{1} et \textsl{2.} 
On conclut avec le lemme \ref{factQuodiscret} et la proposition \ref{propQuoGrlDcn}.

\snii \textsl{3.} Résultat classique (cf.  \aref{Local-global Principle XII-7.13}).
\end{proof}
%

\begin{theorem} \label{lemdvlloc2} ~\\ 
Soient $\gA$ un \advl avec $\DivA$ discret de dimension~$1$, 
 $\alpha$ un \dvr $>0$ 
 \hbox{et $S_{\alpha}=\sotq{x\in\Atl}{\dv (x)\perp \alpha}$}.
Il est clair que $S_\alpha$ est un filtre détachable, et le \thref{lemdvlloc} s'applique. On a les \prts suivantes.  
\begin{enumerate}
\item \label{i1lemdvlloc2} L'anneau $S_\alpha^{-1}\gA$ est un \advl avec $\Div(S_\alpha^{-1}\gA)\simeq\cC(\alpha)$.\\
En particulier $S_\alpha^{-1}\gA$ est un \adv discr\`ete \ssi $\cC(\alpha)$ est isomorphe à $(\ZZ,\geq)$.   
\item \label{i2lemdvlloc2}
\Propeq
\begin{enumerate}
\item  \label{i2alemdvlloc2} $\alpha$ est un \dvr \irdz.
\item  \label{i2Alemdvlloc2} $\cC(\alpha)=\ZZ \alpha$. 
\item  \label{i2blemdvlloc2} $\Idv(\alpha)$ est un \idepz.
\item  \label{i2clemdvlloc2} $S_\alpha$ est un filtre premier de hauteur $\leq 1$
et $\gA=S_\alpha\,\cup\,\Idv(\alpha)$ (union disjointe de deux parties détachables).
\item  \label{i2dlemdvlloc2} $S_\alpha^{-1}\gA$ est un \adv discr\`ete et si $p/1$ est une uniformisante, on a $\alpha=\dvA(p) \mod \cC(\alpha)\epr$.
\end{enumerate}
\item \label{i3lemdvlloc2}  Si en outre l'anneau $\gA$ est \cohz, les quatre ensembles suivants
sont égaux (on rappelle, \thref{lemdivirdadvlgnl}, que $\alpha\mapsto \Idv(\alpha)$ établit une bijection entre les \dvrs \irds et les \idifs premiers $\neq \gen{1}$).
\begin{itemize}
\item Les \idifs premiers  $\neq \gen{1}$. 
\item Les \itfs premiers $\fq\neq \gen{0}$ tels que $\dv(\fq)>0$. 
\item Les \itfs premiers $\fq\neq \gen{0},\gen{1}$ tels que $\fq=\Idv(\fq)$. 
\item Les \itfs premiers détachables 
de hauteur $1$. 
\end{itemize} 
\end{enumerate} 
\end{theorem}
%
\begin{proof}
\textsl{\ref{i1lemdvlloc2}.} On a $\DivA=\cC(\alpha)\boxplus\cC(\alpha)\epr$. On applique donc le point \textsl{3} du \thref{lemdvlloc}. On a (mêmes notations) $H_{S_\alpha}=\cC(\alpha)\epr$, donc $\Div(S_\alpha^{-1}\gA)\simeq\cC(\alpha)$.\\
Pour la derni\`ere affirmation on applique le lemme~\ref{lemVALD}.

\snii
\textsl{\ref{i2lemdvlloc2}.} Notons d'abord que puisque $\alpha>0$, $\DivA$ n'est pas nul, et $\gA$ n'est pas un corps. Par ailleurs les hypoth\`eses
impliquent que $\gA$ est  \dveez, et que 
$S_\alpha$ et $\Idv(\alpha)$ sont deux parties détachables disjointes de $\gA$.

\snii
\textsl{\ref{i2alemdvlloc2}} $\Leftrightarrow$ \textsl{\ref{i2Alemdvlloc2}.} D'apr\`es le lemme \ref{lemirreddim1}.

\snii
\textsl{\ref{i2alemdvlloc2}} $\Leftrightarrow$ \textsl{\ref{i2blemdvlloc2}.} 
D'apr\`es le \thref{lemdivirdadvlgnl}.

\snii
\textsl{\ref{i2clemdvlloc2}} $\Rightarrow$ \textsl{\ref{i2blemdvlloc2}}, et
\textsl{\ref{i2dlemdvlloc2}} $\Rightarrow$ \textsl{\ref{i2blemdvlloc2}.}
Clair.

\snii
\textsl{\ref{i2alemdvlloc2}} $\Rightarrow$ \textsl{\ref{i2clemdvlloc2}} et
\textsl{\ref{i2dlemdvlloc2}}.
D'apr\`es le point \textsl{\ref{i1lemdvlloc2}}
et puisque $\cC(\alpha)=\ZZ\alpha\simeq \Div(S_\alpha^{-1}\gA)$, l'anneau $S_\alpha^{-1}\gA$ est un \adv discr\`ete, et $\alpha$ vu dans $S_\alpha^{-1}\gA$ est égal 
à $\dv(\fraC p1)$ si $p$ engendre le radical. Il reste à voir que le complémentaire de~$S_\alpha$ est bien l'\id $\Idv(\alpha)$ (qui est alors premier de hauteur 1 par \dfnz). 
Tout \dvrz~\hbox{$\xi\geq 0$} s'écrit $n_\xi\alpha+\rho$ avec $n_\xi\in\NN$ et $\rho\perp\alpha$ (lemme~\ref{lemirreddim1}). Par \dfn $\Idv(\alpha)=\sotq{x\in\gA}{\xi=\dvA(x)\geq \alpha}$. Dans la \dcn précédente, cela signifie que $n_\xi>0$, tandis \hbox{que 
$x\in S_\alpha$} signifie~\hbox{$n_\xi=0$}. On a donc bien deux parties \cops de $\gA$. 

\snii
\textsl{\ref{i3lemdvlloc2}.}
Les trois premiers ensembles sont  égaux 
d'apr\`es le corolaire \ref{corcorlemdivirdadvlgnl}. 
Il reste à montrer que $\fp$ est un \itf premier détachable de hauteur $1$ \ssi on a $\dv(\fp)>0$.
\\
L'implication $\Leftarrow $ est donnée par \textsl{\ref{i2blemdvlloc2}} $\Rightarrow$ \textsl{\ref{i2clemdvlloc2}.} 
\\
L'implication $\Rightarrow $ est donnée par le lemme \ref{cor0lemdvlloc}.
\end{proof}
%

\subsec{Stabilité pour les anneaux de \polsz}

\fbox{Dans ce paragraphe, $\gA$ est un \advlz}. 

\smallskip Pour~\hbox{$p\in\AuX=\AXn$}, on note~$\rc(p)$ pour $\rc_{\gA,\uX}(p)$.

\begin{lemma} \label{lemth2AnnDivl}
Une fraction $p/q$ dans $\gK(\uX)$ (avec $p$ et $q\in\AuX$) est dans~$\AuX$ \ssi  $p/q\in\KuX$ et $\dv(\rc(q))\leq \dv(\rc(p))$.  Autrement dit, pour $p$, $q\in\AuX$, $q$ divise $p$ dans $\AuX$ \ssi $q$ divise $p$ dans $\KuX$ et $\dv\big(\rc(q)\big)\leq \dv\big(\rc(p)\big)$. 
\end{lemma}
%
\begin{proof} La condition est évidemment \ncrz. Montrons qu'elle est suffisante. Puisque~$q$ divise $p$ dans $\KuX$ on écrit $qr=ap$ avec $a\in\Atl$
et $r\in\AuX$. On veut montrer que $r/a\in\AuX$. Or $\dv\big(\rc(r)\big)\geq \dv(a)$, car 

\snic{\dv\big(\rc(q)\big)+\dv\big(\rc(r)\big)=\dv(a)+\dv\big(\rc(p)\big).}

\noindent  Et puisque $a\in\Atl$
cela signifie que $a$ divise tous les \coes de $r$.
\end{proof}

Le \tho suivant est une version \cov non \noee du \tho qui affirme que si $\gA$ est un anneau de Krull, il en est de même pour~$\AX$ (voir le \thref{thKrAX} et \cite[\tho 12.4]{Mat}).
\begin{theorem} \label{th2AnnDivl}
Soit $\gA$ un \advl de corps de fractions $\gK$. L'anneau~$\AuX$ est \egmt
\dvlaz. En outre on a un \iso naturel de \grls 
\[ 
\begin{array}{rcl} 
\Div(\AuX)  & \simarrow  & \Div(\KuX)\times \DivA ,\;\hbox{avec} \\[.3em] 
\dv_{\AuX}(f)  & \longmapsto  &  \big(\dv_{\KuX}(f),\dvA(\rc(f)\big)\;\;\hbox{pour }(f\in\AuX).  \end{array}
\]
\end{theorem}
%
\begin{proof} ~\\ 
Le groupe de \dve de $\AuX$ est $G=\gK(\uX)\eti/\Ati$, celui de $\KuX$ est
$H=\gK(\uX)\eti/\Ktl$. \\
Puisque $\KuX$ est un anneau à pgcd, $\Div(\KuX)$
est simplement le groupe~$H$, en passant en notation additive. 
\\
Pour un  $f\in\gK(\uX)\eti$ nous notons $f_G$ sa
classe dans $G$ et  $f_H$ sa classe dans~$H$. 
 Nous notons~$\preceq$ la relation de \dve dans ces groupes.

\noindent 
On a un morphisme de groupes ordonnés 
${\varphi:G\to G'=H\times \DivA}$
donné par 

\snic{\varphi(f_G)=\left(f_{H},\dv\big(\rc(p)\big)-\dv\big(\rc(q)\big)\right)
\qquad \hbox {où $f = p/q$ avec  $p$, $q\in\AuX$}.}

\noindent 
D'apr\`es le lemme \iref{lemth2AnnDivl}, ce morphisme est un \iso sur son image,
ce qui permet d'identifier~$G$ à un sous-groupe ordonné de $G'$.
\\
On va montrer que $\varphi$ satsifait les requêtes du projet divisoriel 2.
Il nous reste à voir que tout \elt $\Delta=(f_{H},\delta)$ du \grlz~$G'$ est \bif d'une famille finie dans $G$. 
On peut supposer $\Delta\geq 0$. Cela signifie que $\delta\geq 0$
et $f_{H}=u_{H}$ pour un~$u$ dans~$\AuX$. 
Cela donne~\hbox{$\Delta=(u_{H},\delta)$}.
\\
Pour $b\in\Atl$, on a $\dv\big(\rc(bu)\big)=\dv(b)+\dv\big(\rc(u)\big)$. 
\\
Soient $b$ tel
que $\dv(b)\geq \delta-\dv\big(\rc(u)\big)$ et  $w=bu\in\AuX$, alors
$\varphi(w_{G})=(u_{H},\delta_1)$ \hbox{avec $\delta_1\geq \delta$}.
\\
L'\id $\rc(u)$ de $\gA$ admet un \ivda qui peut être écrit sous la 
forme~$\rc(v)$ pour \hbox{un $v\in\AuX$}. On a donc

\snic{\dv\big(\rc(uv)\big)=\dv\big(\rc(u)\big)+\dv\big(\rc(v)\big)=\dv(a)}

\noindent  pour un $a\in\Atl$ (voir le
lemme \iref{corprop2Idv}).  En outre $\delta=\dv\big(\rc(r)\big)$ avec~\hbox{$r\in\AuX$}.
\\  
On pose $w'=uvr/a$, on a
$\varphi(w'_G)=\big((uvr)_{H},\delta\big)$, \hbox{et $u_{H}\preceq (uvr)_{H}$}.

\noindent Finalement $\Delta=(u_{H},\delta)=\varphi(w_G)\vi\varphi(w'_G)$ dans $G'$.
\end{proof}

Un corolaire \imd est le \tho suivant.
\begin{theorem} \label{th4AnnDivl} \label{thKrAX}
Soit $\gA$ un \advlz.
 \label{i1th4AnnDivl} Si le \grl  $\DivA$ est discret, ou de dimension $1$, ou   à \dcnpz,
ou a \dcnbz ,  
il en va de même pour $\Div(\AX)$.
\end{theorem}
%
\begin{proof} Ceci résulte de ce que $\Div(\AX)\simeq\DivA\times \Div(\KX)$ (\thref{th2AnnDivl})
et de ce \hbox{que $\Div(\KX)$} est discret et à \dcnb (à fortiori de dimension $1$ et à \dcnpz).
\end{proof}
%

\subsec{Stabilité pour les extensions enti\`eres intégralement closes}

\begin{lemma} \label{lemth1AnnDivl}
Soit $\gA$ un \acl de corps de fractions~$\gK$ \hbox{et $\gL\supseteq \gK$} un \cdiz. Soit $\gB$ la \cli de $\gA$ dans $\gL$.
Soient $a$, $a'$,   $a_1$, \dots, $a_n$ dans $\gA$.
\begin{enumerate}
\item  L'\elt $a$ divise $a'$ dans $\gA$ \ssi il divise $a'$ 
dans~$\gB$.
\item L'\elt $a$ est un pgcd fort de $(\an)$
dans $\gA$ \ssi il l'est dans~$\gB$.
%
%
\end{enumerate}
\end{lemma}
%
\begin{proof} \textsl{1.} Supposons que $x=a'/a\in\gK$ soit dans $\gB$, on doit montrer qu'il est dans~$\gA$. Cela résulte de ce que
tout \elt de $\gB$ est entier \hbox{sur $\gA$: $x\in\gK\cap\gB$} est entier sur~$\gA$
donc dans $\gA$. 

\snii \textsl{2.}
Vu le point \textsl{1}, 
l'affirmation dans $\gB$ est plus forte que celle dans $\gA$.
Comme la notion de pgcd fort est stable par multiplication par un \elt de
$\Atl$, on peut supposer que $a=1$ est pgcd fort des $a_i$
 dans $\gA$ et on doit montrer qu'il l'est dans $\gB$.
Ainsi on doit montrer que \hbox{si  $b\in\gB$} divise $ya_1,\dots,ya_n$ dans $\gB$
($y\in\gB$),
alors l'\elt $z=y/b$ de $\gL$ est  dans $\gB$.
Pour cela il suffit de montrer que $z$ est entier sur $\gA$.
Par hypoth\`ese $za_i\in\gB$ pour chaque $i$.\\
%
%
On dispose de \pols unitaires $g_i\in\AX$ qui annulent les $za_i$, et fournissent autant de \pols de $\KX$ qui annulent $z$: $f_i(z)=g_i(za_i)/a_i^{m_i}$.
On peut calculer dans $\KX$ le pgcd unitaire $f$ de ces derniers \polsz

\snic{f(X)=X^m+\sum_{k<m}c_{m-k}X^k.}

\noindent  Alors pour chaque $i$,
le \pol $a_i^mf(X/a_i)$ divise  $g_i$ dans $\KX$, donc ses \coes $c_1a_i$, $c_2a_i^{2}$, \dots, $c_ma_i^{m}$ sont dans $\gA$ par le \tho de Kronecker, parce que $\gA$ est \iclz. Par exemple \hbox{les $c_2a_i^2$} sont dans~$\gA$. Les~$a_i^{2}$ sont 
de pgcd fort $1$ donc $c_2\in\gA$. De même chaque $c_k$ est dans $\gA$.
Et l'on a bien $z$ entier sur~$\gA$. 
\end{proof}

\begin{definota} \label{notaPolKro} 
Pour toute liste  $(\an)=(\ua)$ dans $\gA$, on note $K_{(\ua)}(T)$
le \polz~\hbox{$\sum_{j=1}^{n}a_jT^{j-1}$}. On dit que c'est \textsl{le \pol de Kronecker associé à la liste ordonnée~$(\ua)$}.
\\
 Si $(\ub)=(\bbm)$ est une autre liste on définit $(\ua)\star(\ub)$
 par l'\egt  $K_{(\ua)\star(\ub)}=K_{(\ua)}\,K_{(\ub)}$.
 \\
Si $\gA$ est un \advlz, le corolaire \ref{corprop2Idv} implique que 

\snic{\dvA(\ua)+\dvA(\ub)=\dvA\big((\ua)\star(\ub)\big).}
\end{definota}

\begin{lemma} \label{lem2th1AnnDivl}
Soit $\gA$ un \advl de corps de fractions~$\gK$ \hbox{et $\gL\supseteq \gK$} un \cdiz. Soit $\gB$ la \cli de $\gA$ dans $\gL$.
 \\ 
Pour toute liste   $(\ub)=(\bn)$ dans $\Btl$ il existe une liste $(\ub')=(b'_1,\dots,b'_m)$ telle que la liste $(\ub)\star(\ub')$ soit dans $\gA$ et admette un pgcd fort dans $\Atl$.
\end{lemma}
%
\begin{proof}
On consid\`ere l'anneau $\gB_1=\gA[\bn]$ qui est \hbox{un \Amoz} \tfz. En fait $\gB_1$ est un quotient d'un anneau $\gC$ (non \ncrt int\`egre)
qui est une \Alg libre de rang fini\footnote{On peut prendre $\gC=\aqo\AXn{h_1(X_1),\dots,h_n(X_n)}$ où $h_i\in\gA[X_i]$ est \untz. Notons que le \pol $N_B$
est un \pol \ndz de $\AT$ ($N_B(0)=\rN_{\gC/\gA}(b_1)$).}. 
On consid\`ere le \polz~\hbox{$B=K_{(\ub)}=\sum_{k=1}^nb_kT^{k-1}$}
vu dans $\CT$, puis l'\elt cotransposé~\hbox{$\wi B\in \CT$}.
\\
Soit $N_B=\rN_{\CT/\AT}(B)\in\AT$. On a $N_B=B\wi B$, et 
en revenant de~$\gC$ à~$\gB_1$, cela nous donne un $C\in\gB_1[T]$ avec $BC=N_B$. Si $C=K_{(\uc)}$ \hbox{et $N_B=K_{(\ud)}$} on a $(\ub)\star(\uc)=(\ud)$.
\\
La liste $(\ud)$ de $\gA$  admet une \ivde ${(\aq)}={(\ua)}$ dans~$\gA$. Ainsi  $(\ud)\star{(\ua)}$ admet un pgcd fort~$g$
dans $\Atl$.  
Le point \textsl{2} du lemme \iref{lemth1AnnDivl} nous dit
que $(\ud)\star{(\ua)}$ admet le pgcd fort~$g$
dans~$\gB$.
Finalement on obtient que la liste  ${(\ub)\star(\uc) \star(\ua)}$ 
est dans $\AT$ et qu'elle admet $g$ pour pgcd fort (dans~$\Atl$ comme dans $\Btl$). 
Ainsi, 
la liste $(\ub')=(\uc) \star(\ua)$ satisfait les requêtes voulues. 
\end{proof}
\begin{theorem} \label{th1AnnDivl}
Soit $\gA$ un \advl de corps de fractions~$\gK$ \hbox{et $\gL\supseteq \gK$} un \cdiz. Soit $\gB$ la \cli de $\gA$ dans $\gL$. 
\begin{enumerate}
\item L'anneau $\gB$ est un \advlz.
\item On a un unique morphisme de \grls $\varphi:\DivA\to\DivB$ 
tel  que 

\snic{\varphi\big(\dvA(a)\big)=\dvB(a)$  pour $a\in\gA.}

\noindent 
Ce morphisme est injectif: cela permet d'identifier $\DivA$ à un sous-\grl de~$\DivB$.
\item 
Soit $x\in\gB$ et $f$ un \pol \unt   de $\AX$ (de degré $d$) qui annule $x$. Avec $\xi=\dvB(x)$ et $D$ le ppcm des entiers $\in\lrb{1..d}$, on a des \elts $\gamma_k\in\DivA\subseteq \DivB$ dans le sous-groupe engendré par les \dvrs des \coes
de $f$, tels que $\Vi_k\abs{D\xi-\gamma_k}=0$.   
\end{enumerate}
\end{theorem}
%
\begin{proof} \textsl{1.}
Toute liste dans $\Btl$ admet une \ivde d'apr\`es 
le lemme~\ref{lem2th1AnnDivl}. 

\snii \textsl{2.}  L'unicité si existence est claire car l'\elt $\dvA(\an)=\Vi_i\dvA(a_i)$ de~$\DivAp$ doit avoir pour image $\Vi_i\dvB(a_i)=\dvB(\an)$.
\\
Pour l'existence, montrons d'abord que l'on peut définir une application $\varphi$
\hbox{de $\DivAp$} dans~$\DivBp$  en posant 

\snic{\varphi(\delta)=\dv_\gB(\an)$ si $\delta=\dv_\gA(\an)$ pour des $a_i$ dans $\Atl.}

\noindent Le lemme
\iref{lemth1AnnDivl} nous dit que $\delta=0$ implique $\varphi(\delta)=0$. 
\\
Supposons que $\dv_\gA(\an)=\dv_\gA(a'_1,\dots,a'_m)$. Il existe donc une liste~$(\ux)$ dans~$\gA$ telle que les familles $(a_ix_j)_{i,j}$ \hbox{et
$(a'_ix_j)_{i,j}$} admettent un même pgcd fort $g$ dans~$\gA$. 
Le lemme~\iref{lemth1AnnDivl} nous dit que les familles $(a_ix_j)_{i,j}$ \hbox{et
$(a'_ix_j)_{i,j}$} admettent aussi le pgcd fort~$g$ dans~$\gB$,
donc 

\snic{\dv_\gB(\ua)+\dv_\gB(\ux)=\dvB(g)=\dv_\gB(\ua')+\dv_\gB(\ux),}

\noindent d'où $\dv_\gB(\ua)=\dv_\gB(\ua')$. Ceci montre que $\varphi$ est bien définie. 
\\
Montrons que $\varphi$ est injective.
Si $\dvB(\an)=\dvB(a'_1,\dots,a'_m)$, on consid\`ere une liste~$(\ux)$ dans~$\gA$ telle que la liste $(a_ix_j)_{i,j}$  admette un  pgcd fort~$g$ dans~$\gA$. 
C'est aussi un pgcd fort dans $\gB$, \hbox{donc $\dvB(\ua)+\dvB(\ux)=\dvB(g)$}.
\hbox{Donc $\dvB(\ua')+\dvB(\ux)=\dvB(g)$}, ce qui signifie que la liste $(a'_ix_j)_{ij}$ admet le pgcd fort $g$ dans $\gB$. 
Comme c'est une liste dans~$\gA$, elle admet aussi $g$ comme pgcd fort dans $\gA$. Donc
$\dvA(\ua')+\dvA(\ux)=\dvA(g)$. \hbox{Et $\dvA(\ua)=\dvA(\ua')$.}
%
\\
L'application $\varphi$ 
que l'on vient de définir est un morphisme injectif de
\mos positifs, de~$\DivAp$ \hbox{dans $\DivBp$} et s'étend de mani\`ere unique en un morphisme de \grlsz, de~$\DivA$ \hbox{dans $\DivB$} 
(vérification laissée \alecz).

\snii\textsl{3.} Soit $f(X)=\sum_{j=1}^{d}a_jX^{j}\in\AX$  \hbox{avec $a_d=1$} et $f(x)=0$.
On consid\`ere les $\alpha_j=\dvB(a_jx^{j})$ pour les $a_j\in\Atl$. 
Le fait \iref{factDivSomme}
nous dit que 

\snic{\Vi_{j> k,\,a_j,a_k\in\Atl}\abs{\dvB(a_jx^{j})-\dvB(a_kx^{k})}=0.}

\noindent  Or 

\snic{\dvB(a_jx^{j})-\dvB(a_k\xi^{k})=(j-k)\dvB(x)-\big(\dvB(a_k)-\dvB(a_j)\big)}

\noindent  et $\dvB(x)\geq 0$. 
On a donc
$\Vi_{j> k}\abs{D\xi-\alpha_{jk}^{+}}=0$ pour des $\alpha_{jk}^{+}\in\DivAp$. 
\end{proof}
%

\begin{theorem} \label{th3AnnDivl}
Soit $\gA$ un \advl de corps de fractions~$\gK$ \hbox{et $\gL\supseteq \gK$} un \cdiz. Soit $\gB$ la \cli de $\gA$ dans $\gL$
(qui est un \advl d'apr\`es le \tho précédent). 
\begin{enumerate}
\item \label{i2th3AnnDivl} Si $\gL$  admet une base  sur $\gK$ et si $\DivA$ est discret (i.e. $\gA$ est \dveez), alors~$\DivB$ est discret\footnote{On ne suppose pas ici que la base de $\gL$ sur $\gK$ soit une partie finie de $\gL$. L'hypothèse est seulement que~$\gL$ est libre sur $\gK$. Dans \cite{MRR}, les auteurs montrent comment construire un \Amo librement engendré par un ensemble $I$ arbitraire, non \ncrt discret. Ce module libre $F$ possède une base qui est une famille $(x_i)_{i\in I}$ dans $F$, et l'ensemble des $x_i$ est en bijection avec $I$ si l'anneau $\gA$ est non trivial. Dans la situation présente où $\gK\subseteq \gL$, comme $\gL$ est discret, l'ensemble $I$ est lui-même discret et l'\evc $\gL$ est isomorphe à $\gK^{(I)}$. Il s'ensuit que tout calcul fini dans $\gL$ se déroule en fait dans un \Kev qui possède une base finie.}.
\item \label{i1th3AnnDivl} Si $\DivA$ est de dimension $\leq 1$, il en va de même pour~$\DivB$.
\end{enumerate}
\end{theorem}
%
\begin{proof}
 \textsl{\ref{i2th3AnnDivl}.} On doit tester $z\in\gB$ pour un $z\in\gL$. Puisque
$\gL$ admet une base sur $\gK$, on peut calculer le \polmin $f(X)$ de 
$z$ sur $\gK$. Et par le \tho de Kronecker, $z$ est zéro d'un \pol \unt de
$\AX$ \ssi son \polmin sur $\gK$ est dans~$\AX$. 

\snii\textsl{\ref{i1th3AnnDivl}.} Le point \textsl{3} du \thref{th1AnnDivl} donne
$\Vi_{j}\abs{D\xi-\gamma_j}=0$ pour des $\gamma_j\in\DivAp$. 
On a un résultat du même type pour~$\zeta=\dv(z)$. 
On aura donc une double \egt

\snic{\Vi_{h}\abs{D\xi-\beta_{h}^{+}}=0=\Vi_{h}\abs{M\zeta-\beta_{h}^{+}}}

\noindent  où les $\beta_{h}\in\DivA$. Par le lemme \iref{lemGrl03}
on obtient $\cC(D\xi)\subseteq \cC(M\zeta)\boxplus \cC(M\zeta)\epr$.
\\
Enfin $\cC(D\xi)=\cC(\xi)$ et $\cC(M\zeta)=\cC(\zeta)$. 
\end{proof}

%

\subsec{Autres \prts de stabilité}

La proposition suivante est une sorte de réciproque du \thref{th1AnnDivl}
dans un cas particulier.

\begin{proposition} \label{propFixAutAdvl}
Soit $\gB$ un \advlz, $\Gamma$ un groupe fini d'\autos de~$\gB$
\hbox{et $\gA = \gB^\Gamma$} le sous-anneau des points fixes de $\Gamma$.
Alors $\gA$ est un \advlz.
\end{proposition}
\begin{proof}
Tout d'abord, $\gA$ est \icl parce que $\gB$
est \iclz. Par ailleurs $\gB$ est entier sur $\gA$ parce que $\Gamma$
est fini. Donc $\gB$ est la \cli de~$\gA$ dans~$\Frac(\gB)$, ce qui nous am\`ene, à la fin de la \demz, dans la situation du \thref{th1AnnDivl}.
\\
On considére une liste finie dans $\gA$ pour laquelle on cherche une
\ivdez. Cette liste donne les \coes d'un \pol $f\in\AX$.
Comme la liste admet une \ivde dans $\gB$, il existe $g$, $h\in\BX$ et $d\in\Btl$ tels que
$$ f\,g=d\,h \hbox{ et }\Gr_\gB(\rc(h))\geq 2. \eqno (*)
$$
En transformant $(*)$ par les $\sigma\in\Gamma$ et en faisant le produit des
\egts obtenues on a une \egt 
$$\ndsp
f^{N}\,G= D\,H \hbox{ avec } N=\abs\Gamma,\,G=\prod_{\sigma\in\Gamma} \sigma(g),\,\hbox{ etc.}
$$
On écrit ceci sous la forme
$$
f\,(f^{N-1} \,G)= D\,H \hbox{ où } G,\,H\in\AX \hbox{ et } D\in\gA\eqno(\#)
$$
On a $\Gr_\gB(\rc(H)) \geq  2$ car $\Gr_\gB(\rc(h))\geq 2$. On applique alors
le point \textsl{2} du lemme \ref{lemth1AnnDivl} 
et on obtient   $\Gr_\gA(\rc(H))\geq 2$.
Ainsi $(\#)$ fournit un \ivda de l'\id $\rc(f)$ dans $\gA$.
\end{proof}
\exl Soit $\gk$ un \cdi de \cara $\neq 2$, $\gB=\kuX$ l'anneau des \pols en $n$ \idtrs ($n\geq 2$) et $\gA$ le sous-anneau des \pols pairs. Alors $\gA$ est un \advlz, en tant qu'égal à $\gB^{\Gamma}$, \hbox{avec $\Gamma=\gen{\sigma}$},  où $\sigma$ échange $X_i$ et $-X_i$ pour chaque $i$. 
On a par exemple dans~$\gA$ un \dvr \ird non principal $\dvA(gX_1,gX_2)$ pour chaque \pol \ird impair $g$ dans $\gB$ ($\dvB(gX_1,gX_2)=\dvB(g)$).
\eoe

\section{Anneaux de Krull}
\label{secAKrull}

\subsec{Définition et premi\`eres \prtsz}

\begin{definition} \label{defiAdKrull}
On appelle \textsl{anneau de Krull} un \advl non trivial dont le groupe des \dvrs
est discret et à \dcnbz.\index{anneau a diviseurs!de Krull}\index{Krull!anneau de ---}%
\end{definition}

Un \cdi est un \aKr dont le groupe des \dvrs est nul. Ce sont les autres \aKrs qui nous intéressent,
ceux pour lesquels existent des \dvrs strictement positifs.

\begin{definition} \label{defiAdvlfac}
Un \advl  est dit \textsl{à \dcncz} (resp. \textsl{à \dcnpz})
si son groupe des \dvrs est à \dcnc (resp. à \dcnpz).
\index{anneau a diviseurs!a dec@à \dcn compl\`ete}%
\index{decomposition@décomposition!anneau a diviseurs a dec@\advl à --- compl\`ete}%
\index{anneau a diviseurs!a dec@à \dcn partielle}%
\index{decomposition@décomposition!anneau a diviseurs a dec@\advl à --- partielle}%
\index{anneau a diviseurs!a dec@à \dcnbz }%
\index{decomposition@décomposition!anneau a diviseurs a dec@\advl à --- bornée}%
\end{definition}

\exls 1) Si $\gk=\ZZ$ ou un \cdiz, $\kXn$ est un anneau à pgcd \dvee
et l'on montre facilement que c'est un \aKrz. Il est à \dcnc lorsque $\gk=\ZZ$ et pour certains \cdisz, comme $\QQ$ et ses extensions finies, ou les \cacz.

\snii 2)  Un \adv discr\`ete est évidemment un \aKr local de dimension $1$.
Pour une réciproque voir le lemme \ref{lemValDiscKrull}.

\snii  3) Les exemples de base d'\aKrs à \dcnc sont les anneaux factoriels et les \doksz, à condition qu'ils soient à \fatz. Voir aussi le \thref{thLasNoetDiv}. 
\eoe

\medskip Le fait qui suit rassemble des conséquences de la proposition
 \ref{propgrldec} concernant les \grls lorsqu'on l'applique au groupe des \dvrsz.
Le point \textsl{3} nous donne une version de \cite[\tho 1, \textsection 1.19]{Edw}, que nous reprenons ici dans un cadre \cof  plus \gnlz. C'est un outil tr\`es utile qui remplace souvent de mani\`ere efficace la \prt de \dcncz.

\begin{fact} \label{fact1Krull} ~
\begin{enumerate}
\item Un \advl à \dcnc est un anneau de Krull. 
\item En \clamaz, les deux notions sont \eqves car tout \grl discret à \dcnb est alors à \dcncz
\footnote{On notera que pour ce point la \dem classique utilise le tiers exclu (pour le test d'irréductibilité d'un \dvrz) mais pas l'axiome du choix.}. 
\item Un \aKr est à \dcnpz. 
\item Le groupe des \dvrs  d'un \aKr $\gA$ est de dimension $\leq 1$: pour 
tout $\alpha\in\DivA$ on a $\DivA=\cC(\alpha)\boxplus\cC(\alpha)\epr$.
\end{enumerate}
\end{fact}
 
Le \tho suivant   nous donne un exemple paradigmatique  d'\aKrz.
\begin{theorem} \label{factKruladvdec}~
\begin{enumerate}
\item \label{i1factKruladvdec} Un anneau \gmq \icl  est un anneau de Krull.  
\item \label{i2factKruladvdec} La \cli d'un anneau \gmq int\`egre dans son corps de fractions est un anneau de Krull.  
\end{enumerate}
\end{theorem}
\begin{proof}
\textsl{\ref{i1factKruladvdec}.} Cas particulier de \textsl{\ref{i2factKruladvdec}.}

\snii \textsl{\ref{i2factKruladvdec}.} 
Avec un changement de variables on obtient une mise en position de \Noe
qui fait appara\^{\i}tre l'anneau \gmq $\gA$ comme une extension finie 
d'un anneau de \polsz~\hbox{$\gC=\gk[X_1,\dots,X_r]$}. L'anneau $\gC$ est un anneau à pgcd à \fabz, donc un \aKrz. 
Si $\gA$ est int\`egre, $\Frac\gA$ est une extension finie \hbox{de $\gk (X_1,\dots,X_r)$}. On  conclut alors par
le \thref{th5AnnDivl}.
\end{proof}

 Concernant le cas crucial des \advs discr\`etes, la situation
en \coma est un peu plus délicate qu'en \clama comme l'indique le lemme suivant, qui compl\`ete le lemme~\ref{lemVALD}.
En \clama tout \aKr est à \dcnc et les cinq points sont \eqvsz.

\begin{lemma} \label{lemValDiscKrull} \textsl{(Anneaux de valuation discr\`ete, 2)} \\
Soit  $\gA$ un \advl
 et un $\alpha>0$ dans $\DivA$ (par exemple $\alpha=\dv(a)$ avec $a\in\Atl\setminus \Ati$). 
 Considérons les \prts suivantes. %
\begin{enumerate}
\item \label{i1lemVDK} $\gA$ est un \adv discr\`ete.
\item \label{i3lemVDK} $\gA$ est un \aKr local de dimension $1$. 
\item \label{i4lemVDK} $\gA$ est un \aKr local  et $\DivA=\cC(\alpha)$. 
\item \label{i5lemVDK} $\gA$ est un anneau principal local à \fab et  $\DivA=\cC(\alpha)$. 
\item \label{i6lemVDK} $\DivA$ est discret et  $\idg{\DivA:\ZZ\alpha}\leq k$ pour un $k\geq 0$. 
\end{enumerate}
On a l'implication  
{\ref{i1lemVDK}} $\Rightarrow$ {\ref{i3lemVDK}}, et les \eqvcs {\ref{i3lemVDK}} $\Leftrightarrow$ {\ref{i4lemVDK}} $\Leftrightarrow$ {\ref{i5lemVDK}} $\Leftrightarrow$ {\ref{i6lemVDK}.} 
 \\
 L'implication {\ref{i3lemVDK}} $\Rightarrow$ {\ref{i1lemVDK}} est valable  si $\gA$ est à \dcncz. 
\end{lemma}
\begin{proof} 
Dans chacun des points on a~$\DivA$ discret, i.e. $\gA$  \dveez.

\snii \textsl{\ref{i5lemVDK}} $\Rightarrow$ \textsl{\ref{i4lemVDK}},
\textsl{\ref{i1lemVDK}} $\Rightarrow$ \textsl{\ref{i6lemVDK}},   
et \textsl{\ref{i1lemVDK}} $\Rightarrow$ \textsl{\ref{i3lemVDK}.} 
Clair.

\snii \textsl{\ref{i3lemVDK}} $\Rightarrow$ \textsl{\ref{i4lemVDK}}, \textsl{\ref{i5lemVDK}} et \textsl{\ref{i6lemVDK}.} En tant qu'\advl local de dimension~\hbox{$1$, $\gA$} est un \adv (\thref{thi2clcohidv}).  Pour \hbox{un $\xi\in\DivAp$} arbitraire, on consid\`ere une \bdp $(\pi_1,\dots,\pi_r)$ pour $(\alpha,\xi)$. Les $\pi_i$ sont deux à deux \ortsz, et dans un \adv deux \dvrs $>0$ sont toujours comparables. Donc un et un seul des~$\pi_i$, par exemple~$\pi_1$, est~\hbox{$>0$}. On a donc
$\alpha=\ell \pi_1$ pour un $\ell\geq 1$, $\xi=m\pi_1$ pour \hbox{un $m\geq 0$},  \hbox{d'où $\ZZ\pi_1\subseteq \DivA\subseteq \QQ \alpha$}.
Enfin si $k$ majore le nombre d'\elts non nuls dans une écriture de~$\alpha$
comme somme d'\elts $\geq 0$, on aura \ncrt $\idg{\DivA:\ZZ\alpha}\leq k$
\hbox{et  $\DivA\subseteq \fraC 1 {k!}\ZZ \alpha$}.  

\snii \textsl{\ref{i4lemVDK}} $\Rightarrow$ \textsl{\ref{i3lemVDK}.}  
On note que $\cC(\alpha)$ est absolument borné. Le point~\textsl{\ref{i1lemVDK}} du \thref{lemKrullfini} nous dit que~$\gA$ est un anneau principal, donc  de dimension $\leq 1$. Il s'agit en fait d'un \adv et comme ce n'est pas un corps, il est de dimension exactement $1$. 

\snii \textsl{\ref{i6lemVDK}} $\Rightarrow$ \textsl{\ref{i5lemVDK}.} De l'in\egt  $\idg{\DivA:\ZZ\alpha}\leq k$ on déduit que $\ZZ\alpha\subseteq \DivA\subseteq \fraC 1{k!}\ZZ\alpha$. Donc $\DivA$ est à \dcnb et absolument borné, donc $\gA$ est un \aKr principal (\thref{lemKrullfini}). Il reste à montrer que $\gA$ est un \aloz.
Supposons que $x+y$ est \ivz. Les \dvrsz~$\dv x$ et~$\dv y$ s'expriment sous forme
$m\pi$ et $n\pi$ pour un $\pi\in\DivA$. 
\\
Comme $0=\dv(x+y)\geq \dv(x)\vi\dv(y)$, on a bien $\dv(x)$ ou  $\dv(y)$ nul.

\snii \textsl{\ref{i5lemVDK}} $\Rightarrow$ \textsl{\ref{i1lemVDK}} (lorsque $\gA$ est à \dcncz): si $\alpha$ est minoré par un \dvr \irdz~$\pi$, il est clair que $\DivA=\ZZ \pi$. 
\end{proof}
%

\subsect{Théor\`eme d'approximation simultanée et conséquences}{Théor\`eme d'approximation simultanée}

 Le point \textsl{2} du \tho suivant  nous donne  une version de \cite[\tho 2, \textsection 1.20]{Edw} dans un cadre \cof  plus \gnlz. 

\begin{theorem} \label{thKrullApproxSim} \emph{(\Tho d'approximation simultanée)} 
Soit $\gA$ un anneau de Krull. 
\begin{enumerate}
\item Soient $(\pi_1,\dots,\pi_r)$ 
des \dvrs $>0$ deux à deux \orts et $(n_1,\dots,n_k)$ dans $\NN$. Notons~\hbox{$\pi=\sum_i\pi_i$}  et $\alpha=\sum_i n_i\pi_i$. Il existe $a\in\Atl$ tel que $\dvA(a)=\alpha+\rho$ avec $\rho\geq 0$ et $\rho\perp \pi$ (à fortiori $\rho\perp \alpha$). 
\item Pour tout $\alpha\in(\DivA)^+$ et tout $\gamma\geq \alpha$ on peut trouver $a\in\Atl$ tel que $\dv a=\alpha+\rho$
avec $\rho\geq 0$ et $\rho\perp \gamma$.
\end{enumerate}
\end{theorem}
%
\begin{proof}
\textsl{1.} \textsl{On suppose en un premier temps les $\pi_j$ \irdsz}.
\\
Pour chaque $i\in\lrbk$ on va trouver un $x_i$ tel que

\snic{\dvA(x_i)=n_i \alpha_i + \beta_i$,  $\beta_i= \sum_{j\neq i} m_j\pi_j+\rho_i$ avec $m_j\geq n_j+1$ et $\rho_i\perp\pi  \quad (*)_i.}

\noindent Construisons par exemple $x_1$. Nous considérons $\gamma_1=(n_1-1)\pi_1+\pi$. C'est un \dvr  que l'on écrit $\dvA(c_1,\dots,c_m)$ pour des $c_i\in\Atl$.
On consid\`ere une \bdp pour $(\pi_1,\dots,\pi_r,c_1,\dots,c_m)$.
\\
Comme elle ne raffine aucun des $\pi_i$, pour l'un des $c_j$, qui est le $x_1$ recherché, la condition~$(*)_1$ est réalisée.
\\
Une fois les $x_i$ construits on consid\`ere $x=\sum_ix_i$.
On calcule une \bdp pour $(x,\xr)$. Comme elle ne raffine pas les 
$\pi_i$ on peut appliquer le lemme~\ref{lem5bdf}, et l'on obtient
que $\dvA(x)=\sum_{i=1}^{r}n_i\pi_i+\rho$ avec $\rho\perp \pi$.

\snii \textsl{Voyons le cas \gnlz.} On reprend le calcul précédent.
Si à une étape du calcul un ou plusieurs des $\pi_j$ se décomposent, on
remplace $(\pi_1,\dots,\pi_r)$ par la liste raffinée et on reprend tous les calculs depuis le début. Cet inconvénient ne peut se produire qu'un nombre fini de fois. \`A la fin le calcul se déroule comme si les $\pi_i$ étaient \irdsz. 

\snii\textsl{2.} On consid\`ere une \bdp $(\pi_1,\dots,\pi_r)$ pour $(\alpha,\gamma)$. 
\\
On écrit $\alpha=\sum_{i\in\lrbr}n_i\pi_i$ et on applique le point \textsl{1.}
\end{proof}
%

\begin{lemma} \label{lem5bdf} Soit $\gA$ un \advl avec $\DivA$  discret.
{ \\ } 
On suppose donnés  $(\xn)$ dans~$\Atl$, avec $\sum_{i}x_i=0$,
 et $(\pi_1,\dots,\pi_r)$ une \bdp pour $(\xin)$ (\hbox{où $\xi_i=\dvA(x_i)$}).
On écrit $\xi_i=\sum_{i=j}^{r}n_{ij}\pi_j$.
Alors pour chaque $j\in\lrbr$, la valeur minimum de $n_{ij}$ est atteinte au moins deux fois. 
\end{lemma}
\begin{proof}
Puisque $\sum_{i}x_i=0$ on a pour chaque $i\in\lrbn$ $\xi_i\geq \Vi_{k\neq i}\xi_k$, ce qui donne $n_{ij}\geq \Vi_{k\neq i}n_{kj}$ pour chaque $j\in\lrbr$.
En particulier si $n_{ij}$ est la plus petite valeur (pour ce~$j$ fixé),
cet entier doit être égal à l'un des $n_{kj}$ pour $k\neq i$.
\end{proof}

On obtient maintenant des corolaires importants du \thref{thKrullApproxSim}.
\begin{theorem} \label{thKrull1,5} \emph{(\Tho un et demi pour les anneaux de Krull)}\\
Soit $\gA$ un anneau de Krull et $\alpha\in\DivAp$. 
\begin{enumerate}
\item Pour tout $a\in\Atl$ tel que $\alpha\leq \dv(a)$ il existe $b\in\Atl$
tel que 

\snic{\alpha=\dv(a)\vi \dv(b)=\dv(a^n)\vi \dv(b)$ pour tout $n\in\NN\etl.}
\item Pour tout $c\in\Ktl$ tel que $\dv(c)\leq \alpha$ il existe $d\in\Ktl$
tel que $ \alpha=\dv(c)\vu\dv(d)$.
\end{enumerate}  
\end{theorem}
%
\begin{proof}
\textsl{1.} On pose  $\gamma=\dv(a)$. 
Le \thoz~\iref{thKrullApproxSim} nous donne un
$b\in\gA$ tel que $\dv(b)=\alpha+\rho$ avec $\rho\geq 0$ et $\rho\perp\gamma$.
Par suite, pour $n\in\NN\etl$ 

\snic{\dv(b)\vi \dv(a^n)=(\alpha+\rho)\vi n\gamma=(\alpha\vu\rho)\vi n\gamma=
(\alpha\vi n\gamma)\vu(\rho\vi n\gamma)=\alpha\vu0=\alpha.}

\snii\textsl{2.} Soit $a$ avec $\dv(a)\geq \alpha$ puis $\beta=\dv(a)-\alpha$.
On a $\dv(c)\leq \alpha$, donc $0\leq \beta\leq \dv(\fraC a c)$ \hbox{avec $\fraC a c\in\Atl$}.  Le point \textsl{1.} nous donne \hbox{un $e\in\Atl$} tel que $\beta=\dv(\fraC a c,e)$, donc 

\snic{\alpha=\dv(a)-\beta=\dv (a)+\big(\dv(\fraC c a)\vu \dv(\fraC 1 e)\big)=\dv(c)\vu\dv(\fraC a e),}

\noindent d'où le résultat avec $d=\fraC a e$.
\end{proof}
%
\begin{corollary} \label{corthKrull1,5}
Soit $\gA$ un anneau de Krull, et $a$, $b\in\Atl$. Il existe $c$ \hbox{et $d\in \Atl$} tels que 
$$ \frac a b = \frac c d 
\;\hbox{ et }\;\dvA(a,b,c,d)=0.$$ 
D'où l'on déduit (en choissant de mettre $a$ en valeur):
\begin{enumerate}
\item [$\bullet$] $\gen{a,b}\gen{a,c}=a\gen{a,b,c,d}$,
\item [$\bullet$] $\dvA(a,b)+\dvA(a,c)=\dvA(a)$,
\item [$\bullet$] $\Idv(a,b)=(a:c)_\gA\,$ et $\,\Idv(a,c)=(a:b)_\gA$.
\end{enumerate}
En particulier, pour tout \dvr $\alpha\geq 0$, et tout $a\in\Atl$ tel que $\dv(a)\geq \alpha$, il existe  $c\in\Atl$ tel que $\Idv(\alpha)=(a:c)_\gA$.
\end{corollary}
\begin{proof}
Posons $\alpha=\dvA({a,b})$ et $\beta=\dv(a)-\alpha$%
(\footnote{Si $\gA$ est \coh on peut prendre $\beta=\dvA(\fb)$ où $\fb=(a:\gen{a,b})=(a:b)$}). 
En appliquant le \tho un et demi à $\beta$ et $a$, il existe $c\in\Atl$ tel que
$\beta=\dv(a,c)$. Puisque $\alpha+\beta=\dv(a)$ on obtient que $\gen{a,b}\gen{a,c}$ admet $a$ pour pgcd fort.
En particulier $bc$ est divisible par $a$, on écrit $bc=ad$ et l'on obtient
$\gen{a,b}\gen{a,c}=a\gen{a,b,c,d}$ donc $\dv(a,b,c,d)=0$.\\
Le dernier point résulte de ce que 

\snic{\dv(a,b)=\dv(a)-\dv(a,c)=\dv(1)\vu\dv(\fraC a c),}

\noindent  donc $\Idv(a,b)=\gA\cap\fraC a c\,\gA$. 
\end{proof}
\rem Si $\gA$ est un \ddp possédant la \prt un et demi (par exemple s'il est de dimension $\leq 1$), on a la même \prt sous une forme plus forte:  $\dv(a,b,c,d)=0$
est remplacé par $\gen{a,b,c,d}=\gen{1}$.
\eoe

\medskip Le  \thref{lemKrullfini} est une version \cov 
du \tho suivant en \clamaz: \textsl{un anneau de Krull avec seulement un nombre fini
de \dvrs \irds est un anneau principal} (conséquence de \cite[Theorem 12.2]{Mat}).

\begin{theorem} \label{lemKrullfini}
Soit $\gA$ un \aKrz. 
\begin{enumerate}
\item Si $\DivA$ est absolument borné, $\gA$ est un anneau principal.
\item Si $\DivA$ est engendré par les \dvrs \irds $(\pi_1,\dots,\pi_r)$, 
on obtient en outre 
\begin{itemize}
\item $\gA$ est un anneau principal à \fatz,
\item  si $\pi_i=\dvA(p_i)$, l'\elt $q=\prod_{i=1}^{r}p_i$ engendre l'\id $\Rad\gA$,
\item  les \idemas détachables de $\gA$ sont les $\gen{p_i}$.
\end{itemize}
\item  Pour un $a\in\Atl$ \propeq
\begin{enumerate}
\item $\DivA=\cC\big(\dvA(a)\big)$ (qui est absolument borné).
\item $\gA$ est un anneau principal et $a\in\Rad\gA$.
\end{enumerate}
\end{enumerate}
\end{theorem}
%
\begin{proof} \textsl{2.}  Tout d'abord, il est clair que $\DivA$ est à \dcncz. Comme un \dvr \ort à tous les $\pi_i$ est nul, pour tous $(n_1,\dots,n_r)$ dans~$\NN$, le \thref{thKrullApproxSim} nous donne un $a\in\Atl$ tel que $\dvA(a)=\sum_in_i\pi_i$. Ainsi tout \dvr est principal, autrement dit $\gA$ est un anneau à pgcd.

Montrons maintenant que~$\gA$ est un anneau de Bézout. Pour cela il suffit 
de montrer que si $b$ et $c$ ont pour pgcd $1$, ils sont \comz.
On écrit 

\snic{\dv(b)=\sum_{i\in I} m_i\pi_i$, $\dv(c)=\sum_{i\in J} m_i\pi_i$
 avec les $m_i>0$  et $I\cap J=\emptyset.}

\noindent Soit $K=\lrbr\setminus(I\cup J)$ et $d=\prod_{k\in K}p_k$. Alors $\dv(b+cd)=0$ par le lemme~\ref{lem5bdf} (dans la \dcn des trois \eltsz~$\dv(b)$, $\dv(cd)$ 
et~$\dv(b+cd)$, la valeur minimum $0$ doit être atteinte au moins deux fois
sur chaque composante~$\pi_i$). \hbox{Ainsi $b+cd\in\Ati$},
et $\gen{b,c}=1$.

Montrons que  $q\in\Rad\gA$. I.e., pour $x\in\Atl$, \hbox{$1+xq\in\Ati$} (ou encore $\dvA(1+xq)=0$). Dans la \dcn des trois \eltsz~$\dvA(1)$, $\dvA(1+xq)$ 
et~$\dvA(xq)$, la valeur minimum $0$ doit être atteinte au moins deux fois
sur chaque composante~$\pi_i$ (lemme~\ref{lem5bdf}).

Montrons que $\Rad\gA\subseteq \gen{q}$. Si $b\in\gA$
et si un des $p_i$ ne figure pas dans les diviseurs de $b$, on a $\gen{b,p_i}=1$ et donc $1+xb=yp_i$ pour $x$ et $y$ convenables, donc $1+xb$ n'est pas \ivz.

\noindent On laisse \alec le soin de montrer que les $\gen{p_i}$ sont les \idemas détachables.

\snii\textsl{1.} Les \dems dans le point \textsl{2} fonctionnent en rempla{\c c}ant les $\pi_i$ du point \textsl{2} par des \bdps pour les \elts de~$\gA$ qui définissent les \dvrs qui entrent dans la preuve. Les détails sont laissés \alecz.

\snii\textsl{3b} $\Rightarrow$ \textsl{3a.} D'apr\`es le lemme \ref{lemBezRad}.

\snii\textsl{3a} $\Rightarrow$ \textsl{3b.} L'anneau est principal d'apr\`es le point \textsl{1} Concernant $a\in\Rad\gA$ on reprend la preuve du point \textsl{2} en s'appuyant sur des \bdpsz.  
\end{proof}

\rems 1) Dans le point \textsl{3a} on ne peut pas remplacer l'hypoth\`ese
par la simple existence d'un \elt \ndz dans $\Rad\gA$:
il y a des anneaux factoriels locaux de dimension de Krull arbitraire. 

\noindent 2) En \clama tout \aKr $\gA$ avec~$\DivA$ absolument borné rel\`eve du point \textsl{2} ci-dessus. 
D'un point de vue \cofz, la situation est plus problématique: il est même impossible d'obtenir qu'un \aKr à \dcnc avec $\DivA$ absolument borné contient un \elt $a$ tel que
$\DivA=\cC(a)$. 
\eoe

\rdb
\medskip \exl \label{exempleFactnonCoh} 
On donne ici un anneau de Krull non \cohz. C'est
aussi un \alo à pgcd. Il n'est pas \coh (cela montre aussi qu'il n'est pas \noez). En \clama c'est un anneau factoriel en tant qu'anneau de Krull à pgcd. En outre il est de \ddk $\leq 2$. C'est l'exemple 5.2 dans \cite[Glaz]{Glaz01}.
\\
On consid\`ere le \cdi $\gF=\FF_2((a_i)_{i\in\NN},(b_i)_{i\in\NN})$, puis l'anneau local~\hbox{$\gB=\gF[x,y]_{1+\gen{x,y}}$}. On pose $p_i=a_ix+b_iy$ et l'on définit un \autoz~$\gamma $ de $\gB$ par 

\snic{\gamma (x)=x$, $\gamma (y)=y$, $\gamma (a_i)=a_i+yp_{i+1}$, $\gamma (b_i)=b_i+xp_{i+1}$
pour tout $i.}

\noindent 
On a $\gamma(p_i)=p_i$ et $\gamma $ engendre le groupe $G=\so{\Id,\gamma}$ d'ordre $2$. Enfin $\gA$ est le sous-anneau~$\gB^G$ des points fixes de $G$. 
\eoe

\subsec{Localisations d'un \aKrz, \dvrs \irdsz}

Le \tho suivant est une conséquence \imde du \thref{lemdvlloc}
et de la proposition~\ref{propQuoGrlDcn}.
On notera que les hypoth\`eses du point \textsl{2} sont toujours satisfaites
en \clamaz. 

\begin{theorem} \label{lemLocKrull}
Soient $\gA$ un \aKrz, $S$ un filtre ne contenant pas $0$, et~$H_S$ le sous-groupe solide de $\DivA$ engendré par les~$\dv_\gA(s)$ \hbox{pour~$s\in S$}.
Alors $\gA_S$ est un \advlz, $\Div\gA_S$ est à \dcnb et de dimension $\leq 1$, et  il y a un unique morphisme  de \grls 
$\varphi_S:\DivA\to\Div \gA_S$
tel 
{que $\varphi_S\big(\dv_\gA(a)\big)=\dv_{\gA_S}(a)$} pour tout $a\in\Atl$.
\\
On a en outre les précisions qui suivent.
\begin{enumerate}
\item Si $\gA_S$ est \dveez, $\gA_S$ est un \aKrz.  
\item Si $H_S$ est détachable et si $\DivA=H_S\boxplus H'$ pour un sous-groupe solide $H'$, $\gA_S$ est un \aKr et le morphisme $\varphi_S$ est un \moquo par~$H_S$: il permet  
d'identifier~$\Div\gA_S$  au \grl quotient $(\DivA)/H_S\simeq H'$.
\end{enumerate}
\end{theorem}

 Rappelons que pour un \advl arbitraire,  on a
une bijection naturelle entre  les ensembles suivants (\thref{lemdivirdadvlgnl}).
\begin{itemize}
\item Les \dvrs \irdsz. 
\item Les \idifs   premiers $\neq \gen{1}$. 
\end{itemize}
Les points \textsl{2} à \textsl{4} du \tho suivant rajoutent quelques précisions pour les \aKrs qui ne sont pas des corps (on suppose d'existence d'un diviseur strictement positif).

\begin{theorem} \label{lemdvlloc2Krull}  
Soient $\gA$ un \aKrz,  $\alpha$ un \dvr $>0$ 
{et $S_{\alpha}=\sotq{x\in\Atl}{\dv x\perp \alpha}$}.
\begin{enumerate}
\item \label{i1lemdvlloc2Krull} L'anneau $\gB=S_\alpha^{-1}\gA$ est un anneau principal avec
$\DivB\simeq\cC(\alpha)$. En particulier on a un \elt \ndz
dans $\Rad\gB$ et une borne à priori sur le nombre d'\elts deux à deux étrangers dans $\gB$.   
\item \label{i2lemdvlloc2Krull}
\Propeq
\begin{enumerate}
\item  $\alpha$ est un \dvr \irdz.
\item  $\cC(\alpha)=\ZZ \alpha$.
\item   $\Idv(\alpha)$ est un \idepz.
\item   $S_\alpha$ est un filtre premier de hauteur $\leq 1$
et $\gA=S_\alpha\,\cup\,\Idv(\alpha)$ (union disjointe de deux parties détachables).
\item   $S_\alpha^{-1}\gA$ est un \adv discr\`ete et si $p/1$ est une uniformisante, on a $\alpha=\dvA(p) \mod \cC(\alpha)\epr$.
\end{enumerate}
\item \label{i3lemdvlloc2Krull} Si $\gA$ est  à \dcncz, on obtient selon le point  \ref{i2lemdvlloc2}. des bijections entre  les trois ensembles suivants.
\begin{itemize}
\item Les \dvrs \irdsz. 
\item Les \idifs   premiers  $\neq \gen{1}$. 
\item Les filtres premiers détachables de hauteur $1$. 
\end{itemize}
\item \label{i4lemdvlloc2Krull} Si $\gA$ est \cohz, les quatre ensembles suivants
sont égaux.
\begin{itemize}
\item Les \idifs premiers  $\neq \gen{1}$. 
\item Les \itfs premiers $\fq\neq \gen{0}$ tels que $\dv(\fq)>0$. 
\item Les \itfs premiers $\fq\neq \gen{0},\gen{1}$ tels que $\fq=\Idv(\fq)$. 
\item Les \itfs premiers détachables 
de hauteur $1$. 
\end{itemize} 
\end{enumerate} 
\end{theorem}
%
\begin{proof}
%
\textsl{\ref{i1lemdvlloc2Krull}.} Résulte du point \textsl{1} du \thref{lemdvlloc2} et du \thref{lemKrullfini}.

\snii
\textsl{\ref{i2lemdvlloc2Krull}.} C'est le point \textsl{2} du \thref{lemdvlloc2}. 
Ici intervient seulement le fait que~$\DivA$ est discret de dimension~$1$.

\snii
\textsl{\ref{i3lemdvlloc2Krull}.} 
Il reste à vérifier qu'un filtre premier $S$ de hauteur~$1$ est de la forme~$S_\pi$ pour un \dvr \irdz~$\pi$.

\noindent Le localisé $\gA_S$ est par \dfn un anneau local de dimension $\leq 1$.
D'apr\`es le point~\textsl{1} du \thref{lemdvlloc}, c'est un \advl dont le groupe des \dvrsz~$\Div(\gA_S)$ est un quotient de $\DivA$.
En tant qu'\advl local de dimension~\hbox{$1$, $\gA_S$} est un \adv (\thref{thi2clcohidv}). 
\\
Soit $x\in\Atl\setminus S$, on a $\dv_{\gA_S}(x)>0$
car $x\notin\gA_S\eti$. On consid\`ere la \dcn de $\xi=\dvA(x)$ sous forme $\sum_i n_i\pi_i$ avec les $\pi_i$ \irds et les $n_i\in\NN\etl$.
Les $\pi_i$ restent deux à deux \orts dans $\gA_S$, et dans un \adv deux \dvrs $>0$ sont toujours comparables. Donc un et un seul des $\pi_i$, appelons le~$\pi$, reste~\hbox{$>0$} dans~$\Div(\gA_S)$. De la même mani\`ere tout \dvr \ird distinct de~$\pi$ dans $\DivA$ s'annule dans~$\Div(\gA_S)$.
Ceci montre que  $\Div(\gA_S)\simeq (\ZZ,\geq)$ avec $\pi$ comme seul \dvr \irdz, correspondant à $1\in\ZZ$ dans l'\isoz. Tous \hbox{les $y\in\gA$} tels que
$\dvA(y)\perp\pi$ dans~$\gA$ sont dans~$S$ car $\dv_{\gA_S}(y)=0$, i.e. ce sont des unités de $\gA_S$. Ainsi on obtient bien $S=S_\pi$. 

\snii
\textsl{\ref{i4lemdvlloc2Krull}.} 
C'est le point \textsl{\ref{i3lemdvlloc2}} du \thref{lemdvlloc2}. 
\end{proof}

On se propose maintenant d'étudier, autant que faire se peut, \gui{toutes} les \lons d'un \aKr à \dcncz.

\smallskip  Le \thref{lemdvllocDcnc} \gns pour les \aKrs à \dcnc des résultats
 simples dans le cas d'un anneau factoriel. Il résulte essentiellement
 du \thref{lemLocKrull} et  du lemme \ref{lemGrldcc}.
En \clama il donne une description exhaustive des localisés
d'un \aKrz. En \coma on se limite aux \lons en des filtres détachables particuliers. Ce \tho compl\`ete pour les \aKrs le lemme \ref{lemGrldcc} qui décrit les sous-groupes détachables d'un \grl à
\dcncz.


\begin{theorem} \label{lemdvllocDcnc} 
Soient $\gA$ un \aKr à \dcnc et $I$ l'ensemble de ses \dvrs \irdsz.
On reprend les notations du \thref{lemdvlloc}.
\begin{enumerate}
\item \label{i1lemdvllocDcnc} Si $S$ est un filtre et si $H_S$ est détachable,
alors $S$ est détachable, égal à 

\snic{\sotq{x\in\gA}{\forall \pi \in I,\,\pi\leq \dv(x)\Rightarrow \pi\in H_S}.}

\noindent  En outre $\Div(\gA_S)\simeq (\DivA)/H_S\simeq {H_S}\epr$.   
\item \label{i2lemdvllocDcnc} Si $H$ est un sous-groupe solide détachable
de $\DivA$, l'ensemble 

\snic{\sotq{x\in\gA}{\forall \pi \in I,\,\pi\leq \dv(x)\Rightarrow \pi\in H}}

\noindent  est un filtre détachable $S$ et $H_S=H$.
\item \label{i3lemdvllocDcnc} On obtient ainsi des bijections entre les trois ensembles suivants.
\begin{itemize}
\item Les filtres $S$ de $\gA$ tels que $H_S$ est détachable.
\item Les sous-groupes solides détachables de $\DivA$. 
\item Les parties détachables de $I$. 
\end{itemize}
\end{enumerate} 
On suppose dans la suite que $S$ est un filtre avec $H_S$ détachable.
\begin{enumerate} \setcounter{enumi}{3}
\item \label{i4lemdvllocDcnc} \Propeq
\begin{enumerate}
\item Le filtre $S$ est de hauteur $1$.
\item L'anneau $\gA_S$ est un \adk à \fatz.
\end{enumerate}
Dans ce cas \propeq
\begin{enumerate}\setcounter{enumii}{2}
\item Le sous-ensemble $I\cap{H_S}\epr$ de $I$ est fini non vide.
\item $\Rad(\gA_S)$ contient un \elt non nul.
\item L'anneau $\gA_S$ est principal et les \dvrs \irds de $\gA_S$ forment un ensemble fini non vide.
\end{enumerate}

\item \label{i5lemdvllocDcnc} Le filtre $S$ est premier de hauteur $1$ \ssi $I\cap{H_S}\epr$ est un singleton $\so\pi$. Dans ce cas l'\idep $\fp=\gA\setminus S$  est égal à $\Idv(\pi)$. 
\item \label{i6lemdvllocDcnc} Si  $S$ est premier de hauteur $\neq 0,1$,
$I\cap{H_S}\epr$ est infini et $\dv(\fp)=0$.
\end{enumerate}
\end{theorem}
%
\facile

\rem Dans le cadre des \aKrs à \dcncz, il ne semble pas que l'on puisse démontrer que~$H_S$ est détachable d\`es que~$S$ est un filtre détachable
(contrairement au cas des anneaux factoriels). De même on ne peut pas calculer en \gnl la hauteur d'un filtre $S$ sous la seule hypoth\`ese que~$H_S$ est détachable.
\eoe

\medskip Un anneau est dit \textsl{pleinement Lasker-Noether} 
lorsqu'il est \noe \coh \fdi et que tout idéal radical est intersection finie d'\ideps \tf (voir \cite{MRR}). Rappelons qu'en \clama tout anneau
\noe est pleinement Lasker-Noether, et donc, d'apr\`es le \thref{thLasNoetDiv}, tout anneau \noe \icl est un  \aKr à \dcncz.

\begin{theorem} \label{thLasNoetDiv} 
Un anneau \icl pleinement Lasker-Noether
est un \aKr à \dcncz.
\end{theorem}
%
\begin{proof}
Dans un \aKr $\gA$, on a vu que tout \elt \ird de~$\DivA$ est de la forme~$\dvA(\fp)$ pour un \idep détachable $\fp$
de hauteur $1$.  Dans \cite[A course in constructive algebra, 1988]{MRR}, il est démontré que pour un anneau pleinement Lasker-Noether, on peut calculer
explicitement les \ideps de hauteur $1$ qui contiennent un $a\in\Atl$ fixé. Ceci permet ensuite de calculer la \dcnc du \dvr
principal $\dv a$ en somme de \dvrs \irds (en utilisant le fait que $\DivA$ est discret). Les détails sont laissés \alecz.   
\end{proof}
%

\subsec{Stabilité pour les extensions \pollesz}

\begin{theorem} \label{th4AnnDivlKrull} \label{thKrAXKrull}
Soit $\gA$ un \aKr et $\gK$ son corps de fractions.
\begin{enumerate}
\item \label{i1terth4AnnDivl}   $\AX$ est aussi un \aKrz.
\item \label{i1bisth4AnnDivl} $\AX$ est   à \dcn compl\`ete \ssi   $\gA$  et $\KX$ sont   à \dcncz.
\end{enumerate}
\end{theorem}
%
\begin{proof} Ceci résulte de ce que $\AX$ est un \advl avec $\Div(\AX)\simeq\DivA\times \Div(\KX)$ (\thref{th2AnnDivl}).
\end{proof}
%

\subsec{Stabilité pour les extensions enti\`eres \iclez s}

\begin{thdef} \label{thNormeDiviseur} \emph{(Norme d'un diviseur)}\\
Soit $\gA$ un \advl non trivial de corps de fractions~$\gK$ \hbox{et $\gL\supseteq \gK$} un corps qui admet une base finie comme \Kevz. Soit $\gB$ la \cli de~$\gA$ dans $\gL$. L'anneau $\gB$ est aussi un \advlz.
On sait construire un 
\homo de groupes ordon\-nés~\hbox{$N\etl:\DivB\to\DivA$} qui satisfait les \prts suivantes.
\begin{enumerate}
\item $N\etl\big(\dvB(b)\big)=\dvA\big(\rN_{\gL/\gK}(b)\big)$ pour tout $b\in\gB.$
\item Pour $\beta\in\DivBp$, on a $\beta=0\iff\rN_{\gB/\gA}(\beta)=0$.
\item Pour tout $\alpha\in\DivA$,  $\rN\iBA(\alpha)=\dex{\gL:\gK}\,\alpha$.
\end{enumerate}
 On note $\rN\iBA$ cet \homoz, on l'appelle \emph{le morphisme norme}.
 \end{thdef}
\begin{proof}
On rappelle qu'une base de $\gL$ sur $\gK$ est aussi une base de $\gL[T]$ sur $\gK[T]$ ou de  $\gL(T)$ sur~$\gK(T)$. Ceci implique que la fonction norme
$\rN_{\gL/\gK}$ s'étend de mani\`ere naturelle en $\rN_{\gL[T]/\gK[T]}$
ou en~$\rN_{\gL (T) /\gK(T)}$. On continue à la noter~$\rN_{\gL/\gK}$.
On pose $r=\dex{\gL:\gK}$.
\\
Rappelons aussi que dans la situation où $\gK=\Frac(\gA)$, $\gA$ \iclz, et $\gB$ \cli de $\gA$ dans $\gL$, on a $\rN_{\gL/\gK}(\gB)\subseteq \gA$: en fait pour~\hbox{$x\in\gB$}, le \polcar de $x$  a tous ses \coes dans~$\gA$, et
l'\elt cotransposé~$\wi x$ (qui vérifie $x\wi x=\rN_{\gL/\gK}(x)$)
est un \elt de $\gB$ (par exemple en utilisant \aref{Corollary III-8.6}). 
\\
Rappelons enfin que $\gB[T]$ est la \cli de $\AT$ dans $\gL(T)$. Tout ceci implique que la norme d'un \elt  $g\in\gB[T]$ est un \elt
de $\AT$ et que $\wi g\in\BT$.
\\
On consid\`ere les ensembles $\Lst(\gB)\etl$
et $\Lst(\gA)$ et l'application $N:\Lst(\gB)\etl\to\Lst(\gA)$
définie comme suit 
\[ 
\begin{array}{ccccrcc} 
N(b_1,\dots,b_m)&=& (a_1,\dots,a_p),  &  \hbox{ où } &  K_{(\ub)}&=&\sum_{k=1}^{m}b_kT^{k-1}\\[.3em]
&&& \hbox{ et }&\rN_{\gL /\gK }(K_{(\ub)})&=&\sum_{\ell=1}^{p}a_\ell T^{\ell-1}  
\end{array}
\]
Enfin on définit $\nu:\Lst(\gB)\etl\to\DivAp$ par $\nu(\ub)=\dvA(N(\ub))$. à priori on a $p\leq (m-1)r+1$.\\
Notons que pour tout $\alpha\in\DivA$, on \hbox{a $\nu(\alpha)=r\,\alpha$}. Cela résulte de ce que \hbox{si $\alpha=\dvA(\ua)$}, alors $\rN_{\gL/\gK}(K_{(\ua)})=K_{(\ua)}^{r}=K_{(\uc)}$, donc $\dvA(\uc)=r\dvA(\ua)$ par le corolaire \ref{corprop2Idv}.\\
On démontre alors les points suivants.

\snii \textsl{1.} Les applications $N$ et $\nu$ sont bien définies, autrement dit les $a_i$ sont dans $\gA$ et l'un au moins est \ndzz. 
En effet  $K_{(\ub)}\in \BT$ donc
$\rN_{\gL /\gK }(K_{(\ub)})\in\AT$ d'apr\`es les remarques préliminaires.
Par ailleurs $K_{(\ub)}$ est \ndzz, donc la multiplication par $K_{(\ub)}$
est injective (de~$\gL[T]$ vers $\gL[T]$), ce qui implique que son
\deter $\rN_{\gL /\gK }(K_{(\ub)})$ est \ndzz.

\snii \textsl{2.}  Si $(\ub)=(b_{i})_{i\in\lrbn}$ et $(\uc)=(c_{j})_{j\in\lrbm}$
sont dans  $\Lst(\gB)\etl$, on a défini 

\snic{(\ub)\star(\uc)=\big(\sum_{i+j=\ell+1}b_ic_j\big)_{\ell\in\lrb{1..m+n-1}}}

\noindent On a $K_{(\ub)\star(\uc)}=K_{(\ub)}\,K_{(\uc)}$, donc $\rN_{\gL /\gK }(K_{(\ub)\star(\uc)})=\rN_{\gL /\gK }(K_{(\ub)})\;\rN_{\gL /\gK }(K_{(\uc)})$.\\
Et  en utilisant le corolaire \ref{corprop2Idv} ceci implique 

\snic{\nu\big((\ub)\star(\uc)\big)=\nu(\ub)+\nu(\uc).}

\snii \textsl{3.}  \fbox{Si $\dvB(\ub)=0$ alors $\nu(\ub)=0$}.
Soit $(\ub')=(b'_1,\dots,b'_m)$ telle que la liste $(\ua)=(\ub)\star(\ub')$ soit dans $\gA$ et admette un pgcd fort $g$ dans $\Atl$  (lemme \ref{lem2th1AnnDivl}). Donc $(\ua)=g\,(\ua')$ avec $\dvA(\ua')=0$.\\
Ceci donne $\dvB(\ub')=\dvB(\ub)+\dvB(\ub')=\dvB(g)$.
Donc $g$ est pgcd fort dans $\gB$ de la liste~$(\ub')$.
En particulier on peut écrire  $(\ub')=g\,(\uc)$
avec $(\ub)\star(\uc)= (\ua')$.
\\
Ceci implique~\hbox{$\nu(\ub)+\nu(\ub')= \nu(\ua')=r\,\dvA(\ua')=0$},
donc $\nu(\ub)=0$. 

\snii \textsl{4.}   \fbox{Si $\nu(\ub)=0$ alors $\dvB(\ub)=0$}. 
\\
En effet, considérons la liste $(\ub')\in\Lst(\gB)\etl$ définie par
$\wi{K_{(\ub)}}=K_{(\ub')}$, alors 

\snic{0=\nu(\ub)=\dvA\big(N(\ub)\big)=\dvA\big(\rc(K_{(\ub)}K_{(\ub')}))\big),}

\noindent et comme $\DivA$ s'identifie à un sous-\grl de $\DivB$ on obtient

\vspace{-.5em}
\[ 
\begin{array}{cclcccc} 
0&=&\dvA\big(\rc(K_{(\ub)}K_{(\ub')})\big)&=&\dvB\big(\rc(K_{(\ub)}K_{(\ub')})\big)  \\[.3em] 
  &=   &\dvB\big({(\ub)}\star{(\ub')}\big)&=& \dvB(\ub)+\dvB(\ub').   
\end{array}
\]

\snii  \textsl{5.}  En conséquence des points précédents l'\elt $\nu(\ub)\in\DivAp$ ne dépend que de $\dvB(\ub)$
et l'application correspondante

\snic{N\etl:\DivBp\to\DivAp}

\noindent est un morphisme 
de \mosz, qui s'étend de mani\`ere unique en un morphisme
(en \gnl non injectif) de groupes ordonnés $N\etl:\DivB\to\DivA$. 
\\
\textsl{Quelques précisions}.\\
On démontre d'abord que l'application $\nu:\Lst(\gB)\etl\to\DivAp$ \gui{passe au quotient}, \cad que si $\dvB(\ub^{(1)})=\dvB(\ub^{(2)})$, alors $\nu(\ub^{(1)})=\nu(\ub^{(2)})$. 
\\
Pour ceci considérons une liste $(\uc)$ dans 
$\Btl$ telle que $(\ub^{(1)})\star(\uc)$ admette un  pgcd fort~$e$,
autrement dit $(\ub^{(1)})\star(\uc)=e\,(\uz)\;(*)$ où la liste~$(\uz)$ admet~$1$
pour pgcd fort. Cela signifie $\dvB(\ub^{(1)})+\dvB(\uc)=\dvB(e)$.
L'\egtz~$(*)$ implique~\hbox{$\nu(\ub^{(1)})+\nu(\uc)=\nu(e)+\nu(\uz)$},
\hbox{i.e. $\nu(\ub^{(1)})+\nu(\uc)=\nu(e)$} car~$\dvB(\uz)=0$ implique $\nu(\uz)=0$. 
\\
Comme 
$\dvB(\ub^{(2)})+\dvB(\uc)=\dvB(e)$,   on on a aussi $\nu(\ub^{(2)})+\nu(\uc)=\nu(e)$ dans $\DivAp$. D'où $\nu(\ub^{(1)})=\nu(\ub^{(2)})$.\\
Une fois que l'on sait que $N\etl$ définit une opération de $\DivBp$ vers $\DivAp$, on voit que c'est un morphisme pour l'addition. Comme l'image réciproque de $0$ est $0$ et puisque $\DivBp$ et $\DivAp$ sont les parties positives de $\DivA$ et $\DivB$, il en résulte que $N\etl$ s'étend de mani\`ere unique en un morphisme de groupes ordonnés.
 \end{proof}

\begin{theorem} \label{th5AnnDivl}
Soit $\gA$ un \aKr de corps de fractions~$\gK$ \hbox{et $\gL\supseteq \gK$} un corps qui admet une base finie sur $\gK$. Soit $\gB$ la \cli de $\gA$ dans $\gL$. 
Alors $\gB$ est un \aKrz.
\end{theorem}
%
\begin{proof}
 On sait déjà que $\gB$ est un \advl (\thref{th1AnnDivl})
et que $\DivB$ est discret (\thref{th3AnnDivl}). Il reste à montrer que $\DivB$ est à \dcnbz.
\\
On consid\`ere une \dcn $\dvB(\ub)=\dvB(\ub^{(1)})+\cdots+\dvB(\ub^{(\ell)})$ \hbox{dans $\DivBp$}. Au moyen du morphisme norme $\rN_{\gB/\gA}$ elle est transformée en
une \dcn de $\rN_{\gB/\gA}(\ub)$ dans~$\DivAp$. Puisque  $\DivA$ est à \dcnbz, si
$\ell$ est suffisamment grand, un des termes $\rN_{\gB/\gA}(\ub^{(i)})$ de cette \dcn est nul. Et ceci implique $\dvB(\ub^{(i)})=0$.
\end{proof}
\rem Comme cas particulier, si $\gA$ est un \dok à \fabz, il en va de même
pour $\gB$ (il n'est donc pas \ncr de supposer l'extension $\gL/\gK$ séparable).
\eoe

\section*{Conclusion}
\addcontentsline{toc}{section}{Conclusion}

Nous sommes assez convaincus par l'introduction de l'article \cite[Aubert]{Aubertadvl}, où l'auteur déplore que seul Jaffard ait compris Lorenzen, alors que Bourbaki, Gilmer et Larsen-McCarthy par exemple se sont enferrés dans les \ids divisoriels qui manquent absolument de finitude. 
Ici nous avons enfoncé encore un peu plus
le clou, en ne faisant jamais référence à l'\gui{ensemble} de tous les idéaux fractionnaires ni à la théorie des $\star$-opérations sur cet \gui{ensemble}.

Notons que nos \advls sont exactement les \textsl{anneaux avec une théorie des pgcds de type fini} de \cite[Lucius]{Luciusadvl}. Lucius attribue la véritable paternité à Aubert tout en rectifiant une erreur.
L'article de Lucius est surout consacré aux \aKrsz:
en rajoutant une condition de type \noez, il obtient ce qu'il appelle les \textsl{anneaux avec une théorie des diviseurs}, qui sont les \aKrsz. 
Il attribue à \cite[Skula]{Skula} le fait d'avoir élucidé le rapport
entre l'approche qu'il propose (que nous avons grosso modo suivie, mais dans un cadre simplifié et \cofz) 
et la présentation
du \pb dans \cite{Bosha}.

Signalons aussi les articles   \cite[Arnold]{Arnold29} et \cite[Clifford]{Clifford38} dans lesquels une théorie purement multiplicative, à savoir l'étude des \gui{\mos avec une théorie des diviseurs} est donnée pour le cas \gui{Krull}, \cad lorsqu'on demande la \dcn unique en facteurs premiers. Plus \prmtz, pour un \mo $(S,\cdot,1)$ \gui{réduit} (ce qui veut dire que $1$ est le seul \elt \ivz) et \gui{régulier} (ce qui veut dire que tout \elt est simplifiable), on examine dans quelles conditions il est contenu dans un \mo $\Sigma$ jouissant des trois \prts suivantes:
\begin{itemize}
\item $\Sigma$ est isomorphe  à un \mo  $(\NN^{(I)},+,0)$ (i.e. un \mo réduit régulier avec \dcn unique en facteurs premiers),
\item pour $a$, $b$ in $S$, on a  $a|b$ dans $S$ \ssi $a|b$ dans $\Sigma$, 
\item tout \elt de $\Sigma$ est borne inférieure d'une famille finie dans $S$. 
\end{itemize}
Autrement dit, c'est l'approche habituellement attribuée à \cite{Bosha} mais dans un  cadre épuré (pas de condition non multiplicative) et plus \gnl (pas d'anneau int\`egre, seulement un \mo multiplicatif). La th\`ese de Clifford, reportée dans  \cite{Clifford38} consiste à \gnr les résultats d'Arnold dans le cadre de \mos plus \gnlsz.
Clifford cite \cite[van der Waerden]{vdW29} comme ayant découvert de mani\`ere indépendante essentiellement les mêmes résultats qu'Arnold, au moins pour le cadre des anneaux \noes int\`egres.

Quant aux travaux d'Aubert et Lucius, que nous avons repris ici dans un cadre \cofz, ils consistent à laisser tomber la condition de \dcn unique en facteurs premiers et à demander seulement que $\Sigma$ soit partie positive d'un \grlz.

\bigskip\noindent  {\bf\large Remerciements.} Le travail du premier auteur a été financé par le projet ERC (FP7/2007-2013) / ERC grant agreement nr. 247219.
Nous remercions Marco Fontana et Mohammed Zafrullah pour toutes les informations et conseils utiles concernant l'approche des PvMD dans la littérature classique, ainsi que leurs réponses à des questions délicates.
Nous remercions aussi tout particuli\`erement Claude Quitté pour sa collaboration efficace, sans laquelle cet article n'aurait pas vu le jour.

\rdb


\newpage
\tableofcontents

\end{document}